\documentclass[reqno, 11pt, letterpaper]{amsart}

\usepackage{amsmath}
\usepackage{amsfonts}
\usepackage{amssymb}
\usepackage{graphicx}
\usepackage{amsthm,color,yfonts,cite} \usepackage{paralist}
\usepackage{mathabx}
\usepackage{hyperref}
\usepackage{physics}
\usepackage{dsfont}
\usepackage{bbm}
\usepackage{mathrsfs}

\newtheorem{theorem}{Theorem}
\newtheorem{definition}[theorem]{Definition}
\newtheorem{proposition}[theorem]{Proposition}
\newtheorem{lemma}[theorem]{Lemma}

\newtheorem{assumption}[theorem]{Assumption}

\theoremstyle{remark}
\newtheorem{remark}[theorem]{Remark}

\makeatletter
\newcommand*{\rom}[1]{\expandafter\@slowromancap\romannumeral #1@}
\makeatother

\newcommand{\ls}{\lesssim}

\newcommand{\la}{\langle}
\newcommand{\ra}{\rangle}
\newcommand{\R}{\mathbb{R}}
\newcommand{\T}{\mathbb{T}}
\newcommand{\C}{\mathbb{C}}
\newcommand{\Z}{\mathbb{Z}}
\newcommand{\pa}{\partial}

\newcommand{\al}{\alpha}

\usepackage{wrapfig}
\usepackage{tikz}
\usetikzlibrary{arrows,calc,decorations.pathreplacing}
\definecolor{light-gray1}{gray}{0.90}
\definecolor{light-gray2}{gray}{0.80}
\definecolor{light-gray3}{gray}{0.60}

\numberwithin{equation}{section}

\numberwithin{theorem}{section}

\numberwithin{table}{section}

\numberwithin{figure}{section}

\ifx\pdfoutput\undefined
  \DeclareGraphicsExtensions{.pstex, .eps}
\else
  \ifx\pdfoutput\relax
    \DeclareGraphicsExtensions{.pstex, .eps}
  \else
    \ifnum\pdfoutput>0
      \DeclareGraphicsExtensions{.pdf}
    \else
      \DeclareGraphicsExtensions{.pstex, .eps}
    \fi
  \fi
\fi

\title[Long-time KdV approximation for the FPUT system]{Long-time Korteweg-de Vries approximation for the Fermi-Pasta-Ulam-Tsingou system}

\date{\today}
\author[Y. Hong]{Younghun Hong}
\address{Department of Mathematics, Chung-Ang University, Seoul 06974, South Korea}
\email{yhhong@cau.ac.kr}

\author[J. Jang]{Junyeong Jang}
\address{Department of Mathematics, Chung-Ang University, Seoul 06974, South Korea}
\email{jyjang0119@cau.ac.kr}

\begin{document}

\begin{abstract}
We establish a quantitative $L^2$-approximation of the infinite Fermi--Pasta--Ulam--Tsingou (FPUT) system by two counter-propagating Korteweg--de Vries (KdV) waves at low Sobolev regularity as the lattice spacing $h$ tends to zero. For initial moving-frame profiles uniformly bounded in $H^s$, $0<s\leq 1$, and KdV initial data given by their continuations to the real line, we obtain an approximation error bounded by $\sim h^{\frac{2s}{5}}e^{K_0|t|}$. Consequently, the approximation error tends to zero uniformly on logarithmically growing time intervals. Our proof combines conservation of the rescaled FPUT Hamiltonian with an $h$-uniform local well-posedness theory in $L^2$ and persistence of Sobolev regularity, yielding $h$-uniform Sobolev bounds with at most exponential growth in time. We also introduce a frequency-localized auxiliary equation that enables us to compare the Fourier restriction norms associated with the FPUT and KdV flows.
\end{abstract}

\maketitle

\section{Introduction}
In the early 1950s, Fermi, Pasta, Ulam, and Tsingou carried out pioneering numerical experiments on the MANIAC I computer at Los Alamos to study a one-dimensional chain of particles coupled by weakly nonlinear
nearest-neighbor interactions, now known as the FPUT
system~\cite{FermiPastaUlam1955,Dauxois2008}. The experiments were designed to test the expectation that nonlinear coupling would drive the system toward thermal equilibrium, with energy equipartitioned among its linear normal modes. Contrary to this expectation, energy initially concentrated in one or a few low-frequency modes spread to other modes but later returned remarkably close to its initial distribution, a phenomenon now known
as FPUT recurrence. Zabusky and Kruskal subsequently connected this recurrence to the KdV limit and soliton dynamics~\cite{ZK1965}. In a related development, Toda introduced a lattice with exponential
interactions, now known as the Toda lattice, which was later shown to be completely integrable~\cite{Toda1967Vibration,Henon1974TodaIntegrals,Flaschka1974TodaIntegrals}. The FPUT system has since become a paradigmatic model for studying dispersive waves, solitons, near-integrability, metastability,
thermalization, and small-amplitude and continuum
limits~\cite{BermanIzrailev2005,Gallavotti2008}.

In this article, we consider the FPUT system on the infinite lattice $\mathbb{Z}$, consisting of an infinite chain of particles coupled by nonlinear springs. Let \(q(z), p(z):\mathbb{Z}\to\mathbb{R}\) denote the displacement and momentum, respectively, of the $z$-th particle in the chain. The formal Hamiltonian of the system is given by
\begin{equation}\label{eq:FPUT-Hamiltonian}
\mathcal{H}(q,p)
=
\sum_{z\in\mathbb{Z}}
\left\{
\frac{p(z)^2}{2}
+
V\bigl(q(z+1)-q(z)\bigr)
\right\},
\end{equation}
where $V:\mathbb{R}\to\mathbb{R}$ is the nearest-neighbor interaction potential. The corresponding Hamiltonian equations are explicitly given by
\begin{equation}\label{eq:FPUT-system}
\left\{
\begin{aligned}
\partial_t q(t,z)
&=
p(t,z),\\
\partial_t p(t,z)
&=
V'\bigl(q(t,z+1)-q(t,z)\bigr)
-
V'\bigl(q(t,z)-q(t,z-1)\bigr),
\end{aligned}
\right.
\end{equation}
where $(q(t,z), p(t,z)):\R\times\Z\to \R\times\R$. Throughout this article, we impose the following normalization for the potential:
\begin{equation}\label{eq:potential-assumption}
V\in C^4(\mathbb{R}),
\qquad
V(0)=V'(0)=0,
\qquad
V''(0)=V'''(0)=1.
\end{equation}

\begin{remark}
\noindent\textup{(i)}
Any potential $V\in C^4(\mathbb{R})$ satisfying
$V(0)=V'(0)=0$, $V''(0)=a>0$, and $V'''(0)=b\neq0$
(see Friesecke and Pego~\cite{FP1999,FP2002,FP2004-1,FP2004-2})
can be reduced to the normalized setting above by a simple rescaling.
More precisely, setting
$\tilde{t}=\sqrt{a} t$,
$\tilde{q}(\tilde{t},z)=\frac{b}{a}q(t,z)$, and
$\tilde{p}(\tilde{t},z)=\frac{b}{a^{3/2}}p(t,z)$,
and defining
$\tilde{V}(r)=\frac{b^2}{a^3}V\bigl(\frac{a}{b}r\bigr)$,
we find that $(\tilde{q},\tilde{p})$ satisfies the corresponding
Hamiltonian system with interaction potential $\tilde{V}$.
Moreover, $\tilde{V}(0)=\tilde{V}'(0)=0$ and
$\tilde{V}''(0)=\tilde{V}'''(0)=1$.

\par\smallskip
\noindent\textup{(ii)}
After the normalization in \textup{(i)}, the assumptions
in \eqref{eq:potential-assumption} cover the $\alpha$--FPU system
with interaction potential
$V(r)=\frac12r^2+\frac{\alpha}{3}r^3$, $\alpha\neq0$,
for which Zabusky and Kruskal formally derived the KdV equation
as a long-wave approximation~\cite{ZK1965}.
They also cover the potential $V(r)=e^r-1-r$, which gives rise
to the completely integrable Toda lattice.

\par\smallskip
\noindent\textup{(iii)}
The two nondegeneracy conditions imposed on $V$ play distinct roles
in the KdV approximation.
The condition $V''(0)=a>0$ ensures that the linearized lattice
supports stable acoustic waves with sound speed $\sqrt{a}$;
in a frame moving at this speed, the leading long-wave dispersive
correction gives rise to the Airy part of the KdV equation.
The condition $V'''(0)=b\neq0$ produces the leading quadratic term
in the force law and, consequently, the quadratic nonlinearity
in the KdV equation.

\par\smallskip
\noindent\textup{(iv)}
If, instead, $V'''(0)=0$ but $V^{(4)}(0)\neq0$,
as in the $\beta$--FPU system, the leading nonlinear correction
is cubic, and an appropriate small-amplitude, long-wave scaling
leads to the modified KdV equation; see~\cite{Abl2011,KY2025}.
\end{remark}

The goal of this article is to rigorously justify the KdV approximation
for low-regularity solutions of the FPUT system over long time intervals.
For solitary waves, Friesecke and Pego~\cite{FP1999} established
the existence of small-amplitude waves whose suitably rescaled profiles
converge to the KdV soliton.
See also~\cite{FW1994,FP2002,FP2004-1,FP2004-2,Miz2013}
for related existence and stability results.
For more general dynamics, Schneider and Wayne~\cite{SW2000}
showed that sufficiently regular FPUT solutions can be approximated
by two counter-propagating waves governed by the KdV equation.
Further developments include the normal-form approach of
Bambusi and Ponno~\cite{BP2006} and the extension to longer time
intervals by Khan and Pelinovsky~\cite{KP2017}.
More recently, dispersive PDE methods have enabled the study of
the KdV approximation at lower regularity and over longer time
intervals~\cite{HKY2021,KY2025,LK2026,KL2026}.
More broadly, the rigorous derivation of continuum equations from
lattice models remains an active area of research; see, for example,
\cite{SSB1986,KLS2013,HY2019-1,HKNY2021,HKY2023,KOVW2023,KOVW2025}.

\subsection{Rescaled FPUT system}

In the literature, the FPUT-to-KdV approximation is usually formulated in the
small-amplitude regime, which is equivalent to the continuum-limit formulation under a simple rescaling. We adopt the latter formulation because it provides a natural framework for comparing the regularity requirements for the approximation with those in the low-regularity well-posedness theory for the limiting KdV equation.

To begin, we introduce the relative displacement, or strain variable,
\begin{equation}
r(t,z):=q(t,z+1)-q(t,z).
\end{equation}
The FPUT system \eqref{eq:FPUT-system} can then be rewritten
as the second-order differential-difference equation
\begin{equation}\label{eq:FPUT-system-original-scale}
\partial_t^2 r(t,z)
=
V'\bigl(r(t,z+1)\bigr)
+
V'\bigl(r(t,z-1)\bigr)
-
2V'\bigl(r(t,z)\bigr).
\end{equation}
Let $h>0$ denote the rescaled lattice spacing.
To study the continuum limit $h\to0$, we introduce the lattice
\[
h\mathbb{Z}:=\{hm:m\in\mathbb{Z}\}
\]
and define the discrete Laplacian $\Delta_h$ by
\[
\Delta_h f(z)
:=
\frac{f(z+h)+f(z-h)-2f(z)}{h^2},
\qquad z\in h\mathbb{Z}.
\]
We then introduce the long-wave, small-amplitude rescaling
$r_h:\mathbb{R}\times h\mathbb{Z}\to\mathbb{R}$ by
\begin{equation}\label{eq:FPUT-rescaling}
r_h(t,z)
:=
\frac{1}{h^2}
r\left(\frac{t}{h^3},\frac{z}{h}\right).
\end{equation}
Under this rescaling, \eqref{eq:FPUT-system-original-scale} becomes
\begin{equation}\label{eq:FPUT-system-rescaled}
h^6\partial_t^2 r_h
=
\Delta_h\bigl(V'(h^2r_h)\bigr).
\end{equation}
By the assumptions in \eqref{eq:potential-assumption}
and Taylor's theorem, we write
\begin{equation}
V'(h^2r_h)
=
h^2r_h
+
\frac{h^4}{2}r_h^2
+
\frac{h^4}{2}\mathcal{R}_h[r_h],
\end{equation}
where the higher-order remainder is defined by
\begin{equation}\label{eq:higher-order-remainder}
\mathcal{R}_h[y]
:=
\frac{2}{h^4}
\left\{
V'(h^2y)-h^2y-\frac{h^4}{2}y^2
\right\},
\qquad y\in\mathbb{R}.
\end{equation}
Substituting this expansion into \eqref{eq:FPUT-system-rescaled}
and dividing by $h^6$, we obtain the initial-value problem
\begin{equation}\label{eq:FPUT-system-expanded}
\left\{
\begin{aligned}
\partial_t^2 r_h
-
\frac{1}{h^4}\Delta_h r_h
&=
\frac{1}{2h^2}\Delta_h\bigl(r_h^2\bigr)
+
\frac{1}{2h^2}\Delta_h\bigl(\mathcal{R}_h[r_h]\bigr),\\
r_h(0)&=r_{h,0},\\
\partial_t r_h(0)&=r_{h,1}.
\end{aligned}
\right.
\end{equation}
For finite-energy solutions, the corresponding rescaled Hamiltonian
is given by
\begin{equation}\label{eq: rescaled FPUT-Hamiltonian}
\mathcal{H}_h(t)
=
h\sum_{z\in h\mathbb{Z}}
\left\{
\frac12
\left|h^2\bigl(|\nabla_h|^{-1}\partial_t r_h\bigr)(t,z)\right|^2
+
\frac{1}{h^4}V\bigl(h^2r_h(t,z)\bigr)
\right\},
\end{equation}
where $|\nabla_h|:=(-\Delta_h)^{1/2}$.

\subsection{Basic notation}

We introduce the notation needed to formulate the KdV approximation.

\subsubsection{Function spaces and differential operators}

To reflect the Riemann-sum scaling underlying the continuum limit,
we define
\[
L_z^p(h\mathbb{Z})
:=
\big\{
f_h:h\mathbb{Z}\to\mathbb{C}:
\|f_h\|_{L_z^p(h\mathbb{Z})}<\infty
\big\},
\]
equipped with the rescaled lattice 
$L^p$-norms
\[
\|f_h\|_{L_z^p(h\mathbb{Z})}
:=
\begin{cases}
\displaystyle
\left(
h\sum_{z\in h\mathbb{Z}}|f_h(z)|^p
\right)^{1/p},
& 1\leq p<\infty,\\[1.2ex]
\displaystyle
\sup_{z\in h\mathbb{Z}}|f_h(z)|,
& p=\infty.
\end{cases}
\]
For $f_h\in L_z^2(h\mathbb{Z})$, we define its lattice Fourier transform by
\[
(\mathcal{F}_h f_h)(\xi)
=\widehat{f_h}(\xi)
:=
h\sum_{z\in h\mathbb{Z}}f_h(z)e^{-iz\xi},
\qquad \xi\in\mathbb{T}_h,
\]
where the series is understood in $L^2(\mathbb{T}_h)$ and
$\mathbb{T}_h$ is the one-dimensional torus of period $\frac{2\pi}{h}$:
\[
\mathbb{T}_h
:=
\mathbb{R}\big/\left(\frac{2\pi}{h}\mathbb{Z}\right)
\simeq
\left[-\frac{\pi}{h},\frac{\pi}{h}\right).
\]
For $g_h\in L^2(\mathbb{T}_h)$, the inverse lattice Fourier transform
is given by
\[
(\mathcal{F}_h^{-1}g_h)(z)
=\widecheck{g_h}(z)
:=
\frac{1}{2\pi}
\int_{-\frac{\pi}{h}}^{\frac{\pi}{h}}
g_h(\xi)e^{iz\xi} d\xi,
\qquad z\in h\mathbb{Z}.
\]
These definitions correspond to our conventions for the continuum
Fourier transform and its inverse:
\[
\hat{f}(\xi)
=
\int_{\mathbb{R}}f(x)e^{-ix\xi} dx,
\qquad
\check{g}(x)
=
\frac{1}{2\pi}
\int_{\mathbb{R}}g(\xi)e^{ix\xi} d\xi.
\]

\begin{remark}
We use $\xi$ for the frequency variable
in both $\mathbb{R}$ and $\mathbb{T}_h$.
When the distinction is needed, we write $[\xi]$ for the unique
representative of $\xi\in\mathbb{R}$ modulo $\frac{2\pi}{h}$
lying in $[-\frac{\pi}{h},\frac{\pi}{h})$.
\end{remark}

For functions on $h\mathbb{Z}$, we introduce two lattice analogues
of the spatial derivative.

\begin{definition}[Lattice derivative operators $\nabla_h$ and $\partial_h$]
We define $\nabla_h$ as the discrete Fourier multiplier
\[
\widehat{\nabla_h f_h}(\xi)
:=
m_h(\xi)\widehat{f_h}(\xi),
\qquad
m_h(\xi):=\frac{2i}{h}\sin\left(\frac{h[\xi]}{2}\right).
\]
We define $\partial_h$ as the discrete Fourier multiplier
\[
\widehat{\partial_h f_h}(\xi)
:=
i[\xi]\widehat{f_h}(\xi).
\]
\end{definition}

\begin{remark}
Both $\nabla_h$ and $\partial_h$ are lattice analogues of
the derivative $\partial_x$ on $\mathbb{R}$:
for each fixed $\xi\in\mathbb{R}$, their symbols converge to $i\xi$ as $h\to0$. Note that $\Delta_h=\nabla_h^2$.
\end{remark}

For $s\in\mathbb{R}$, we define the lattice Sobolev space
$H_z^s(h\mathbb{Z})$ by the norm
\[
\|f_h\|_{H_z^s(h\mathbb{Z})}
:=
\big\|
(1-\partial_h^2)^{\frac{s}{2}}f_h
\big\|_{L_z^2(h\mathbb{Z})}
=
\left(
\frac{1}{2\pi}
\int_{-\frac{\pi}{h}}^{\frac{\pi}{h}}
(1+|\xi|^2)^s
|\widehat{f_h}(\xi)|^2 d\xi
\right)^{1/2}.
\]

\begin{definition}[Almost-translation operator on $h\mathbb{Z}$]
For $a\in\mathbb{R}$, we define $e^{a\partial_h}$ as the discrete
Fourier multiplier
\[
\widehat{e^{a\partial_h}f_h}(\xi)
:=
e^{ia[\xi]}\widehat{f_h}(\xi).
\]
\end{definition}

\begin{remark}\label{rmk: almost translation on hZ}
On $\mathbb{R}$, the operator $e^{a\partial_x}$ acts as translation
by $a\in\mathbb{R}$:
$(e^{a\partial_x}f)(x)=f(x+a)$.
On $h\mathbb{Z}$, if $a\in h\mathbb{Z}$, then $e^{a\partial_h}$
agrees with the corresponding lattice translation:
$(e^{a\partial_h}f_h)(z)=f_h(z+a)$.
If $a\notin h\mathbb{Z}$, however, $e^{a\partial_h}$ does not
correspond to a pointwise translation on the lattice and,
in general, does not preserve products.
\end{remark}

\subsubsection{Continuation and discretization}

To compare functions on $h\mathbb{Z}$ with functions on $\mathbb{R}$,
we use the continuation and discretization operators introduced
by Nakamura and Tadano~\cite{NT2021}.
These operators transfer functions between the lattice and continuum
settings.

Throughout this paper, we fix $\varphi$ and $\alpha$ as follows.
\begin{assumption}
\label{assumption for varphi}
\begin{enumerate}[$(1)$]
\item
We assume that $\varphi\in\mathcal{S}(\R;\R)$ is even,
$\operatorname{supp}\hat{\varphi}\subset(-2\pi,2\pi)$, and
$\hat{\varphi}(0)=1$, where $\hat{\varphi}$ denotes the Fourier
transform of $\varphi$ on $\mathbb{R}$.

\item
For every $\xi\in\mathbb{R}$,
\begin{equation}\label{phi assumption 1}
\sum_{m\in\Z}|\hat{\varphi}(\xi+2\pi m)|^2=1.
\end{equation}
\end{enumerate}
The support condition in \textup{(1)} and \eqref{phi assumption 1}
imply that there exists $\alpha\in(0,\pi)$ such that
\begin{equation}\label{phi assumption 2}
\hat{\varphi}(\xi)=1,
\qquad |\xi|\leq\alpha.
\end{equation}
\end{assumption}

\begin{definition}[Discretization and continuation]
\label{def: discretization and continuation}
We define the discretization operator
$\mathfrak{D}_h:L_x^2(\R)\to L_z^2(h\Z)$ by
\begin{equation}\label{eq: definition of Dh}
(\mathfrak{D}_h f)(z)
:=
\frac{1}{h}\int_\R
\overline{\varphi\left(\frac{x-z}{h}\right)}
f(x)\,dx,
\qquad z\in h\Z.
\end{equation}
The continuation operator
$\mathfrak{C}_h:L_z^2(h\Z)\to L_x^2(\R)$
is its adjoint, $\mathfrak{C}_h=\mathfrak{D}_h^*$, and is given by
\begin{equation}\label{eq: definition of Ch}
(\mathfrak{C}_h f_h)(x)
:=
\sum_{z\in h\Z}
\varphi\left(\frac{x-z}{h}\right)f_h(z),
\qquad x\in\R,
\end{equation}
where the series converges in $L_x^2(\R)$.
\end{definition}

\begin{remark}
The operator $\mathfrak{C}_h$ is an $L^2$-isometry, and its adjoint
$\mathfrak{D}_h=\mathfrak{C}_h^*$ is a partial isometry satisfying
$\mathfrak{D}_h\mathfrak{C}_h=\textup{Id}_{L_z^2(h\Z)}$;
see Lemma~\ref{lem: isometry property of Ch and Dh}.
\end{remark}

\subsection{Main result: KdV limit of the FPUT system}
We decompose $r_h$ into right- and left-moving components by writing
\begin{equation}\label{eq:FPUT-profile-decomposition}
r_h(t,z)=
e^{-\frac{t}{h^2}\partial_h}u_h^+(t,z)
+
e^{\frac{t}{h^2}\partial_h}u_h^-(t,z).
\end{equation}
Then the profiles $u_h^\pm$ satisfy the coupled FPUT system (see Section~\ref{subsec: Derivation of the coupled FPUT system})
\begin{equation}\label{eq: coupled FPUT system} 
\pa_tu_h^\pm
=
\mp\frac{\nabla_h-\pa_h}{h^2}u_h^\pm
\mp\frac{1}{4}\nabla_h
e^{\pm\frac{t}{h^2}\pa_h}
\big(r_h^2+\mathcal{R}_h[r_h]\big),
\end{equation}
where $\mathcal{R}_h$ is the higher-order remainder defined in \eqref{eq:higher-order-remainder}. In the formal continuum limit $h\to0$ (see Remark~\ref{rmk: Formal KdV limit in integral form}), the profiles $u_h^\pm$ converge to solutions $w^\pm$ of the corresponding KdV equations
\begin{equation}\label{eq:KdV}
\partial_t w^\pm
\pm
\frac{1}{24}\partial_x^3w^\pm
\pm
\frac{1}{2}w^\pm\partial_xw^\pm
=
0.
\end{equation}
Consequently, the leading-order behavior of the solution to the rescaled FPUT system \eqref{eq:FPUT-system-expanded} is described by two decoupled KdV waves propagating in opposite directions:
\begin{equation}\label{eq:FPUT-KdV-approximation}
(\mathfrak{C}_hr_h)(t,x)
\underset{h\to0}\approx
w^+\big(t,x-\tfrac{t}{h^2}\big)
+
w^-\big(t,x+\tfrac{t}{h^2}\big).
\end{equation}

Our main theorem establishes an $L^2$-convergence bound for the FPUT--KdV approximation in the low-regularity range $0<s\leq1$, with at most exponential growth in time. 

\begin{theorem}[Exponential-in-time KdV approximation for the FPUT system]\label{thm: main theorem}
Let $0<s\leq 1$, and let $\mathfrak{C}_h:L_z^2(h\Z)\to L_x^2(\R)$ be the continuation operator defined in \eqref{eq: definition of Ch}. Suppose that
$$
\sup_{h\in (0,1]}
\|u_{h,0}^\pm\|_{H_z^s(h\Z)}<\infty
\quad\textup{and}\quad
\|w_0^\pm\|_{H_x^s(\R)}<\infty.
$$
Let $w^\pm(t)\in C_t(\R;H_x^s(\R))$ be the unique solution to the KdV equation \eqref{eq:KdV} with initial data $w_0^\pm$. Then there exist $h_0\in(0,1]$ and constants $K_0, K_s>0$ such that, for every $h\in(0,h_0]$, the coupled FPUT system \eqref{eq: coupled FPUT system} with initial data $u_{h,0}^\pm$ admits a unique global solution $u_h^\pm(t)\in C_t(\R;H_z^s(h\Z))$ and, for every $t\in\R$,
\begin{equation}\label{ineq: convergence bound}
\|\mathfrak{C}_hu_h^\pm(t)-w^\pm(t)\|_{L_x^2(\R)}
\leq
e^{K_0|t|}
\Big(
K_0\|\mathfrak{C}_hu_{h,0}^\pm-w_0^\pm\|_{L_x^2(\R)}
+
K_sh^{\frac{2s}{5}}\Big).
\end{equation}
\end{theorem}

\begin{remark}[Dependence of the constants]
The constants $K_0$ and $K_s$, as well as the threshold $h_0$, are independent of $h\in(0,1]$ and $t\in\R$. With the potential $V$ and the continuation operator $\mathfrak{C}_h$ fixed throughout the paper, the threshold $h_0$ and the exponential rate $K_0$ depend only on $s$ and the lower-order bounds
$$
\sup_{h\in(0,1]}\sum_{\pm}\|u_{h,0}^\pm\|_{L_z^2(h\Z)}
\quad\textup{and}\quad
\sup_\pm\|w_0^\pm\|_{L_x^2(\R)},
$$
whereas $K_s$ may additionally depend on the corresponding $H^s$-bounds in the hypothesis. The same choices of $h_0$, $K_0$ and $K_s$ can be made uniformly when $w_0^\pm$ is replaced by an $h$-dependent family, provided that the corresponding bounds are uniform in $h$.
\end{remark}

\begin{remark}[Logarithmic time validity]
If \(\|\mathfrak{C}_hu_{h,0}^\pm-w_0^\pm\|_{L_x^2(\R)}\leq Ch^{\frac{2s}{5}}\), then 
$$
\lim_{h\to 0}\sup_{|t|\leq c\log(1/h)}\sum_{\pm}\|\mathfrak{C}_hu_h^\pm(t)-w^\pm(t)\|_{L_x^2(\R)}=0
$$
for every $0<c<\frac{2s}{5K_0}$. Thus, in the rescaled time variable, the temporal interval of validity grows at least logarithmically as $h\to 0$.
\end{remark}

In view of the profile decomposition \eqref{eq:FPUT-profile-decomposition}, Theorem~\ref{thm: main theorem} yields the following continuum limit for the strain variable $r_h$. The corresponding small-amplitude limit for the original FPUT system follows from the scaling relation \eqref{eq:FPUT-rescaling}.
\begin{theorem}[Continuum and small-amplitude limits for the FPUT system]\label{thm: continuum limit}
Let $0<s\leq 1$, and suppose that 
$$
\sup_{h\in(0,1]}\Big(\|r_{h,0}\|_{H_z^s(h\Z)}+\big\|h^2\nabla_h^{-1}r_{h,1}\big\|_{H_z^s(h\Z)}\Big)<\infty.
$$
For each $h\in(0,1]$, let $w^{\pm,h}(t)\in C_t(\R;H_x^s(\R))$ be the solution to the KdV equation with initial data \(w^{\pm,h}(0)=\mathfrak{C}_hr_{h,0}^\pm\) with 
$$
r_{h,0}^\pm:=\frac{1}{2}(r_{h,0}\mp h^2\nabla_h^{-1}r_{h,1}).
$$
Then, there exist $h_0\in(0,1]$ and constants $K_0, K_s>0$ such that for every $h\in(0,h_0]$, the rescaled FPUT system \eqref{eq:FPUT-system-expanded} with initial data $(r_{h,0}, r_{h,1})$ admits a unique global solution $r_h(t)\in C_t(\R;H_z^s(h\Z))$, and the following statements hold for every $T>0$.
\begin{enumerate}[(1)]
\item (Continuum limit)
\begin{equation}\label{ineq: continuum limit}
\sup_{|t|\leq T}
\Big\|(\mathfrak{C}_hr_h)(t,\cdot )-w^{+,h}\big(t,  \cdot -\tfrac{t}{h^2}\big)-w^{-,h}\big(t,  \cdot +\tfrac{t}{h^2}\big)\Big\|_{L_x^2(\R)}
\leq K_se^{K_0T}h^{\frac{2s}{5}}.
\end{equation}
\item (Small-amplitude limit) Scaling back, \(r(t,z):=h^2r_h(h^3t, hz):\mathbb{R}\times\mathbb{Z}\to\mathbb{R}\) solves the original FPUT system \eqref{eq:FPUT-system-original-scale}, and 
\begin{equation}\label{ineq: small-amplitude limit}
\sup_{|t|\leq \frac{T}{h^3}}
\big\|(\mathfrak{C}_1r)(t,\cdot )-h^2w^{+,h}(h^3t, h(\cdot -t))-h^2w^{-,h}(h^3t, h(\cdot +t))\big\|_{L_x^2(\R)}
\leq K_se^{K_0T}h^{\frac{3}{2}+\frac{2s}{5}}.
\end{equation}
\end{enumerate}
\end{theorem}

\begin{remark}[Comparison with previous results]
The KdV approximation was first rigorously established by Schneider and Wayne~\cite{SW2000} for sufficiently regular solutions on fixed intervals of rescaled time, and extended to logarithmically longer intervals by Khan and Pelinovsky~\cite{KP2017}. The first author and collaborators~\cite{HKY2021} subsequently established the short-time approximation for $s>\frac{3}{4}$ using Fourier-analytic methods. More recently, Liu and Koch~\cite{LK2026} established long-time approximation for $H^1$ solutions of the Toda lattice using its complete integrability. For the general FPU system, Koch and Liu~\cite[Theorem~1.2]{KL2026} established a global-in-time approximation in $H^{-1}$ for uniformly $L^2$-bounded initial data, resolving the long-time approximation question raised in~\cite{HKY2021}.

Our result was obtained independently\footnote{The preprint of Koch and Liu~\cite{KL2026} appeared on arXiv shortly before ours.} and provides a complementary quantitative approximation. For initial data uniformly bounded in $H^s$, $0<s\leq 1$, we obtain an $L^2$ error of order $e^{K_0|t|}h^{\frac{2s}{5}}$. Both results yield logarithmically growing intervals of validity in their respective error norms. In this setting, direct Sobolev interpolation of the global $H^{-1}$ estimate in~\cite[Theorem~1.2(ii)]{KL2026} with additional $h$-uniform $H^s$ bounds for both flows, growing at most exponentially in time, gives an $L^2$ error of order $e^{K_s|t|}h^{\frac{2s}{5(1+s)}}$.

Both arguments use conservation of the rescaled FPUT Hamiltonian to control the $L^2$ norm. We combine this control with $h$-uniform persistence of regularity to establish the additional Sobolev bounds. A frequency-localized auxiliary equation then allows us to compare the lattice and continuum flows in Bourgain spaces and obtain the stated $L^2$ rate without the above interpolation loss.
\end{remark}

\begin{remark}[Approximation on the lattice]
\label{rmk: Approximation on the lattice}

Our choice of continuation provides direct control of the original lattice solution with the same error bound. Indeed, by Lemma~\ref{lem: isometry property of Ch and Dh},
$\mathfrak{D}_h$ is an $L^2$-contraction satisfying
$\mathfrak{D}_h\mathfrak{C}_h=\textup{Id}_{L_z^2(h\Z)}$.
Thus, applying $\mathfrak{D}_h$ to \eqref{ineq: continuum limit},
we obtain
\[
\sup_{|t|\leq T}
\Big\|
r_h(t,\cdot)
-\mathfrak{D}_h\Big[
w^{+,h}\bigl(t,\cdot-\tfrac{t}{h^2}\bigr)
+w^{-,h}\bigl(t,\cdot+\tfrac{t}{h^2}\bigr)
\Big]
\Big\|_{L_z^2(h\Z)}\leq K_se^{K_0T}h^{\frac{2s}{5}}.
\]
Similarly, applying $\mathfrak{D}_1$ to
\eqref{ineq: small-amplitude limit} yields an approximation
of the original strain:
\[
\sup_{|t|\leq T/h^3}
\Big\|
r(t,\cdot)
-h^2\mathfrak{D}_1\Big[
w^{+,h}\bigl(h^3t,h(\cdot-t)\bigr)
+w^{-,h}\bigl(h^3t,h(\cdot+t)\bigr)
\Big]
\Big\|_{\ell^2(\Z)}\leq K_se^{K_0T}h^{\frac32+\frac{2s}{5}}.
\]
\end{remark}

\subsection{Proof strategy and new ideas}

Our analysis follows the strategy of Hong, Kwak, and Yang~\cite{HKY2021}. We first reformulate the FPUT as a coupled system of integral equations (see Subsection~\ref{subsec: Derivation of the coupled FPUT system}), a form suited to the Fourier-analytic methods of the low-regularity theory of Kenig, Ponce, and Vega~\cite{KPV1996}. By establishing linear and bilinear estimates for the FPUT in Fourier restriction norms, also known as Bourgain norms~\cite{Bou1993-1,Bou1993-2}, we obtain solution bounds uniform in $h$. We then estimate the KdV approximation error. Lowering the regularity threshold and extending the approximation to longer time intervals, however, require several new ideas to overcome the obstacles encountered in the earlier work.

The first obstacle arises when comparing solutions of the FPUT with those of the KdV in $L^2$. To estimate their difference, we use Bourgain norms to handle the derivative nonlinearities. These norms, however, depend sensitively on the underlying linear flows: choosing a norm associated with either the FPUT phase or the Airy phase can impose additional regularity requirements on the solution of the other equation. In the earlier work~\cite{HKY2021}, this difficulty was avoided by using mixed space-time norms, but the maximal-function estimate required $s>\frac34$. In this article, we instead introduce the frequency-localized FPUT \eqref{eq: lattice auxiliary equation} as an intermediate model and first approximate the full FPUT dynamics by those of this auxiliary equation (Lemma \ref{lem: first reduction}). On the chosen low-frequency range, the FPUT and Airy Bourgain norms are comparable uniformly in $h$, allowing us to estimate the low-frequency approximation error in a common Bourgain norm (Lemma \ref{lem: Low-frequency comparison of the FPUT and Airy flows}). The remaining high-frequency part of the KdV solution is controlled by its Sobolev regularity. We previously used this idea in related settings~\cite{HY2024,HJY2025}.

A second difficulty concerns the comparison of lattice and continuum functions. In the earlier work~\cite{HKY2021}, we used piecewise-linear interpolation. However, this interpolation produces frequency copies outside the fundamental domain even for low-frequency lattice inputs, obstructing the flow identities required by the present argument. We instead use the continuation operator $\mathfrak{C}_h$ to compare the frequency-localized FPUT and KdV flows (Lemma \ref{lem: second reduction}). On the relevant low-frequency range, this operator identifies the lattice Bourgain norm with the continuum Bourgain norm associated with the extended FPUT phase (Lemma \ref{lemma: Low-frequency identification of Bourgain norms}).

The final and main obstacle concerns the use of conservation laws to iterate the local approximation estimate. Conservation of the rescaled FPUT Hamiltonian \eqref{eq: rescaled FPUT-Hamiltonian} provides only $L^2$-level control uniformly for sufficiently small $h>0$, while we also seek to estimate the approximation error in $L^2$. Since our quantitative error estimates rely on a positive regularity gap between the solution bounds and the error norm, this conservation law alone does not provide the higher-regularity bounds needed for the iteration. To overcome this obstacle, we establish persistence of regularity uniformly in $h$ for the coupled FPUT system (Proposition~\ref{prop: exponential bound}).
More precisely, under the assumptions \(\|u_{h,0}^\pm\|_{L_z^2(h\Z)}\leq R_0\) and \(\|u_{h,0}^\pm\|_{H_z^s(h\Z)}\leq R_s\) for \(s>0\), we obtain the higher Sobolev bound
\[
\sum_\pm \|u_h^\pm(t)\|_{H_z^s(h\Z)}
\leq Ce^{K|t|}R_s,
\]
which holds uniformly for sufficiently small $h>0$. Such a bound follows by iterating the uniform local $L^2$ theory
with persistence of regularity, using the conserved quantity
to control the $L^2$ norm.
Our key observation is that this exponential growth is compatible
with a quantitative approximation estimate on time intervals
of order $\log(1/h)$.
We use the same argument in other settings~\cite{HJ2026-1,HJ2026-2}.
In recent work, Koch and Liu~\cite{KL2026} address the same obstacle
by measuring the approximation error in a negative Sobolev norm
and using the rescaled Hamiltonian to iterate the local approximation
estimate.

\subsection{Organization of the paper}
The rest of the paper is organized as follows. In Section~\ref{sec: Preliminaries}, we derive the coupled FPUT system and introduce the function spaces and their basic properties used below. In Section~\ref{sec: bilinear estimates}, we establish the bilinear estimates and bounds for the translation correction, while we treat the higher-order remainder for general potentials in Section~\ref{sec: Estimates for the translation correction and high-order remainder}. In Section~\ref{sec: Uniform bounds and persistence of regularity}, we prove uniform local bounds and exponential-in-time Sobolev bounds. Finally, in Section~\ref{sec: KdV Approximation and Global Extension}, we establish the local KdV approximation, extend it globally by iteration, and deduce the continuum and small-amplitude limits. Appendix~\ref{appendix: Proof of integral estimates} contains the proof of the integral estimates underlying the bilinear estimates.

\subsection{Notation}\label{subsec: Notation}
For the reader's convenience, we summarize the relevant notation in
the following table.
\begin{center}
\small
\renewcommand{\arraystretch}{1.2}
\setlength{\tabcolsep}{4pt}
\begin{tabular}{@{}clc@{}}
\hline
\textbf{symbol}
&
\textbf{description}
&
\textbf{definition}
\\
\hline
$h\Z$, $\T_h$
&
lattice and its Fourier domain
&
$\T_h=\R/(\frac{2\pi}{h}\Z)$
\\
$z$, $x$
&
lattice and continuum spatial variables
&
$z\in h\mathbb{Z},\ x\in\mathbb{R}$
\\
$[\xi]$
&
representative of $\xi$ modulo $2\pi/h$
&
\([\xi]\in[-\frac{\pi}{h},\frac{\pi}{h})\)
\\
$\mathfrak{D}_h,$ $\mathfrak{C}_h$
&
discretization and continuation
&
\eqref{eq: definition of Dh}, \eqref{eq: definition of Ch}
\\
$P_{\leq N}^h$, $P_{\leq N}$
&
lattice and continuum frequency projections
&
Section~\ref{subsec: Notation}
\\
$X_{h,\pm}^{s,b}$, $X_\pm^{s,b}$
&
FPUT and Airy Bourgain spaces
&
Section~\ref{subsec: Bourgain spaces}
\\
\hline
\end{tabular}
\end{center}

Throughout the paper, \(h\in(0,1]\) denotes the lattice spacing, and we consider the continuum limit \(h\to0\). We use the subscript \(h\), as in \(u_h\) and \(v_h\), for functions on \(h\mathbb Z\) and their Fourier transforms on \(\mathbb T_h\), while functions on \(\mathbb R\), such as \(w\), carry no subscript. We denote the lattice and continuum spatial variables by \(z\in h\mathbb Z\) and \(x\in\mathbb R\), respectively, and use \(\xi\) for the corresponding Fourier variable in either \(\mathbb T_h\) or \(\mathbb R\). In both settings, \(t\) denotes time and \(\tau\) its Fourier-dual variable.

For convenience, we regard the lattice Fourier domain
\(\mathbb{T}_h:=\mathbb{R}/(\frac{2\pi}{h}\mathbb{Z})\) as the periodic interval \([-\frac{\pi}{h},\frac{\pi}{h})\). This viewpoint is useful for evaluating frequency integrals, but periodicity must be handled
carefully because Fourier convolution may wrap frequencies across the boundary. Accordingly, whenever a frequency expression may lie outside this interval, we write \([\xi]\) for its unique representative modulo
\(\frac{2\pi}{h}\) in \([-\frac{\pi}{h},\frac{\pi}{h})\); when it is clear that \(\xi\in[-\frac{\pi}{h},\frac{\pi}{h})\), we simply write \(\xi\).

For nonnegative numbers $A$ and $B$, we write $A\ls B$ if $A\leq CB$ for some constant $C>0$, independent of $h$, and $A\sim B$ if $A\ls B$ and $B\ls A$. We denote the Japanese bracket by
$$
\la y\ra:=\sqrt{1+|y|^2}.
$$
For a set $A\subset\R$, let $\mathbbm{1}_A$ denote its characteristic function and define the sharp frequency projection $P_A$ on $\R$ by
$$
\widehat{(P_Af)}(\xi)=\mathbbm{1}_A(\xi)\hat{f}(\xi).
$$
The corresponding projection on $h\Z$, with symbol $\mathbbm{1}_A$ restricted to $[-\frac{\pi}{h},\frac{\pi}{h})$, is denoted by $P_A^h$. In particular, we write $P_{\leq N}:=P_{[-N,N]}$ and $P_{\leq N}^h:=P_{[-N,N]}^h$.
Throughout the paper, we set $N_0:=\frac{1}{2}h^{-\frac{2}{5}}$.
For given $T>0$, we denote by $\eta_T(t)$ the smooth time cut-off satisfying
\begin{equation}\label{eq: time cutoff}
\eta_T(t):=\eta(t/T)
\quad \textup{where}\quad
0\leq\eta\leq1, \quad \eta(t)\equiv 1  \textup{ on }[-1,1],\quad\operatorname{supp}\eta\subset[-2,2].
\end{equation}

\subsection{Tool and computational resource disclosure}
During the preparation of this work, the authors used ChatGPT (OpenAI; GPT-5.5, GPT-5.6, and GPT-6.0) to improve the writing and presentation. In particular, ChatGPT assisted in consolidating multiple integral estimates proved by the authors and used in the proofs of Lemmas~\ref{lem: FPUT bilinear estimates},
\ref{lem: bilinear estimates for translation correction},
and~\ref{lem: estimates for the translation correction}
into the unified integral estimate in Lemma~\ref{lem: unified integral bounds}. The authors independently verified all mathematical statements, proofs, and references in the final manuscript and take full
responsibility for its contents.

\subsection{Acknowledgements}
This work was supported by the National Research Foundation of Korea (NRF) grant funded by the Korean government (MSIT) (No. RS-2023-00219980 and RS-2026-25479401).

\section{Preliminaries}\label{sec: Preliminaries}
We first derive the coupled FPUT system for the moving-frame profiles. We then collect basic lattice Fourier identities and Sobolev norm equivalences and establish the properties of the continuation and discretization operators. Finally, we introduce the Bourgain spaces and recall the linear estimates used below.

\subsection{Derivation of the coupled FPUT system}\label{subsec: Derivation of the coupled FPUT system}
Following \cite{HKY2021}, we reformulate the FPUT system \eqref{eq:FPUT-system-expanded} as the coupled system
\eqref{eq: coupled FPUT integral form}, which is well suited to Fourier analysis. We first write \eqref{eq:FPUT-system-expanded} as the first-order system
\[
\pa_t\mathbf y
=
\frac{\nabla_h}{h^2}
\begin{bmatrix}
0&1\\
1&0
\end{bmatrix}
\mathbf y
+
\frac{\nabla_h}{2}
\big(r_h^2+\mathcal R_h[r_h]\big)
\begin{bmatrix}
0\\
1
\end{bmatrix},
\]
where \(\mathbf{y} =[r_h\ \ h^2\nabla_h^{-1}\pa_tr_h]^\top\). Then, introducing
\[\mathbf r_h
:=
\begin{bmatrix}
r_h^+\\
r_h^-
\end{bmatrix}
=
\frac12
\begin{bmatrix}
1&-1\\
1&1
\end{bmatrix}
\mathbf y
=
\begin{bmatrix}
\frac12(r_h-h^2\nabla_h^{-1}\pa_t r_h)\\
\frac12(r_h+h^2\nabla_h^{-1}\pa_t r_h)
\end{bmatrix},\]
we diagonalize the linear part:
\[
\pa_t\mathbf r_h
=
\frac{\nabla_h}{h^2}
\begin{bmatrix}
-1&0\\
0&1
\end{bmatrix}
\mathbf r_h
+
\frac{\nabla_h}{4}
\big(r_h^2+\mathcal R_h[r_h]\big)
\begin{bmatrix}
-1\\
1
\end{bmatrix},
\]
equivalently,
\begin{equation}\label{eq: def of transformed variables}
\left\{\begin{aligned}
\pa_t r_h^\pm
&=
\mp\frac{\nabla_h}{h^2}r_h^\pm
\mp\frac{\nabla_h}{4}
\big(r_h^2+\mathcal R_h[r_h]\big),\\
r_h^\pm(0)
&=
r_{h,0}^\pm
:=
\frac12
\big(r_h(0)\mp h^2\nabla_h^{-1}\pa_t r_h(0)\big).
\end{aligned}\right.
\end{equation}
Next, introducing the moving-frame profiles
\begin{equation}\label{eq: def of uh}
u_h^\pm(t,z)
:=
e^{\pm\frac{t}{h^2}\pa_h}r_h^\pm(t,z),
\end{equation}
we obtain
\begin{equation}\label{eq: coupled FPUT PDE form}
\pa_t u_h^\pm
=
\mp\frac{\nabla_h-\pa_h}{h^2}u_h^\pm
\mp\frac14\nabla_h
\big\{
e^{\pm\frac{t}{h^2}\pa_h}(r_h^2)
+
e^{\pm\frac{t}{h^2}\pa_h}\mathcal R_h[r_h]
\big\},
\end{equation}
where 
\begin{equation}\label{decomposition of rh}
r_h(t,z)
=
r_h^+(t,z)+r_h^-(t,z)
=
e^{-\frac{t}{h^2}\pa_h}u_h^+(t,z)
+
e^{\frac{t}{h^2}\pa_h}u_h^-(t,z).
\end{equation}

\begin{remark}
The decomposition \eqref{decomposition of rh} describes two
counter-propagating components through Fourier multipliers.
At times \(t=h^3m\), \(m\in\Z\), these multipliers act as ordinary
lattice translations:
\[
r_h(t,z)
=
e^{-hm\pa_h}u_h^+(t,z)
+
e^{hm\pa_h}u_h^-(t,z)=
u_h^+\left(t,z-\frac{t}{h^2}\right)
+
u_h^-\left(t,z+\frac{t}{h^2}\right),
\]
since the shifts \(\pm \frac{t}{h^2}=\pm hm\) preserve \(h\Z\).
For \(t\notin h^3\Z\), the points \(z\pm \frac{t}{h^2}\) lie outside
\(h\Z\), so these expressions are not defined for
lattice functions. The Fourier multiplier formulation
\eqref{decomposition of rh}, however, remains valid for every
\(t\in\R\).
\end{remark}
To compare \(r_h(t,z)\) with counter-propagating KdV flows, we rewrite the quadratic term \(e^{\pm\frac{t}{h^2}\pa_h}(r_h^2)\) in \eqref{eq: coupled FPUT PDE form} as
\[
\big(e^{\pm\frac{t}{h^2}\pa_h}r_h\big)^2
+\mathcal T_h^\pm[r_h]=
\big(
u_h^\pm
+e^{\pm\frac{2t}{h^2}\pa_h}u_h^\mp
\big)^2
+\mathcal T_h^\pm[r_h],
\]
where 
\begin{equation}\label{eq: product translation correction}
\mathcal T_h^\pm[r_h]
:=
e^{\pm\frac{t}{h^2}\pa_h}(r_h^2)
-
\big(e^{\pm\frac{t}{h^2}\pa_h}r_h\big)^2
\end{equation}
is the translation correction.

Finally, to obtain an integral formulation, we define the linear FPUT propagators by
\begin{equation}\label{eq: linear FPU flow}
S_h^\pm(t)
:=e^{\pm its_h(-i\partial_h)}
=e^{\mp\frac{t}{h^2}(\nabla_h-\pa_h)},
\end{equation}
where 
\begin{equation}\label{eq: FPUT phase function}
s_h(\xi):=
\frac{1}{h^2}\bigg\{[\xi]-\frac{2}{h}\sin\bigg(\frac{h[\xi]}{2}\bigg)\bigg\}. 
\end{equation}
Then, by Duhamel's formula, \eqref{eq: coupled FPUT PDE form} is equivalent
to the \textit{coupled FPUT system}
\begin{equation}\label{eq: coupled FPUT integral form}
\boxed{\quad
\left\{
\begin{aligned}
u_h^\pm(t)
&=
S_h^\pm(t)u_{h,0}^\pm\mp\frac14
\int_0^t S_h^\pm(t-t')\nabla_h
\Big\{
\big(
u_h^\pm
+e^{\pm\frac{2t'}{h^2}\pa_h}u_h^\mp
\big)^2\\
&\qquad\qquad\qquad\qquad\qquad\qquad\qquad
+\mathcal T_h^\pm[r_h]
+e^{\pm\frac{t'}{h^2}\pa_h}\mathcal R_h[r_h]
\Big\}(t')dt',\\
u_h^\pm(0)
&=
u_{h,0}^\pm
=
\frac12
\big(
r_h(0)\mp h^2\nabla_h^{-1}\pa_t r_h(0)
\big),
\end{aligned}
\right.
\quad}
\end{equation}
where \(\mathcal T_h^\pm[r_h]\) is the translation correction, and
\(e^{\pm\frac{t}{h^2}\pa_h}\mathcal R_h[r_h](t)\)
is the higher-order remainder in the moving frame;
see \eqref{eq: product translation correction} and \eqref{eq:higher-order-remainder}.

\begin{remark}[Preservation of real-valuedness]\label{remark: Preservation of real-valuedness}
Real-valued initial data give real-valued profiles throughout their
existence interval. Indeed, since \(s_h\) is odd and the symbol of
\(\nabla_h\) is purely imaginary and odd, we have \(\overline{S_h^\pm(t)f_h}
=
S_h^\pm(t)\overline{f_h}\) and \(\overline{\nabla_h f_h}
=
\nabla_h\overline{f_h}\). The same commutation property holds for \(\nabla_h^{-1}\) and
\(e^{a\pa_h}\), \(a\in\R\). Thus, these operators map real-valued
functions to real-valued functions, and \(u_{h,0}^\pm\) are real
whenever \(r_h(0)\) and \(\pa_t r_h(0)\) are real.
Moreover, the nonlinear terms are real for real-valued profiles,
since they involve products, the operator \(e^{a\partial_h}\), and the real-valued
function \(V'\). Starting with \(S_h^\pm(t)u_{h,0}^\pm\), every
successive approximation in the Duhamel iteration is therefore
real-valued, and so is its limit \(u_h^\pm(t)\).
\end{remark}

\begin{remark}[Formal KdV limit in integral form]
\label{rmk: Formal KdV limit in integral form}
The integral formulation \eqref{eq: coupled FPUT integral form}
makes the formal KdV limit transparent.
For \(h|\xi|\ll1\), Taylor's theorem gives \(s_h(\xi)
=
\frac{\xi^3}{24}+O(h^2|\xi|^5)\) and \(\frac{2i}{h}\sin\big(\frac{h\xi}{2}\big)
=
i\xi+O(h^2|\xi|^3)\). Thus, as \(h\to0\), the linear FPUT flow \(S_h^\pm(t)\) and
the operator \(\nabla_h\) formally converge to the Airy flow
\(S^\pm(t)\) and \(\pa_x\), respectively, where
\[
S^\pm(t)
:=
e^{\pm it s(-i\partial_x)}
=
e^{\mp\frac{t}{24}\pa_x^3}
\]
and \(s(\xi):=\frac{\xi^3}{24}\). Moreover, the contributions of the higher-order remainder
\(\mathcal R_h[r_h]\), the rapidly translated interactions
involving the counter-propagating component, and the translation
correction \(\mathcal T_h^\pm[r_h]\) to the Duhamel integral
in \eqref{eq: coupled FPUT integral form} are formally negligible
as \(h\to0\). Consequently, the limiting profiles \(w^\pm\)
formally satisfy
\begin{equation}\label{eq:KdV integral form}
\boxed{\quad
w^\pm(t)
=
S^\pm(t)w_0^\pm
\mp
\frac14
\int_0^t
S^\pm(t-t')\pa_x
\big(w^\pm(t')\big)^2dt'.
\quad}
\end{equation}
These are the integral formulations of the two decoupled KdV equations \eqref{eq:KdV}.
\end{remark}

\subsection{Lattice Fourier analysis and Sobolev norms}
We collect basic Fourier identities and Sobolev norm equivalences on $h\Z$.
Since the lattice Fourier transform is defined on $\T_h\simeq[-\frac{\pi}{h},\frac{\pi}{h})$, we adopt the following convention for its $\frac{2\pi}{h}$-periodic extension to $\R$. This convention is important for almost translations and nonlinear products, where interactions crossing the boundary of $[-\frac{\pi}{h},\frac{\pi}{h})$ give rise to frequency-wrapping. For $f_h\in L_z^2(h\Z)$, we denote the $\frac{2\pi}{h}$-periodic extension of its lattice Fourier transform by
$$
\widetilde{\mathcal{F}}_hf_h(\xi):=\widehat{f_h}([\xi]), \quad \xi\in\R.
$$
In particular, for $a\in\R$,
$$
\widetilde{\mathcal{F}}_h\big(e^{a\pa_h}f_h\big)(\xi)
=e^{ia[\xi]}\widetilde{\mathcal{F}}_hf_h(\xi)\quad \textup{for a.e. }\xi\in\R. 
$$

\begin{lemma}
Let $f_h, g_h\in L_z^2(h\Z)$. Then the following identities hold.
\begin{enumerate}[(1)]
\item (Parseval’s identity) One has
$$
h\sum_{z\in h\Z}f_h(z)\overline{g_h(z)}
=\frac{1}{2\pi}\int_{-\frac{\pi}{h}}^{\frac{\pi}{h}}\widehat{f_h}(\xi)\overline{\widehat{g_h}(\xi)} d\xi.
$$
\item For $\xi\in [-\frac{\pi}{h}, \frac{\pi}{h})$, we have
$$
\mathcal{F}_h(f_hg_h)(\xi)
=\frac{1}{2\pi}\int_{-\frac{\pi}{h}}^{\frac{\pi}{h}}\widehat{f_h}([\xi-\xi_1]) \widehat{g_h}(\xi_1) d\xi_1.
$$
\end{enumerate}
\end{lemma}
\begin{lemma}[Triangle inequality on $\T_h$]\label{lem: triangle inequality on Th}
If $\xi, \xi_1\in[-\frac{\pi}{h},\frac{\pi}{h})$, then
\begin{equation}\label{ineq: triangle inequality on Th}
|\xi|\leq |\xi_1|+\big|[\xi-\xi_1]\big|.
\end{equation}
Moreover, if $[\xi-\xi_1]\neq \xi-\xi_1$, then
\begin{equation}\label{ineq: triangle inequality on Th for wrapped frequency}
|\xi|\leq \frac{\pi}{h}\leq 
|\xi_1|+\big|[\xi-\xi_1]\big|.
\end{equation}
\end{lemma}
\begin{proof}
If $[\xi-\xi_1]=\xi-\xi_1$, \eqref{ineq: triangle inequality on Th} follows from the triangle inequality. Otherwise, we have
$$
\big|\xi-\xi_1-[\xi-\xi_1]\big|=\frac{2\pi}{h},
$$
since $|\xi-\xi_1|<\frac{2\pi}{h}$. Thus, 
$$
|\xi_1|+|[\xi-\xi_1]|\geq \frac{2\pi}{h}-|\xi|\geq\frac{\pi}{h}\geq |\xi|.
$$
This completes the proof.
\end{proof}

We next introduce lattice Sobolev norms and their equivalent formulations, which will be used to estimate the higher-order remainder \(\mathcal R_h[r_h]\);
see Lemma~\ref{lem: estimates for the higher-order remainder}.
\begin{definition}[Lattice \(W^{s,p}\)-norms]
\label{def: lattice Sobolev norm}
For \(s\geq0\) and \(1<p<\infty\), define the inhomogeneous (resp., homogeneous) lattice Sobolev norm
\[
\|f_h\|_{W_z^{s,p}(h\Z)}
:=
\|\la\pa_h\ra^s f_h\|_{L_z^p(h\Z)},
\qquad\Big(\textup{resp., }
\|f_h\|_{\dot W_z^{s,p}(h\Z)}
:=
\||\pa_h|^s f_h\|_{L_z^p(h\Z)}\Big),
\]
where \(\pa_h\) is the lattice Fourier multiplier with symbol
\(i[\xi]\). In particular,
\(H_z^s(h\Z)=W_z^{s,2}(h\Z)\) and
\(\dot H_z^s(h\Z)=\dot W_z^{s,2}(h\Z)\).
\end{definition}

The following lemma gives equivalent expressions for the Sobolev
seminorms in Definition~\ref{def: lattice Sobolev norm} using
\(\nabla_h\), whose Fourier symbol is
\(\frac{2i}{h}\sin\big(\frac{h[\xi]}{2}\big)\).
For first-order seminorms, one may also use the forward or backward
difference operator, defined by
\[
\pa_h^+f_h(z)
:=
\frac{f_h(z+h)-f_h(z)}{h},
\qquad
\pa_h^-f_h(z)
:=
\frac{f_h(z)-f_h(z-h)}{h},
\qquad
z\in h\Z.
\]
\begin{lemma}[Equivalence of lattice Sobolev seminorms
{\cite[Proposition~1.2]{HY2019-2}}]
Let \(h\in(0,1]\), \(1<p<\infty\), and \(s\geq0\).
Then, for $f_h\in W_z^{s,p}(h\Z)$, we have
\begin{equation}\label{ineq: norm equivalence of the lattice Sobolev norm}
\|f_h\|_{\dot W_z^{s,p}(h\Z)}
\sim
\big\||\nabla_h|^s f_h\big\|_{L_z^p(h\Z)}.
\end{equation}
In particular, when \(s=1\),
\[
\|f_h\|_{\dot W_z^{1,p}(h\Z)}
\sim
\|\pa_h^\pm f_h\|_{L_z^p(h\Z)}.
\]
The implicit constants may depend on \(s\) and \(p\), but are
independent of \(h\).
\end{lemma}

For fractional orders \(0<s<1\), the homogeneous \(H^s\)-norm
also admits a characterization by spatial differences.

\begin{lemma}[Fractional difference characterization]
Let \(0<h\leq1\) and \(0<s<1\). Then, for \(f_h\in H_z^s(h\Z)\),
\begin{equation}\label{ineq: characterization of the lattice Sobolev norm}
\|f_h\|_{\dot H_z^s(h\Z)}^2
\sim_s
h\sum_{k\in\Z\setminus\{0\}}
\frac{
\|f_h(\cdot+hk)-f_h\|_{L_z^2(h\Z)}^2
}{|hk|^{1+2s}}.
\end{equation}
\end{lemma}

\begin{proof}
We follow the argument in \cite[Lemma~4.3]{HY2019-1}.
By Parseval's identity and Fubini's theorem, 
\[
h\sum_{k\in\Z\setminus\{0\}}
\frac{
\|f_h(\cdot+hk)-f_h\|_{L_z^2}^2
}{|hk|^{1+2s}}=
\frac{1}{2\pi}
\int_{-\frac{\pi}{h}}^{\frac{\pi}{h}}
h^{-2s}
\left(
\sum_{k\in\Z\setminus\{0\}}
\frac{4\sin^2(\frac{hk\xi}{2})}{|k|^{1+2s}}
\right)
|\widehat{f_h}(\xi)|^2 d\xi.
\]
Therefore, it suffices to show that for \(|y|\leq\pi\),
\begin{equation}\label{ineq: characterization of the lattice Sobolev norm step 1}
\sum_{k\in\Z\setminus\{0\}}
\frac{4\sin^2(\frac{ky}{2})}{|k|^{1+2s}}
\sim_s |y|^{2s}.
\end{equation}
The case \(y=0\) is immediate. Suppose that \(0<|y|\leq\frac{1}{2}\). Then, by \(4\sin^2(\frac{ky}{2})\lesssim\min\{|ky|^2,1\}\),
\[
\begin{aligned}
\sum_{k\in\Z\setminus\{0\}}
\frac{4\sin^2(\frac{ky}{2})}{|k|^{1+2s}}
&\lesssim
|y|^2
\sum_{1\leq|k|\leq|y|^{-1}}
|k|^{1-2s}
+
\sum_{|k|>|y|^{-1}}
\frac{1}{|k|^{1+2s}}\\
&\lesssim_s
|y|^2|y|^{2s-2}+|y|^{2s}
\lesssim_s |y|^{2s}.
\end{aligned}
\]
For the lower bound, we observe that \(\big|\sin\big(\frac{ky}{2}\big)\big|
\geq\sin(\frac{1}{4})>0\) if \(\frac{1}{2|y|}\leq|k|\leq\frac{1}{|y|}\). Hence,
\[
\sum_{k\in\Z\setminus\{0\}}
\frac{4\sin^2(\frac{ky}{2})}{|k|^{1+2s}}
\gtrsim
\sum_{\frac{1}{2|y|}\leq|k|\leq\frac{1}{|y|}}
\frac{1}{|k|^{1+2s}}\gtrsim_s
|y|^{-1}|y|^{1+2s}
=
|y|^{2s}.
\]
It remains to consider \(\frac{1}{2}\leq|y|\leq\pi\).
In this region,
\[
4\sin^2\bigg(\frac14\bigg)
\leq
4\sin^2\bigg(\frac y2\bigg)
\leq
\sum_{k\in\Z\setminus\{0\}}
\frac{4\sin^2(\frac{ky}{2})}{|k|^{1+2s}}\leq
4\sum_{k\in\Z\setminus\{0\}}
\frac{1}{|k|^{1+2s}}
\lesssim_s1.
\]
Since \(|y|^{2s}\sim_s1\) on this region,
\eqref{ineq: characterization of the lattice Sobolev norm step 1}
follows. This completes the proof.
\end{proof}

\subsection{Continuation/discretization}
As we noted before, the comparison between the lattice FPUT profiles and the continuum KdV solutions requires transferring functions between $h\Z$ and $\R$. In this subsection, we establish the properties of the continuation operator $\mathfrak{C}_h$ and discretization operator $\mathfrak{D}_h$.
\begin{lemma}[Fourier representations of the continuation/discretization operators]
Suppose that $f_h\in L_z^2(h\Z)$ and $f\in L_x^2(\R)$. Then, for almost every \(\xi\in\R\),
\begin{equation}\label{eq: Ch FT}
\widehat{\mathfrak{C}_hf_h}(\xi)=\hat{\varphi}(h\xi)\widetilde{\mathcal{F}}_hf_h(\xi),
\end{equation}
where \(\widetilde{\mathcal{F}}_hf_h\) is the periodic extension of the lattice Fourier transform \(\mathcal{F}_hf_h\) to $\mathbb{R}$. On the other hand, for almost every \(\xi\in[-\frac{\pi}{h}, \frac{\pi}{h})\), we have
\begin{equation}\label{eq: Dh FT}
\mathcal{F}_h(\mathfrak{D}_hf)(\xi)
=\sum_{k\in\Z}\overline{\hat{\varphi}(h\xi+2\pi k)}\hat{f}\bigg(\xi+\frac{2\pi k}{h}\bigg).
\end{equation}
\end{lemma}

\begin{proof}
By direct computation, we obtain
$$
\begin{aligned}
\widehat{\mathfrak{C}_hf_h}(\xi)
&=
\int_\R\sum_{z\in h\Z}\varphi\bigg(\frac{x-z}{h}\bigg)f_h(z)e^{-ix\xi} dx
=
\sum_{z\in h\Z}f_h(z)\int_\R\varphi\bigg(\frac{x-z}{h}\bigg)e^{-ix\xi} dx\\
&=
\sum_{z\in h\Z}f_h(z) h\hat{\varphi}(h\xi)e^{-iz\xi}
=
\hat{\varphi}(h\xi) h\sum_{z\in h\Z}f_h(z) e^{-iz\xi}
=
\hat{\varphi}(h\xi)\widetilde{\mathcal{F}}_hf_h(\xi).
\end{aligned}
$$
For \eqref{eq: Dh FT}, we refer to \cite[Lemma 2.1]{NT2021}.
\end{proof}

The next lemma gives the basic mapping properties of the discretization/continuation operators. The continuation operator $\mathfrak{C}_h$ is an $L^2$-isometry, while its adjoint $\mathfrak{D}_h$ is a partial isometry satisfying $\mathfrak{D}_h\mathfrak{C}_h=\textup{Id}$. Moreover, the corresponding $H^s$ bounds are uniform in $h$.
\begin{lemma}\label{lem: isometry property of Ch and Dh}
Let $0< h\leq 1$. Then, 
\begin{equation}\label{eq: L2 isometry}
\|\mathfrak{C}_hf_h\|_{L_x^2(\R)}=\|f_h\|_{L_z^2(h\Z)}
\quad
\textup{and}
\quad
\mathfrak{D}_h\mathfrak{C}_h=\textup{Id}_{L_z^2(h\Z)}.
\end{equation}
Moreover, for every $s\geq 0$, we have
\begin{equation}\label{ineq: Ch Hs equivalence}
\|f_h\|_{H_z^s(h\Z)}
\leq
\|\mathfrak{C}_hf_h\|_{H_x^s(\R)}
\ls
\|f_h\|_{H_z^s(h\Z)}
\end{equation}
and
\begin{equation}\label{ineq: Dh Hs bound}
\|\mathfrak{D}_hf\|_{H_z^s(h\Z)}
\leq\|f\|_{H_x^s(\R)}.
\end{equation}
\end{lemma}

\begin{proof}
For \eqref{eq: L2 isometry}, we refer to \cite[Lemma A.1]{NT2021}. For \eqref{ineq: Ch Hs equivalence}, we recall from Assumption \ref{assumption for varphi} that \(\hat{\varphi}\) is supported in \([-c, c]\) for some $c\in (\pi,2\pi)$. Then, by Plancherel's identity and \eqref{eq: Ch FT}, we write 
$$
\begin{aligned}
\|\mathfrak{C}_hf_h\|_{H_x^s(\R)}^2
&=\frac{1}{2\pi}
\int_\R\langle\xi\rangle^{2s}|\hat{\varphi}(h\xi)|^2|\widetilde{\mathcal{F}}_hf_h(\xi)|^2 d\xi\\
&=\frac{1}{2\pi}\sum_{k\in\{-1,0,1\}}
\int_{-\frac{\pi}{h}+\frac{2\pi k}{h}}^{\frac{\pi}{h}+\frac{2\pi k}{h}}\langle\xi\rangle^{2s}|\hat{\varphi}(h\xi)|^2|\widetilde{\mathcal{F}}_hf_h(\xi)|^2 d\xi\\
&=
\frac{1}{2\pi}\sum_{k\in\{-1,0,1\}}\int_{-\frac{\pi}{h}}^{\frac{\pi}{h}}\bigg\langle\xi+\frac{2\pi k}{h} \bigg\rangle^{2s}|\hat{\varphi}(h\xi+2\pi k)|^2|\widetilde{\mathcal{F}}_hf_h(\xi)|^2 d\xi,
\end{aligned}
$$
where in the last step, we used that \(\widetilde{\mathcal{F}}_hf_h(\xi+\frac{2\pi k}{h})=\widetilde{\mathcal{F}}_hf_h(\xi)\). If $k=\pm1$, $\xi\in[-\frac{\pi}{h}, \frac{\pi}{h}]$, and $\hat{\varphi}(h\xi+2\pi k)\neq0$, then
$$
|\xi|\leq\bigg|\xi+\frac{2\pi k}{h}\bigg|\leq \frac{c}{2\pi-c}|\xi|.
$$
Thus, replacing \(\langle\xi+\frac{2\pi k}{h} \rangle^{2s}\) by \(\langle\xi\rangle^{2s}\) in the above integral and using \eqref{phi assumption 1}, we prove \eqref{ineq: Ch Hs equivalence}. 

By \eqref{eq: Dh FT}, Cauchy–Schwarz and \eqref{phi assumption 1}, we obtain
$$
\la\xi\ra^{2s}|\mathcal{F}_h(\mathfrak{D}_hf)(\xi)|^2
\leq
\sum_{k\in\Z}\bigg\langle\xi+\frac{2\pi k}{h}\bigg\rangle^{2s}\bigg|\hat{f}\bigg(\xi+\frac{2\pi k}{h}\bigg)\bigg|^2,
$$
since $|\xi|\leq |\xi+\frac{2\pi k}{h}|$ for $\xi\in[-\frac{\pi}{h},\frac{\pi}{h})$. Then integrating over $[-\frac{\pi}{h},\frac{\pi}{h})$, \eqref{ineq: Dh Hs bound} follows.
\end{proof}
The next lemma shows that \(\mathfrak{C}_h\mathfrak{D}_h\) acts as the identity on functions with Fourier support in $|\xi|\leq\frac{\alpha}{h}$. In this range, the identification operators also commute with frequency localization.
\begin{lemma}[Low-frequency compatibility]
Let $0<N\leq \frac{\al}{h}$, and let $f_h\in L_z^2(h\Z)$ and $f\in L_x^2(\R)$. Then for almost every $\xi\in\R$,
\begin{equation}\label{eq: localized Ch FT}
\mathcal{F}_x\big(\mathfrak{C}_hP_{\leq N}^hf_h\big)(\xi)
=\mathbbm{1}_{[-N,N]}(\xi)\hat{\varphi}(h\xi)\widetilde{\mathcal{F}}_hf_h(\xi), \quad 
\mathfrak{C}_hP_{\leq N}^h=P_{\leq N}\mathfrak{C}_h,
\end{equation}
and, for almost every $\xi\in[-\tfrac{\pi}{h},\tfrac{\pi}{h})$,
\begin{equation}\label{eq: localized Dh FT}
\mathcal{F}_h\big(\mathfrak{D}_hP_{\leq N}f\big)(\xi)
=\mathbbm{1}_{[-N,N]}(\xi)\overline{\hat{\varphi}(h\xi)}\hat{f}(\xi),\qquad \mathfrak{D}_hP_{\leq N}=P_{\leq N}^h\mathfrak{D}_h.
\end{equation}
In particular, $\mathfrak{C}_h\mathfrak{D}_hP_{\leq N}=P_{\leq N}$.
\end{lemma}

\begin{proof}
For $\xi\in\R$, write $\xi^*=\xi-\frac{2\pi k}{h}\in[-\frac{\pi}{h}, \frac{\pi}{h})$ for a unique $k\in\Z$.
Then \eqref{eq: Ch FT} gives
\begin{equation}\label{eq: localized Ch FT step 1}
\mathcal{F}_x\big(\mathfrak{C}_hP_{\leq N}^hf_h\big)(\xi)
=\hat{\varphi}(h\xi^*+2\pi k)\mathbbm{1}_{[-N,N]}(\xi^*)\widehat{f_h}(\xi^*).
\end{equation}
If $|\xi^*|>N$, this expression vanishes. If $|\xi^*|\leq N$, then $|\hat{\varphi}(h\xi^*)|=1$ by \eqref{phi assumption 2}. Thus, the condition \eqref{phi assumption 1} implies that $\hat{\varphi}(h\xi^*+2\pi j)=0$ for any integer $j\neq0$. Therefore, \eqref{eq: localized Ch FT step 1} vanishes unless $k=0$. 
Moreover, since
$$
\mathcal{F}_x(P_{\leq N}\mathfrak{C}_hf_h)(\xi)
=
\mathbbm{1}_{[-N,N]}(\xi)\hat{\varphi}(h\xi)\widetilde{\mathcal{F}}_hf_h(\xi),
$$
we also obtain $\mathfrak{C}_hP_{\leq N}^h=P_{\leq N}\mathfrak{C}_h$. This proves \eqref{eq: localized Ch FT}.

For \eqref{eq: localized Dh FT}, \eqref{eq: Dh FT} yields 
$$
\mathcal{F}_h\big(\mathfrak{D}_hP_{\leq N}f\big)(\xi)=
\sum_{k\in\Z}\overline{\hat{\varphi}(h\xi+2\pi k)}\mathbbm{1}_{[-N,N]}\bigg(\xi+\frac{2\pi k}{h}\bigg)\hat{f}\bigg(\xi+\frac{2\pi k}{h}\bigg).
$$
For $k\neq 0$, we have $|\xi+\frac{2\pi k}{h}|\geq \frac{\pi }{h}>\frac{\alpha}{h}\geq N$. Hence, only the term $k=0$ remains. Taking adjoints in $\mathfrak{C}_hP_{\leq N}^h=P_{\leq N}\mathfrak{C}_h$ and using $\mathfrak{D}_h=\mathfrak{C}_h^*$, we also obtain $\mathfrak{D}_hP_{\leq N}=P_{\leq N}^h\mathfrak{D}_h$.
Finally, using \eqref{eq: localized Ch FT} and \eqref{eq: localized Dh FT} we have
$$
\mathcal{F}_x\big(\mathfrak{C}_h\mathfrak{D}_hP_{\leq N}f\big)(\xi)
=
\mathbbm{1}_{[-N,N]}(\xi)\hat{\varphi}(h\xi)\overline{\hat{\varphi}(h\xi)}\hat{f}(\xi)
=\mathbbm{1}_{[-N,N]}(\xi)\hat{f}(\xi),
$$
since $|\hat{\varphi}(h\xi)|=1$ where $|\xi|\leq N\leq\frac{\alpha}{h}$.
\end{proof}
The low-frequency identities in the preceding lemma imply the following compatibility between lattice and continuum Fourier multipliers.
\begin{lemma}[Low-frequency multiplier identity]\label{lem: Low-frequency multiplier identity}
Let $p(-i\partial_h)$ and $p(-i\partial_x)$ denote the Fourier multipliers on $h\Z$ and $\R$, respectively, associated with the same bounded measurable symbol $p$. If $0<N\leq \frac{\alpha}{h}$, then
\begin{equation}\label{eq: commutativity of multipliers in low frequency}
\mathfrak{C}_hp(-i\partial_h)P_{\leq N}^hf_h=p(-i\partial_x)\mathfrak{C}_hP_{\leq N}^hf_h.
\end{equation}
\end{lemma}
This identity will be used to derive the continuum formulation of the frequency-localized auxiliary equation; see \eqref{eq: continuum auxiliary equation} below.

\subsection{Bourgain spaces}\label{subsec: Bourgain spaces}
We introduce Bourgain spaces, also called \(X^{s,b}\)-spaces or Fourier restriction spaces, to exploit the dispersive smoothing associated with the linear FPUT and Airy flows. Adapted to their respective phase functions, these spaces measure spatial Sobolev regularity and modulation relative to the linear flows. They provide the framework for the bilinear estimates in the next section.

For \(u_h:\R\times h\Z\to\C\) and \(u:\R\times\R\to\C\),
we define the space-time Fourier transforms by
\[
\widetilde{u_h}(\tau,\xi)
=
\mathcal F_{t,h}u_h(\tau,\xi)
:=
h\sum_{z\in h\Z}
\int_\R u_h(t,z)e^{-i(t\tau+z\xi)} dt,
\qquad
(\tau,\xi)\in\R\times\T_h,
\]
\[
\widetilde u(\tau,\xi)
=
\mathcal F_{t,x}u(\tau,\xi)
:=
\int_\R\int_\R
u(t,x)e^{-i(t\tau+x\xi)} dx dt,
\qquad
(\tau,\xi)\in\R\times\R.
\]

For \(s,b\in\R\), we define the Bourgain spaces
\(X_{h,\pm}^{s,b}=X_{h,\pm}^{s,b}(\R\times h\Z)\) and
\(X_\pm^{s,b}=X_\pm^{s,b}(\R\times\R)\), associated with the
linear FPUT flow \(S_h^\pm(t)\) and the Airy flow \(S^\pm(t)\),
respectively, by the norms
\[
\begin{aligned}
\|u_h\|_{X_{h,\pm}^{s,b}}
&:=
\big\|
\la\xi\ra^s
\la\tau\mp s_h(\xi)\ra^b
\widetilde{u_h}(\tau,\xi)
\big\|_{L_{\tau,\xi}^2(\R\times\T_h)},\\
\|u\|_{X_\pm^{s,b}}
&:=
\big\|
\la\xi\ra^s
\la\tau\mp s(\xi)\ra^b
\widetilde u(\tau,\xi)
\big\|_{L_{\tau,\xi}^2(\R\times\R)}.
\end{aligned}
\]
For pairs \(\mathbf u=(u_h^+,u_h^-)\), we use the product space
\[
\mathbf X_h^{s,b}
:=
X_{h,+}^{s,b}\times X_{h,-}^{s,b},
\]
equipped with the norm
\[
\|\mathbf u\|_{\mathbf X_h^{s,b}}
:=
\|u_h^+\|_{X_{h,+}^{s,b}}
+
\|u_h^-\|_{X_{h,-}^{s,b}}.
\]

\begin{remark}[Sensitivity to the phase function]
\label{rmk: sensitivity on the phase function}
Bourgain norms depend on the phase function of the underlying linear flow, and using a different phase may require additional spatial regularity. Although \(s_h(\xi)\to s(\xi)\) for each fixed \(\xi\), this pointwise convergence does not generally yield a comparison between the FPUT and Airy Bourgain norms that is uniform in \(h\). A direct comparison of the FPUT and KdV solutions in a single Bourgain space may therefore require more regularity than is available in our setting. Following \cite{HY2024}, we address this issue by introducing a frequency-localized auxiliary equation.
\end{remark}

We recall the following standard estimates, which hold for Bourgain spaces on both $\R\times h\Z$ and $\R\times\R$, independently of the associated real-valued phase function. We refer to \cite{Tao2006, LP2015} for their proofs and further details.
\begin{lemma}[Basic estimates for the Bourgain spaces]\label{lem: Xsb properties}
Let $s,b\in\R$ and $T\in(0,1]$. Let $\eta_T(t)=\eta(\frac{t}{T})\in C_c^\infty(\R)$ be the smooth time cut-off defined in \eqref{eq: time cutoff}. Let $X^{s,b}$ be a Bourgain space associated with a real-valued phase function, and let $S(t)$ denote the corresponding linear flow.
Then, the following estimates hold.
\begin{enumerate}[$(i)$]
\item (Embedding) If $b>\frac{1}{2}$, then $\|u\|_{C_t(\R;H^s)}\ls \|u\|_{X^{s,b}}$.
\item (Linear flow estimate) 
If $b>\frac{1}{2}$, then
$$
\|\eta_T(t)S(t)f\|_{X^{s,b}}
\ls T^{\frac{1}{2}-b}\|f\|_{H^s}.
$$
\item (Stability with respect to smooth time localization)
If $-\frac{1}{2}<b'\leq b<\frac{1}{2}$, then
$$
\|\eta_T(t)u\|_{X^{s,b'}}
\ls
T^{b-b'}\|u\|_{X^{s,b}}.
$$
Moreover, if $\frac{1}{2}<b\leq1$, then
$$
\|\eta_T(t)u\|_{X^{s,b}}
\ls
T^{\frac{1}{2}-b}\|u\|_{X^{s,b}}.
$$

\item (Inhomogeneous term estimate) 
If $\frac{1}{2}<b\leq 1$, then
$$
\bigg\|
\eta_T(t) \int_0^tS(t-t_1)F(t_1)dt_1
\bigg
\|_{X^{s,b}} \ls T^{\frac{1}{2}-b}\|F\|_{X^{s,b-1}}.
$$
\end{enumerate}
\end{lemma}

In the contraction mapping argument below, the $X^{s,b}$-norm is used to propagate $H^s$-regularity, whereas contraction is established in the weaker $X^{0,b}$-metric. The following lemma shows that the resulting metric space is complete, as required for the application of the Banach fixed point theorem.

\begin{lemma}[Completeness of the fixed-point space]\label{lem: completeness of the fixed-point space}
Let $s\geq 0$, $b\in\R$ and $M_0, M_s\geq 0$. Set
$$
\mathcal{X}_h:=
\Big\{u_h\in X_{h,\pm}^{s,b}: \|u_h\|_{X_{h,\pm}^{0,b}}\leq M_0,\quad \|u_h\|_{X_{h,\pm}^{s,b}}\leq M_s\Big\}
$$
and endow $\mathcal{X}_h$ with the metric
$$
d_h(u_h, v_h):=\|u_h-v_h\|_{X_{h,\pm}^{0,b}}.
$$
Then $(\mathcal{X}_h, d_h)$ is complete. The same statement holds with $X_{h,\pm}^{s,b}$ replaced by $X_\pm^{s,b}$.
\end{lemma}
Lemma~\ref{lem: completeness of the fixed-point space} follows from the following weak lower-semicontinuity property; see \cite[Theorem~1.2.5]{Caz2003} and for a similar application, see \cite[Proof of Proposition~3.2]{Soh2011-1}.
\begin{proposition}
Let $X\hookrightarrow Y$ be two Banach spaces, and let $1<p,q\leq \infty$. Let $I\subseteq\R$ be an open interval, possibly $I=\R$. Suppose that $\{f_n\}_{n\geq 0}$ is bounded in $L^q(I;Y)$, and let $f:I\to Y$ be such that $f_n(t)\rightharpoonup f(t)$ in $Y$ as $n\to\infty$, for almost every $t\in I$. If $\{f_n\}_{n\geq 0}$ is bounded in $L^p(I;X)$ and $X$ is reflexive, then $f\in L^p(I;X)$ and
$$
\|f\|_{L^p(I;X)}\leq\liminf_{n\to\infty}\|f_n\|_{L^p(I;X)}.
$$
\end{proposition}

\section{Bilinear estimates for the FPUT system}\label{sec: bilinear estimates}

In this section, we establish the bilinear estimates (Lemma \ref{lem: FPUT bilinear estimates} and Lemma \ref{lem: bilinear estimates for translation correction}), which are needed to control the quadratic nonlinear terms
\[\nabla_h
\Big\{
(u_h^\pm)^2+\big(e^{\pm\frac{2t}{h^2}\pa_h}u_h^\mp\big)^2+2u_h^\pm\cdot e^{\pm\frac{2t}{h^2}\pa_h}u_h^\mp+\mathcal{T}_h^\pm[r_h]\Big\}\]
in the coupled FPUT system \eqref{eq: coupled FPUT integral form}. In Fourier variables, the
lattice derivative \(\nabla_h\) has the multiplier
\(\frac{2i}{h}\sin\big(\frac{h[\xi]}{2}\big)\). The estimates below exhibit a \textit{gain of one lattice derivative} in the sense that the frequency and modulation weights in the Bourgain norms control this multiplier uniformly in \(h\), even though it can be as large as \(O(h^{-1})\). This gain is the lattice analogue of the KdV \textit{bilinear smoothing} in Kenig, Ponce, and Vega~\cite{KPV1996}, and was observed in the FPUT setting by Hong, Kwak, and Yang~\cite{HKY2021}. The second and third estimates further yield an explicit positive power of \(h\) from the regularity gap \(s>s'\).

\begin{lemma}[Bilinear estimates for FPUT]
\label{lem: FPUT bilinear estimates}
Let \(h\in(0,1]\), \(s\geq0\), \(\delta\in(0,\frac14)\), and
\(b\in(\frac12,\frac34-\delta)\). Then,  the
following estimates hold uniformly in \(h\in(0,1]\).
\begin{enumerate}[\((i)\)]
\item \textup{(Bilinear estimate)}
\begin{equation}\label{ineq: FPUT bilinear estimate 1}
\|\nabla_h(u_hv_h)\|_{X_{h,\pm}^{s,-(1-b-\delta)}}
\lesssim
\|u_h\|_{X_{h,\pm}^{s,b}}
\|v_h\|_{X_{h,\pm}^{0,b}}
+
\|u_h\|_{X_{h,\pm}^{0,b}}
\|v_h\|_{X_{h,\pm}^{s,b}}.
\end{equation}

\item \textup{(Bilinear estimates for fast-moving waves)}
For every \(\max\{s-1,0\}\leq s'\leq s\), we have
\begin{equation}\label{ineq: FPUT bilinear estimate 2}
\begin{aligned}
&\big\|
\nabla_h\big(
e^{\pm\frac{2t}{h^2}\pa_h}u_h
\cdot
e^{\pm\frac{2t}{h^2}\pa_h}v_h
\big)
\big\|_{X_{h,\pm}^{s',-(1-b-\delta)}}\\
&\lesssim
h^{\min\{s-s',\frac12\}}
\Big(
\|u_h\|_{X_{h,\mp}^{s,b}}
\|v_h\|_{X_{h,\mp}^{0,b}}
+
\|u_h\|_{X_{h,\mp}^{0,b}}
\|v_h\|_{X_{h,\mp}^{s,b}}
\Big)
\end{aligned}
\end{equation}
and
\begin{equation}\label{ineq: FPUT bilinear estimate 3}
\begin{aligned}
&\big\|
\nabla_h\big(
u_h\cdot e^{\pm\frac{2t}{h^2}\pa_h}v_h
\big)
\big\|_{X_{h,\pm}^{s',-(1-b-\delta)}}\\
&\lesssim
h^{\min\{s-s',\frac12\}}
\Big(
\|u_h\|_{X_{h,\pm}^{s,b}}
\|v_h\|_{X_{h,\mp}^{0,b}}
+
\|u_h\|_{X_{h,\pm}^{0,b}}
\|v_h\|_{X_{h,\mp}^{s,b}}
\Big).
\end{aligned}
\end{equation}
\end{enumerate}
\end{lemma}

\begin{remark}
Estimates of the type
\eqref{ineq: FPUT bilinear estimate 1},
\eqref{ineq: FPUT bilinear estimate 2}, and
\eqref{ineq: FPUT bilinear estimate 3}
were established in
\cite[Lemmas~6.1, 6.3, and 6.4]{HKY2021}, respectively.
Compared with the formulations stated there, the estimates above place the higher spatial regularity on only one input factor while measuring the other at the \(X^{0,b}\) level. This refinement is simple but crucial to our analysis; it allows us to close the \(H^s\) estimates in Proposition~\ref{prop: Local uniform bounds for the FPUT system} on a lifespan determined solely by the \(L^2\)-norm of the initial data. It is therefore a key ingredient in establishing persistence of regularity.
\end{remark}

\begin{remark}
The proof of Lemma~\ref{lem: FPUT bilinear estimates} follows the method of Kenig, Ponce, and Vega~\cite{KPV1996}, which is standard in the low-regularity theory of dispersive PDEs. The key step is to reduce the bilinear estimates to uniform bounds for one-dimensional multiplier integrals, as in Lemma~\ref{lem: unified integral bounds}. Although the underlying analysis is similar to that in \cite{HKY2021}, we streamline the presentation by treating the three estimates simultaneously and reducing their common calculations to the unified integral estimate in Lemma~\ref{lem: unified integral bounds}. We also carefully account for the frequency-wrapping interactions.
\end{remark}

We first isolate the frequency-wrapping contributions to the bilinear terms in Lemma~\ref{lem: FPUT bilinear estimates} and establish the corresponding estimates. To formulate them, we introduce the frequency projections
\(P_+\) and \(P_-\) defined by
\[
\widehat{P_\pm f_h}(\xi)
=
\mathbbm{1}_{\mathbb{T}_h^\pm}(\xi)\widehat{f_h}(\xi),
\]
where \(\mathbb{T}_h^+\) and \(\mathbb{T}_h^-\) denote, respectively, the nonnegative and negative frequency regions of \(\mathbb{T}_h\).
Under the identification of \(\mathbb{T}_h\) with the periodic interval
\([-\frac{\pi}{h},\frac{\pi}{h})\), these regions are represented by
\[
\mathbb{T}_h^-
\simeq
\left[-\frac{\pi}{h},0\right),
\qquad
\mathbb{T}_h^+
\simeq
\left[0,\frac{\pi}{h}\right).
\]
Here, \(\simeq\) denotes this periodic identification rather than literal equality; the left-hand sides are subsets of the quotient \(\mathbb{T}_h\), whereas the right-hand sides are their representatives in the chosen half-open interval. Frequency-wrapping in the product \(u_hv_h\) can occur only when the two input frequencies have the same sign and their sum lies outside this interval. Consequently, the corresponding wrapped components are
\[
P_-\big\{(P_+u_h)(P_+v_h)\big\}
\qquad\text{and}\qquad
P_+\big\{(P_-u_h)(P_-v_h)\big\}.
\]

The following lemma gives the estimates for the frequency-wrapping components of
\eqref{ineq: FPUT bilinear estimate 1},
\eqref{ineq: FPUT bilinear estimate 2}, and
\eqref{ineq: FPUT bilinear estimate 3}. For use in a later application, we include an additional time-oscillatory factor 
\(e^{\frac{2\pi it\ell}{h^3}}\).

\begin{lemma}[Frequency-wrapping bilinear estimates for FPUT]
\label{lem: bilinear estimates for translation correction}
Let \(h\in(0,1]\), \(s\geq0\),
\(\delta\in(0,\frac14)\), and
\(b\in(\frac12,\frac34-\delta)\). Then, for 
\(\ell\in\{-1,0,1\}\) and $\max\{s-1,0\}\leq s'\leq s$, the following estimates hold:
\begin{enumerate}[\((i)\)]
\item (Frequency wrapping component of \eqref{ineq: FPUT bilinear estimate 1})
\begin{equation}\label{ineq: translation correction bilinear estimate 1}
\begin{aligned}
&\Big\|
\nabla_h
e^{\frac{2\pi it\ell}{h^3}}P_-
\big\{
(P_+u_h)(P_+v_h)
\big\}
\Big\|_{X_{h,\pm}^{s',-(1-b-\delta)}}\\
&\lesssim
h^{s-s'}
\Big(
\|u_h\|_{X_{h,\pm}^{s,b}}
\|v_h\|_{X_{h,\pm}^{0,b}}
+
\|u_h\|_{X_{h,\pm}^{0,b}}
\|v_h\|_{X_{h,\pm}^{s,b}}
\Big).
\end{aligned}
\end{equation}
\item (Frequency wrapping component of \eqref{ineq: FPUT bilinear estimate 2})
\begin{equation}\label{ineq: translation correction bilinear estimate 2}
\begin{aligned}
&\Big\|
\nabla_h
e^{\frac{2\pi it\ell}{h^3}}P_-
\big\{
\big(P_+e^{\pm\frac{2t}{h^2}\pa_h}u_h\big)
\big(P_+e^{\pm\frac{2t}{h^2}\pa_h}v_h\big)
\big\}
\Big\|_{X_{h,\pm}^{s',-(1-b-\delta)}}\\
&\lesssim
h^{s-s'}
\Big(
\|u_h\|_{X_{h,\mp}^{s,b}}
\|v_h\|_{X_{h,\mp}^{0,b}}
+
\|u_h\|_{X_{h,\mp}^{0,b}}
\|v_h\|_{X_{h,\mp}^{s,b}}
\Big).
\end{aligned}
\end{equation}
\item (Frequency wrapping component of \eqref{ineq: FPUT bilinear estimate 3})
\begin{equation}\label{ineq: translation correction bilinear estimate 3}
\begin{aligned}
&\Big\|
\nabla_h
e^{\frac{2\pi it\ell}{h^3}}P_-
\big\{
\big(P_+u_h\big)
\big(P_+e^{\pm\frac{2t}{h^2}\pa_h}v_h\big)
\big\}
\Big\|_{X_{h,\pm}^{s',-(1-b-\delta)}}\\
&\lesssim
h^{s-s'}
\Big(
\|u_h\|_{X_{h,\pm}^{s,b}}
\|v_h\|_{X_{h,\mp}^{0,b}}
+
\|u_h\|_{X_{h,\pm}^{0,b}}
\|v_h\|_{X_{h,\mp}^{s,b}}
\Big).
\end{aligned}
\end{equation}
\end{enumerate}
The corresponding estimates also hold with \(P_+\) and \(P_-\) interchanged throughout.
\end{lemma}

The proofs of Lemmas~\ref{lem: FPUT bilinear estimates} and~\ref{lem: bilinear estimates for translation correction} rely on the following key integral estimates, which are proved in Appendix~\ref{appendix: Proof of integral estimates}.

\begin{lemma}[Key integral estimates]\label{lem: unified integral bounds}
Let \(h\in(0,1]\), \(\tau,\xi\in\mathbb R\), \(b>\frac12\),
\(\ell\in\{-1,0,1\}\), and
\(\sigma_1,\sigma_2\in\{-1,1\}\). Define
\begin{equation}\label{eq: definition of mathscr I}
\mathscr{I}_{\tau,\xi,\ell}^{h;\sigma_1,\sigma_2}
:=
\frac{\frac{1}{h^2}\sin^2(\frac{h[\xi]}{2})}
{\la\tau-s_h([\xi])\ra^{\frac12}}
\mathcal I_{\tau,\xi,\ell}^{h;\sigma_1,\sigma_2},
\end{equation}
where
\begin{equation}\label{eq: definition of mathcal I}
\begin{aligned}
\mathcal I_{\tau,\xi,\ell}^{h;\sigma_1,\sigma_2}
:=
\int_{\mathbb{R}}
\frac{\mathbbm{1}_{\{-\frac{\pi}{h}\leq \xi_1<\frac{\pi}{h}\}\cap\{-\frac{\pi}{h}\leq\xi-\xi_1<\frac{\pi}{h}\}}}
{\big\la
\tau
-\frac{(1+\sigma_1)\xi_1
+(1+\sigma_2)(\xi-\xi_1)}{h^2}
-\frac{2\pi\ell}{h^3}
+\sigma_1s_h(\xi_1)
+\sigma_2s_h(\xi-\xi_1)
\big\ra^{2b}}d\xi_1.
\end{aligned}
\end{equation}
Then the following estimates hold.

\begin{enumerate}[$(i)$]
\item If
\(\xi\in[-\frac{\pi}{h},\frac{\pi}{h})\), then
\begin{equation}\label{ineq: unified integral bound-1}
\mathscr I_{\tau,\xi,0}^{h;-1,-1}
\lesssim1,
\end{equation}
whereas for
\((\sigma_1,\sigma_2)\neq(-1,-1)\),
\begin{equation}\label{ineq: unified integral bounds 2 and 3}
\mathscr I_{\tau,\xi,0}^{h;\sigma_1,\sigma_2}
\lesssim
\left|\sin\left(\frac{h\xi}{4}\right)\right|.
\end{equation}

\item If
\(\xi\in\mathbb R\setminus
[-\frac{\pi}{h},\frac{\pi}{h})\), then
\begin{equation}\label{ineq: unified integral bound-1'}
\mathscr I_{\tau,\xi,\ell}^{h;\sigma_1,\sigma_2}
\lesssim1,
\end{equation}
for every
\(\ell\in\{-1,0,1\}\) and
\((\sigma_1,\sigma_2)\in\{-1,1\}^2\).
\end{enumerate}
\end{lemma}

We also recall the following elementary integral inequality.

\begin{lemma}
Let \(\beta,\gamma\in\mathbb{R}\) and
\(b'\geq b>\frac12\). Then,
\begin{equation}\label{ineq: elementary integral estimates}
\int_{\mathbb{R}}
\frac{dx}
{\la x-\beta\ra^{2b'}\la x-\gamma\ra^{2b}}
\lesssim
\frac{1}{\la\beta-\gamma\ra^{2b}}.
\end{equation}
\end{lemma}

We now prove the bilinear estimates.

\begin{proof}[Proof of Lemma~\ref{lem: bilinear estimates for translation correction}]
\noindent\textbf{\underline{Step 1. Preliminary reduction.}} By the time-reversal symmetry \(t\mapsto-t\), it suffices to prove the
estimates with the upper signs. Indeed, by replacing \(u_h(t,z)\) and
\(v_h(t,z)\) by \(u_h(-t,z)\) and \(v_h(-t,z)\), respectively, 
their space-time Fourier transforms are sent to
\(\widetilde{u_h}(-\tau,\xi)\) and
\(\widetilde{v_h}(-\tau,\xi)\). Consequently, this transformation
interchanges \(X_{h,+}^{s,b}\) and \(X_{h,-}^{s,b}\), replaces
\(\ell\) with \(-\ell\) in
\(e^{\frac{2\pi it\ell}{h^3}}\), and reverses the sign in
\(e^{\pm\frac{2t}{h^2}\pa_h}\). Thus,
every lower-sign estimate follows from its upper-sign counterpart. The estimates with \(P_+\) and \(P_-\) interchanged follow similarly.
Indeed, for \(0\leq\xi<\frac{\pi}{h}\), the unfolded output frequency is
\(\xi-\frac{2\pi}{h}\), and on the corresponding wrapping region, \([\xi-\xi_1]=\left(\xi-\frac{2\pi}{h}\right)-\xi_1\). Since \([\xi-\frac{2\pi}{h}]=\xi\) and
\(\xi-\frac{2\pi}{h}\notin[-\frac{\pi}{h},\frac{\pi}{h})\),
the same calculation using \eqref{ineq: unified integral bound-1'}
yields the desired estimates.

To prove the three estimates simultaneously, we introduce
\[
(\sigma_1,\sigma_2)
\in
\{(-1,-1),(1,1),(-1,1)\},
\]
corresponding respectively to
\eqref{ineq: translation correction bilinear estimate 1},
\eqref{ineq: translation correction bilinear estimate 2}, and
\eqref{ineq: translation correction bilinear estimate 3}. We then
write
\[
\begin{aligned}
&\mathcal F_h
\Big[
P_-
\big\{
(P_+e^{\frac{(1+\sigma_1)t}{h^2}\pa_h}u_h)
(P_+e^{\frac{(1+\sigma_2)t}{h^2}\pa_h}v_h)
\big\}
\Big](t,\xi)\\
&=
\frac{\mathbbm 1_{[-\frac{\pi}{h},0)}(\xi)}{2\pi}
\int_{\xi+\frac{\pi}{h}}^{\frac{\pi}{h}}
e^{\frac{it}{h^2}\left\{
(1+\sigma_1)\xi_1
+(1+\sigma_2)[\xi-\xi_1]
\right\}}
\widehat{u_h}(t,\xi_1)
\widehat{v_h}(t,[\xi-\xi_1])d\xi_1.
\end{aligned}
\]
Here, the interval of integration is given by $(\xi+\frac{\pi}{h}, \frac{\pi}{h})$; for each
\(-\frac{\pi}{h}\leq\xi<0\), the conditions \(0\leq\xi_1<\frac{\pi}{h}\) and \(0\leq[\xi-\xi_1]<\frac{\pi}{h}\) are equivalent, up to endpoints of measure zero, to
\(\xi+\frac{\pi}{h}<\xi_1<\frac{\pi}{h}\), because 
\begin{equation}\label{eq: frequency wrapping xi-xi1}
[\xi-\xi_1]
=
\xi-\xi_1+\frac{2\pi}{h}.
\end{equation}
Hence, it follows that 
\begin{equation}\label{eq: space-time Fourier transform of frequency wrapping interaction}
\begin{aligned}
&\mathcal F_{t,h}
\Big[
e^{\frac{2\pi it\ell}{h^3}}P_-
\big\{
(P_+e^{\frac{(1+\sigma_1)t}{h^2}\pa_h}u_h)
(P_+e^{\frac{(1+\sigma_2)t}{h^2}\pa_h}v_h)
\big\}
\Big](\tau,\xi)\\
&=
\frac{\mathbbm 1_{[-\frac{\pi}{h},0)}(\xi)}{(2\pi)^2}
\int_{\mathbb R}
\int_{\xi+\frac{\pi}{h}}^{\frac{\pi}{h}}
\widetilde{u_h}(\tau_1,\xi_1)\widetilde{v_h}\big(
\tau-\tau_1-(\cdots),
[\xi-\xi_1]
\big)
d\xi_1d\tau_1
=:\mathscr{A}(\tau,\xi),
\end{aligned}
\end{equation}
where for brevity, we write
\[(\cdots)=\frac{(1+\sigma_1)\xi_1
 +(1+\sigma_2)[\xi-\xi_1]}{h^2}
+\frac{2\pi\ell}{h^3}.\]

Taking absolute values inside the integral in \eqref{eq: space-time Fourier transform of frequency wrapping interaction}, we see that \(|\mathscr{A}(\tau,\xi)|\) is bounded by the same expression with \(\widetilde{u_h}\) and \(\widetilde{v_h}\) replaced by \(\lvert\widetilde{u_h}\rvert\) and \(\lvert\widetilde{v_h}\rvert\), respectively. Since this replacement preserves the corresponding Bourgain norms, we may assume that \(\widetilde{u_h},\widetilde{v_h}\geq0\) almost everywhere. In addition, the triangle inequality
\eqref{ineq: triangle inequality on Th for wrapped frequency} gives
\[
\la\xi\ra^{s'}
\lesssim
h^{s-s'}
\big(
\la\xi_1\ra^{s}
+
\la
\xi-\xi_1+\tfrac{2\pi}{h}
\ra^{s}
\big)
\]
for \(0\leq s'\leq s\leq s'+1\). We insert this inequality into \eqref{eq: space-time Fourier transform of frequency wrapping interaction} and split the resulting expression into two terms. The first frequency weight is absorbed into the corresponding \(X^{s,b}\)-norm of \(u_h\), and the second into that of \(v_h\). It therefore suffices to prove the underlying estimate with \(s=s'=0\), since the factor \(h^{s-s'}\) is retained outside.

\noindent\textbf{\underline{Step 2. Reduction to an integral estimate.}} Fix \(-\frac{\pi}{h}\leq\xi<0\). We now apply the Cauchy--Schwarz inequality with respect to
\((\tau_1,\xi_1)\) to the double integral on the right-hand side of
\eqref{eq: space-time Fourier transform of frequency wrapping interaction}.
Its absolute value squared is bounded by
\[
\begin{aligned}
|\mathscr{A}(\tau,\xi)|^2&\lesssim\Bigg\{
\int_{\mathbb R}
\int_{\xi+\frac{\pi}{h}}^{\frac{\pi}{h}}
\frac{d\xi_1d\tau_1}
{\la\tau_1+\sigma_1s_h(\xi_1)\ra^{2b}
 \big\la
 \tau-\tau_1-(\cdots)
 +\sigma_2s_h([\xi-\xi_1])
 \big\ra^{2b}}
\Bigg\}\\
&\quad\times
\Bigg\{
\int_{\mathbb R}
\int_{\xi+\frac{\pi}{h}}^{\frac{\pi}{h}}
\la\tau_1+\sigma_1s_h(\xi_1)\ra^{2b}
|\widetilde{u_h}(\tau_1,\xi_1)|^2\\
&\qquad\times
\big\la
\tau-\tau_1
-(\cdots)
+\sigma_2s_h([\xi-\xi_1])
\big\ra^{2b}
\big|
\widetilde{v_h}(
\tau-\tau_1
-(\cdots),
[\xi-\xi_1])
\big|^2
d\xi_1 d\tau_1
\Bigg\}.
\end{aligned}
\]
For the first factor on the right-hand side, expanding the abbreviation \((\cdots)\) and applying the integral estimate \eqref{ineq: elementary integral estimates} to the
\(\tau_1\)-integration, we obtain
\[
\begin{aligned}
&\int_{\mathbb R}
\int_{\xi+\frac{\pi}{h}}^{\frac{\pi}{h}}
\frac{d\xi_1d\tau_1}
{\la\tau_1+\sigma_1s_h(\xi_1)\ra^{2b}
 \big\la
 \tau-\tau_1
 -(\cdots)
 +\sigma_2s_h([\xi-\xi_1])
 \big\ra^{2b}}\\
&\lesssim
\int_{\xi+\frac{\pi}{h}}^{\frac{\pi}{h}}
\frac{d\xi_1}
{\Big\la
\tau
-\frac{(1+\sigma_1)\xi_1
 +(1+\sigma_2)[\xi-\xi_1]}{h^2}
-\frac{2\pi\ell}{h^3}
+\sigma_1s_h(\xi_1)
+\sigma_2s_h([\xi-\xi_1])
\Big\ra^{2b}}=
\mathcal I_{\tau,\xi+\frac{2\pi}{h},\ell}
^{h;\sigma_1,\sigma_2},
\end{aligned}
\]
where the last identity follows from the definition (see \eqref{eq: definition of mathcal I}) and \([\xi-\xi_1]
=
\left(\xi+\frac{2\pi}{h}\right)-\xi_1\) (see \eqref{eq: frequency wrapping xi-xi1}). Thus, by Fubini's theorem and by extending the interval of integration to the whole lattice Fourier domain, we obtain
\[
\begin{aligned}
&\Big\|
\nabla_he^{\frac{2\pi it\ell}{h^3}}P_-
\big\{
(P_+e^{\frac{(1+\sigma_1)t}{h^2}\pa_h}u_h)
(P_+e^{\frac{(1+\sigma_2)t}{h^2}\pa_h}v_h)
\big\}
\Big\|_{X_{h,+}^{0,-(1-b-\delta)}}^2\\
&=\bigg\|\frac{\frac{2}{h}\sin\big(\frac{h\xi}{2}\big)}
{\la\tau-s_h(\xi)\ra^{(1-b-\delta)}}\mathscr{A}(\tau,\xi)\bigg\|_{L_{\tau,\xi}^2}^2\\
&\lesssim
\sup_{\substack{\tau\in\mathbb R\\
-\frac{\pi}{h}\leq\xi<0}}
\underbrace{\frac{\frac{1}{h^2}\sin^2\big(\frac{h\xi}{2}\big)}
{\la\tau-s_h(\xi)\ra^{\frac{1}{2}}}
\mathcal I_{\tau,\xi+\frac{2\pi}{h},\ell}
^{h;\sigma_1,\sigma_2}}_{=(*)}\times
\begin{cases}
\|u_h\|_{X_{h,+}^{0,b}}^2
\|v_h\|_{X_{h,+}^{0,b}}^2
&\textup{if }(\sigma_1,\sigma_2)=(-1,-1),\\[2mm]
\|u_h\|_{X_{h,-}^{0,b}}^2
\|v_h\|_{X_{h,-}^{0,b}}^2
&\textup{if }(\sigma_1,\sigma_2)=(1,1),\\[2mm]
\|u_h\|_{X_{h,+}^{0,b}}^2
\|v_h\|_{X_{h,-}^{0,b}}^2
&\textup{if }(\sigma_1,\sigma_2)=(-1,1),
\end{cases}
\end{aligned}
\]
where \(2(1-b-\delta)\geq\frac{1}{2}\) is also used in the last step. Therefore, the proof of the lemma reduces to showing that \((*)\) is bounded uniformly in \(\tau\in\mathbb{R}\) and \(-\frac{\pi}{h}\leq\xi<0\).

\noindent\textbf{\underline{Step 3. Integral estimate.}} Suppose that \(-\frac{\pi}{h}\leq\xi<0\). Then, since \(\xi+\frac{2\pi}{h}\notin[-\frac{\pi}{h},\frac{\pi}{h})\) and \([\xi+\frac{2\pi}{h}]=\xi\), it follows from \eqref{ineq: unified integral bound-1'} that 
\[
(*)=
\frac{
\frac{1}{h^2}
\sin^2\Big(
\frac{h[\xi+\frac{2\pi}{h}]}{2}
\Big)}
{\la
\tau-s_h\left([\xi+\frac{2\pi}{h}]\right)
\ra^{\frac12}}
\mathcal I_{\tau,\xi+\frac{2\pi}{h},\ell}
^{h;\sigma_1,\sigma_2}=
\mathscr I_{\tau,\xi+\frac{2\pi}{h},\ell}
^{h;\sigma_1,\sigma_2}
\lesssim1,
\]
which completes the proof.
\end{proof}

The proof of Lemma~\ref{lem: FPUT bilinear estimates} is similar to that of the previous lemma. Therefore, we sketch the common steps and explain the steps that differ in detail.

\begin{proof}[Proof of Lemma~\ref{lem: FPUT bilinear estimates}]
Again, by time-reversal symmetry, it suffices to prove the estimates with the upper signs. We take the space-time Fourier transforms of the products, and  decompose each convolution into its non-wrapping and frequency-wrapping parts. The latter are controlled by Lemma~\ref{lem: bilinear estimates for translation correction} with \(\ell=0\). More precisely, \eqref{ineq: translation correction bilinear estimate 1} with \(s=s'\), \eqref{ineq: translation correction bilinear estimate 2} and \eqref{ineq: translation correction bilinear estimate 3} correspond to the frequency wrapping contributions of \eqref{ineq: FPUT bilinear estimate 1}, \eqref{ineq: FPUT bilinear estimate 2} and \eqref{ineq: FPUT bilinear estimate 3}, respectively. Therefore, it remains to estimate the non-wrapping contributions.

To treat the three estimates simultaneously, we consider
\[
(\sigma_1,\sigma_2)
\in
\{(-1,-1),(1,1),(-1,1)\},
\]
corresponding respectively to
\eqref{ineq: FPUT bilinear estimate 1},
\eqref{ineq: FPUT bilinear estimate 2} and
\eqref{ineq: FPUT bilinear estimate 3}. On the non-wrapping region, we have \([\xi-\xi_1]=\xi-\xi_1\) and thus, the relevant part of the space-time Fourier transform is
\[
\begin{aligned}
&\frac{\frac{2i}{h}
\sin\big(\frac{h\xi}{2}\big)}{(2\pi)^2}
\int_{\mathbb R}
\int_{\substack{
-\frac{\pi}{h}\leq\xi_1<\frac{\pi}{h}\\
-\frac{\pi}{h}\leq\xi-\xi_1<\frac{\pi}{h}
}}
\widetilde{u_h}(\tau_1,\xi_1)\\
&\qquad\qquad\qquad\times
\widetilde{v_h}\left(
\tau-\tau_1
-\frac{(1+\sigma_1)\xi_1
 +(1+\sigma_2)(\xi-\xi_1)}{h^2},
\xi-\xi_1
\right)
d\xi_1d\tau_1.
\end{aligned}
\]
Here, an input corresponding to \(\sigma_j=-1\) is measured in \(X_{h,+}^{s,b}\), whereas one corresponding to \(\sigma_j=1\) is measured in \(X_{h,-}^{s,b}\). Moreover, by the argument in the proof of Lemma~\ref{lem: bilinear estimates for translation correction}, we may assume that \(\widetilde{u_h}, \widetilde{v_h}\geq0\).

First, we prove \eqref{ineq: FPUT bilinear estimate 1}, corresponding to \((\sigma_1,\sigma_2)=(-1,-1)\). Note that on the non-wrapping region, we have 
\begin{equation}\label{eq: very basic triangle-type  inequality for frequencies}
\la\xi\ra^{s}
\lesssim
\la\xi_1\ra^{s}+\la\xi-\xi_1\ra^{s}.
\end{equation}
We distribute the spatial weight between the two input factors.
Then, by Plancherel's theorem, the Cauchy--Schwarz inequality in
\((\tau_1,\xi_1)\), and
\eqref{ineq: elementary integral estimates}, exactly as in Step 2 of the proof of Lemma~\ref{lem: bilinear estimates for translation correction}, the proof reduces to showing
\[
\begin{aligned}
&\sup_{\substack{\tau\in\mathbb R\\
\xi\in[-\frac{\pi}{h},\frac{\pi}{h})}}
\frac{\frac{1}{h^2}\sin^2(\frac{h\xi}{2})}
{\la\tau-s_h(\xi)\ra^{\frac12}}
\mathcal I_{\tau,\xi,0}^{h;-1,-1}=
\sup_{\substack{\tau\in\mathbb R\\
\xi\in[-\frac{\pi}{h},\frac{\pi}{h})}}
\mathscr I_{\tau,\xi,0}^{h;-1,-1}
\lesssim1.
\end{aligned}
\]
Indeed, this bound follows immediately from
\eqref{ineq: unified integral bound-1}.

Next, we assume \((\sigma_1,\sigma_2)=(\pm1,1)\), and prove \eqref{ineq: FPUT bilinear estimate 2} and \eqref{ineq: FPUT bilinear estimate 3}. In this case, the only difference from the proof of \eqref{ineq: FPUT bilinear estimate 1} is that we use 
\[
\la\xi\ra^{s'}=\frac{\la\xi\ra^{s}}{\la\xi\ra^{s-s'}}
\lesssim
\frac{\la\xi_1\ra^{s}
+
\la\xi-\xi_1\ra^{s}}{\la\xi\ra^{s-s'}},
\]
instead of \eqref{eq: very basic triangle-type  inequality for frequencies}. Repeating the argument in Step 2 of the proof of Lemma~\ref{lem: bilinear estimates for translation correction}, but with the above bound, reduces the proof to showing
\[
\sup_{\substack{\tau\in\mathbb R\\
\xi\in[-\frac{\pi}{h},\frac{\pi}{h})}}
\frac{\frac{1}{h^2}\sin^2(\frac{h\xi}{2})}
{\la\xi\ra^{2(s-s')}
 \la\tau-s_h(\xi)\ra^{\frac12}}
\mathcal I_{\tau,\xi,0}^{h;\sigma_1,\sigma_2}=
\sup_{\substack{\tau\in\mathbb R\\
\xi\in[-\frac{\pi}{h},\frac{\pi}{h})}}
\frac{1}
{\la\xi\ra^{2(s-s')}}\mathscr I_{\tau,\xi,0}^{h;\sigma_1,\sigma_2}
\lesssim
h^{2\min\{s-s',\frac12\}}.
\]
Indeed, by
\eqref{ineq: unified integral bounds 2 and 3}, we have 
\[
\begin{aligned}
\frac{1}
{\la\xi\ra^{2(s-s')}}\mathscr I_{\tau,\xi,0}^{h;\sigma_1,\sigma_2}
&\lesssim
\frac{1}
{\la\xi\ra^{2(s-s')}}\left|\sin(\frac{h\xi}{4})\right|
\lesssim
\frac{1}
{\la\xi\ra^{2(s-s')}}\min\big\{h|\xi|,1\big\}\lesssim
h^{2\min\{s-s',\frac12\}},
\end{aligned}
\]
since $-\frac{\pi}{h}\leq \xi<\frac{\pi}{h}$.
\end{proof}

As an application of Lemma \ref{lem: bilinear estimates for translation correction}, we show bounds for the translation correction 
\[\mathcal{T}_h^\pm[r_h]
=
e^{\pm\frac{t}{h^2}\pa_h}(r_h^2)
-\big(e^{\pm\frac{t}{h^2}\pa_h}r_h\big)^2.\]

\begin{lemma}[Estimates for the translation correction]\label{lem: estimates for the translation correction}
Let $0 \leq s' \leq s\leq1$, $0<\delta<\frac{1}{4}$ and $\frac{1}{2}<b<\frac{3}{4}-\delta$. For each $j=1,2$, let $u_{h,j}^\pm:\R\times h\Z\to \R$ and define the corresponding strain by
$$
r_{h,j}(t,z)
:=
e^{-\frac{t}{h^2}\pa_h}u_{h,j}^+(t,z)
+e^{\frac{t}{h^2}\pa_h}u_{h,j}^-(t,z),
\quad z\in h\Z.
$$
Suppose that
$$
\|u_{h,j}^\pm\|_{X_{h,\pm}^{0,b}}\leq R_0
\quad\textup{and}\quad
\|u_{h,j}^\pm\|_{X_{h,\pm}^{s,b}}\leq R_s.
$$
Then, uniformly for $h\in(0,1]$ and $j=1,2$, one has
\begin{equation}\label{ineq: translation correction estimate 1}
\big\|\nabla_h\mathcal{T}_h^\pm[r_{h,j}]\big\|_{X_{h,\pm}^{s',-(1-b-\delta)}}
\ls  h^{s-s'}R_0R_s
\end{equation}
and
\begin{equation}\label{ineq: translation correction estimate 2}
\big\|\nabla_h\big(\mathcal{T}_h^\pm[r_{h,1}]-\mathcal{T}_h^\pm[r_{h,2}]\big)\big\|_{X_{h,\pm}^{0,-(1-b-\delta)}}
\ls R_0\sum_\pm\|u_{h,1}^\pm-u_{h,2}^\pm\|_{X_{h,\pm}^{0,b}}.
\end{equation}
\end{lemma}

\begin{proof}
For \(a\in\mathbb{R}\), we claim that
\begin{equation}\label{eq: general translation correction}
\begin{aligned}
&e^{a\pa_h}(f_hg_h)
-(e^{a\pa_h}f_h)(e^{a\pa_h}g_h)\\
&=
\big(e^{-\frac{2\pi ia}{h}}-1\big)
P_-
\big\{
(e^{a\pa_h}P_+f_h)(e^{a\pa_h}P_+g_h)
\big\}
+
\big(e^{\frac{2\pi ia}{h}}-1\big)
P_+
\big\{
(e^{a\pa_h}P_-f_h)(e^{a\pa_h}P_-g_h)
\big\}.
\end{aligned}
\end{equation}
Indeed, for
\(\xi\in[-\frac{\pi}{h},\frac{\pi}{h})\), we have
\[
\begin{aligned}
&\mathcal{F}_h\big\{
e^{a\pa_h}(f_hg_h)
-(e^{a\pa_h}f_h)(e^{a\pa_h}g_h)
\big\}(\xi)\\
&=
\frac{1}{2\pi}
\int_{-\frac{\pi}{h}}^{\frac{\pi}{h}}
\big(
e^{ia([\xi]-[\xi-\xi_1]-\xi_1)}-1
\big)
\widehat{e^{a\pa_h}f_h}([\xi-\xi_1])
\widehat{e^{a\pa_h}g_h}(\xi_1) d\xi_1.
\end{aligned}
\]
Up to endpoints of measure zero, frequency-wrapping occurs precisely when \([\xi-\xi_1],\xi_1\in\mathbb{T}_h^+\) and \([\xi]\in\mathbb{T}_h^-\) or when \([\xi-\xi_1],\xi_1\in\mathbb{T}_h^-\) and \([\xi]\in\mathbb{T}_h^+\). In the former case, \([\xi]-[\xi-\xi_1]-\xi_1=-\frac{2\pi}{h}\), so that
\[
e^{ia([\xi]-[\xi-\xi_1]-\xi_1)}
=
e^{-\frac{2\pi ia}{h}}.
\]
On the other hand, in the latter case, \([\xi]-[\xi-\xi_1]-\xi_1=\frac{2\pi}{h}\), and hence
\[
e^{ia([\xi]-[\xi-\xi_1]-\xi_1)}
=
e^{\frac{2\pi ia}{h}}.
\]
For non-wrapping interactions, \([\xi]-[\xi-\xi_1]-\xi_1=0\), and therefore the corresponding exponential factor equals \(1\).
This proves \eqref{eq: general translation correction}.

Applying \eqref{eq: general translation correction} pointwise in time with \(a=\pm\frac{t}{h^2}\) and \(f_h=g_h=r_{h,j}(t)\), and noting that
\[
e^{\pm\frac{t}{h^2}\pa_h}r_{h,j}
=
u_{h,j}^\pm
+
e^{\pm\frac{2t}{h^2}\pa_h}u_{h,j}^\mp,
\]
we obtain
\[
\begin{aligned}
\mathcal{T}_h^\pm[r_{h,j}]
&=
e^{\mp\frac{2\pi it}{h^3}}
P_-
\Big\{
\big(P_+
u_{h,j}^\pm
+
P_+e^{\pm\frac{2t}{h^2}\pa_h}u_{h,j}^\mp
\big)^2
\Big\}+
e^{\pm\frac{2\pi it}{h^3}}
P_+
\Big\{
\big(
P_-
u_{h,j}^\pm
+
P_-e^{\pm\frac{2t}{h^2}\pa_h}u_{h,j}^\mp
\big)^2
\Big\}\\
&\quad-
P_-
\Big\{
\big(
P_+u_{h,j}^\pm
+
P_+e^{\pm\frac{2t}{h^2}\pa_h}u_{h,j}^\mp
\big)^2
\Big\}-
P_+
\Big\{
\big(
P_-
u_{h,j}^\pm
+
P_-e^{\pm\frac{2t}{h^2}\pa_h}u_{h,j}^\mp
\big)^2
\Big\}.
\end{aligned}
\]
We expand
\[
\big(
P_-
u_{h,j}^\pm
+
P_-e^{\pm\frac{2t}{h^2}\pa_h}u_{h,j}^\mp
\big)^2=
(P_-u_{h,j}^\pm)^2
+
(P_-e^{\pm\frac{2t}{h^2}\pa_h}u_{h,j}^\mp)^2
+
2(P_-u_{h,j}^\pm)(
P_-e^{\pm\frac{2t}{h^2}\pa_h}u_{h,j}^\mp),
\]
and similarly for the square involving \(P_+\). We apply
\eqref{ineq: translation correction bilinear estimate 1},
\eqref{ineq: translation correction bilinear estimate 2}, and
\eqref{ineq: translation correction bilinear estimate 3} to the first, second, and third terms, respectively. Consequently, \eqref{ineq: translation correction estimate 1} follows:
\[
\begin{aligned}
\big\|
\nabla_h\mathcal{T}_h^\pm[r_{h,j}]
\big\|_{X_{h,\pm}^{s',-(1-b-\delta)}}
&\lesssim
h^{s-s'}
\Big(
\|u_{h,j}^\pm\|_{X_{h,\pm}^{s,b}}
\|u_{h,j}^\pm\|_{X_{h,\pm}^{0,b}}+
\|u_{h,j}^\mp\|_{X_{h,\mp}^{s,b}}
\|u_{h,j}^\mp\|_{X_{h,\mp}^{0,b}}\\
&\qquad\qquad+
\|u_{h,j}^\pm\|_{X_{h,\pm}^{s,b}}
\|u_{h,j}^\mp\|_{X_{h,\mp}^{0,b}}+
\|u_{h,j}^\pm\|_{X_{h,\pm}^{0,b}}
\|u_{h,j}^\mp\|_{X_{h,\mp}^{s,b}}
\Big)\\
&\lesssim
h^{s-s'}R_0R_s.
\end{aligned}
\]
For the difference estimate, we apply \(A_1^2-A_2^2=(A_1-A_2)(A_1+A_2)\) to each projected square in the preceding representation. Then, by Lemma~\ref{lem: bilinear estimates for translation correction}, we obtain
\[
\begin{aligned}
\big\|
\nabla_h\big(
\mathcal{T}_h^\pm[r_{h,1}]
-
\mathcal{T}_h^\pm[r_{h,2}]
\big)
\big\|_{X_{h,\pm}^{0,-(1-b-\delta)}}&\lesssim
\sum_{j=1}^2
\Big(
\|u_{h,j}^+\|_{X_{h,+}^{0,b}}
+
\|u_{h,j}^-\|_{X_{h,-}^{0,b}}
\Big)
\sum_\pm
\|u_{h,1}^\pm-u_{h,2}^\pm\|_{X_{h,\pm}^{0,b}}\\
&\lesssim
R_0
\sum_\pm
\|u_{h,1}^\pm-u_{h,2}^\pm\|_{X_{h,\pm}^{0,b}}.
\end{aligned}
\]
This proves \eqref{ineq: translation correction estimate 2} and completes the proof.
\end{proof}

\section{Higher-order remainder estimates for general potentials}\label{sec: Estimates for the translation correction and high-order remainder}

Next, we show that the higher-order remainder 
\[\nabla_h\mathcal{R}_h[r_h]=\frac{2}{h^4}
\nabla_h\bigg\{
V'(h^2r_h)
-
h^2r_h
-
\frac{h^4}{2}r_h^2
\bigg\}\]
in the coupled FPUT system \eqref{eq: coupled FPUT integral form} vanishes as \(h\to0\).
\begin{lemma}[Higher-order remainder estimates]\label{lem: estimates for the higher-order remainder}
Let $0\leq s'\leq s\leq 1$, $b>\frac{1}{2}$, $T>0$, and $R_0, R_s>0$. For each $h\in(0,1]$ and $j=1,2$, let $\mathbf{u}_{h,j}=(u_{h,j}^+,u_{h,j}^-):\R\times h\Z\to \R^2$, and define $r_{h,j}$ by
$$
r_{h,j}(t,z)
:=
e^{-\frac{t}{h^2}\pa_h}u_{h,j}^+(t,z)
+e^{\frac{t}{h^2}\pa_h}u_{h,j}^-(t,z).
$$
Suppose that $V$ satisfies \eqref{eq:potential-assumption} and that for $j=1,2$,
$$
\sup_{0<h\leq 1}\|\mathbf{u}_{h,j}\|_{\mathbf{X}_{h}^{0,b}}\leq R_0
\quad\textup{and}\quad
\sup_{0<h\leq 1}\|\mathbf{u}_{h,j}\|_{\mathbf{X}_{h}^{s,b}}\leq R_s.
$$
Then, there exists $h_0=h_0(R_0)\in(0,1]$ such that for every $h\in(0,h_0]$, 
\begin{align}
\big\|\nabla_h
\mathcal{R}_h[r_{h,j}]\big\|_{L_t^2([-T, T];H_z^{s'}(h\Z))}
&\ls
h^{1+s-s'}T^{\frac{1}{6}}R_0^2R_s,\label{ineq: general nonlinearity estimate 1}\\
\Big\|\nabla_h
\Big(\mathcal{R}_h[r_{h,1}]-\mathcal{R}_h[r_{h,2}]\Big)\Big\|_{L_t^2([-T, T];L_z^2(h\Z))}
&\ls
h
T^{\frac{1}{6}}R_0^2\|\mathbf{u}_{h,1}-\mathbf{u}_{h,2}\|_{\mathbf{X}_h^{0,b}},\label{ineq: general nonlinearity estimate 2}
\end{align}
where the implicit constants depend on \(\|V^{(4)}\|_{C([-1,1])}\).
\end{lemma}

\begin{remark}
Taylor's theorem gives
\[
|\mathcal R_h[y]|
=
\frac{2}{h^4}
\left|
V'(h^2y)-h^2y-\frac{h^4}{2}y^2
\right|
\lesssim h^2|y|^3,
\qquad h^2|y|\leq1.
\]
Thus, the factor \(h^2\) compensates for the \(O(h^{-1})\) size of the lattice derivative \(\nabla_h\), leaving an additional factor of \(h\). Using only the elementary bound \(\|f_h\|_{L_z^\infty}\leq h^{-1/2}\|f_h\|_{L_z^2}\) for the two \(L_z^\infty\) factors consumes this remaining factor and yields
\[
\|\nabla_h\mathcal R_h[r_h]\|_{L_t^2([-T,T];H_z^{s'})}
\lesssim_{s,V}
h^{s-s'}T^{\frac12}R_0^2R_s
\]
under the assumptions of Lemma~\ref{lem: estimates for the higher-order remainder}. This bound gives smallness as \(h\to0\) when \(s>s'\). The Strichartz estimates \eqref{ineq: Strichartz for almost translation} allow us to retain the additional factor of \(h\) and obtain
\[
\|\nabla_h\mathcal R_h[r_h]\|_{L_t^2([-T,T];H_z^{s'})}
\lesssim_{s,V}
h^{1+s-s'}T^{\frac16}R_0^2R_s.
\]
In particular, the remainder tends to zero even when \(s'=s\). Its decay in \(h\) is therefore stronger than the corresponding bounds in Lemmas~\ref{lem: FPUT bilinear estimates}, \ref{lem: bilinear estimates for translation correction}, and \ref{lem: estimates for the translation correction}.
\end{remark}
\begin{lemma}[Gagliardo–Nirenberg inequality  {\cite[Proposition~2.4]{HY2019-2}}]
Let $s>0$, $1\leq p \leq q\leq \infty$, and $0<\theta<1$ satisfy $\frac{1}{q}=\frac{1}{p}-\theta s$. Then for $h\in(0,1]$,
\begin{equation}\label{ineq: Gagliardo–Nirenberg}
\|f_h\|_{L_z^q(h\Z)}
\ls
\|f_h\|_{L_z^p(h\Z)}^{1-\theta}
\big\||\pa_h|^sf_h\big\|_{L_z^p(h\Z)}^\theta,
\end{equation}
where the implicit constant is independent of $h\in (0,1]$.
\end{lemma}
\begin{lemma}[Strichartz estimates]\label{lemma: Strichartz estimates}
Let $h\in(0,1]$, $a\in\R$ and $b>\frac{1}{2}$. Suppose that $2 \leq q, r\leq \infty$ and $\frac{2}{q}+\frac{1}{r}=\frac{1}{2}$ with $(q,r)\neq(4,\infty)$. Then, we have
\begin{equation}\label{ineq: Strichartz for almost translation}
\big\||\pa_h|^{\frac{1}{q}}e^{at\pa_h}u_h\big\|_{L_t^q(\R;L_z^r(h\Z))}
\ls 
\|u_h\|_{X_{h,\pm}^{0,b}},
\end{equation}
where the implicit constant is independent of \(h\in(0,1]\) and \(a\in\mathbb{R}\).
\end{lemma}

\begin{proof}[Sketch of proof]
We follow the argument in \cite[Lemma~5.6]{HKY2021}, that is, \eqref{ineq: Strichartz for almost translation} with \(a=0\). Indeed, the additional phase \(at\xi\) is linear in \(\xi\), so
\[
\partial_\xi^2\big\{(z+at)\xi\pm ts_h(\xi)\big\}
=
\pm ts_h''(\xi).
\]
Thus, van der Corput's lemma gives the same dispersive bound, with a constant independent of \(a\) and \(h\).
The standard \(TT^*\) argument and the transfer principle for \(b>\frac12\) then yield the desired estimate.
\end{proof}
\begin{lemma}\label{lem: almost translation mixed Lp bound}
Let $h\in(0,1]$, $a\in\R$, $b>\frac{1}{2}$ and $T>0$. Then
\begin{equation}\label{ineq: almost translation mixed L4 bound}
\|e^{at\pa_h}u_h\|_{L_t^4([-T,T];L_z^\infty)}
\ls
T^{\frac{1}{12}}\|u_h\|_{X_{h,\pm}^{0,b}}.
\end{equation}
\end{lemma}
\begin{proof}
For any fixed $q\in (4,6)$, choose $r$ and $\theta$ so that
$$
\frac{1}{r}=\frac{1}{2}-\frac{2}{q},
\qquad
\theta=\frac{q}{r}=\frac{q}{2}-2\in(0,1).
$$
By the Gagliardo-Nirenberg inequality \eqref{ineq: Gagliardo–Nirenberg} and Hölder's inequality,
$$
\begin{aligned}
\|e^{at\pa_h}u_h(t)\|_{L_z^\infty}
&\ls
\|e^{at\pa_h}u_h(t)\|_{L_z^r}^{1-\theta}
\big\||\pa_h|^{1/q}e^{at\pa_h}u_h(t)\big\|_{L_z^r}^{\theta}\\
&\ls
\|e^{at\pa_h}u_h(t)\|_{L_z^2}^{\frac{2}{r}(1-\theta)}
\|e^{at\pa_h}u_h(t)\|_{L_z^\infty}^{(1-\frac{2}{r})(1-\theta)}
\big\||\pa_h|^{1/q}e^{at\pa_h}u_h(t)\big\|_{L_z^r}^{\theta}.
\end{aligned}
$$
Using the unitarity and the identity $1-(1-\frac{2}{r})(1-\theta)=\frac{6}{r}$, we obtain
$$
\|e^{at\pa_h}u_h(t)\|_{L_z^\infty}
\ls
\|u_h(t)\|_{L_z^2}^{1-\frac{q}{6}}
\big\||\pa_h|^{1/q}e^{at\pa_h}u_h(t)\big\|_{L_z^r}^{\frac{q}{6}}.
$$
Then the embedding $X_{h,\pm}^{0,b}\hookrightarrow C_tL_z^2$ and the Strichartz estimate \eqref{ineq: Strichartz for almost translation} yield
$$
\|e^{at\pa_h}u_h\|_{L_t^6L_z^\infty}
\ls
\|u_h\|_{L_t^\infty L_z^2}^{1-\frac{q}{6}}
\big\||\pa_h|^{1/q}e^{at\pa_h}u_h\big\|_{L_t^qL_z^r}^{\frac{q}{6}} \ls \|u_h\|_{X_{h,\pm}^{0,b}}.
$$
Hence, Hölder's inequality in time gives
$$
\|e^{at\pa_h}u_h\|_{L_t^4([-T,T];L_z^\infty)}
\ls
T^{\frac{1}{12}}\|e^{at\pa_h}u_h\|_{L_t^6([-T,T];L_z^\infty)}
\ls
T^{\frac{1}{12}}\|u_h\|_{X_{h,\pm}^{0,b}}.
$$
This completes the proof.
\end{proof}
\begin{lemma}\label{lemma: rh uniform bound}
Under the assumptions of
Lemma~\ref{lem: estimates for the higher-order remainder},
there exists \(h_0=h_0(R_0)\in(0,1]\) such that for \(h\in(0,h_0]\) and \(j=1,2\),
\[
\sup_{t\in\mathbb R}\|r_{h,j}(t)\|_{H_z^s(h\Z)}
\lesssim R_s
\quad\textup{and}\quad
\sup_{t\in\mathbb R}h^2\|r_{h,j}(t)\|_{L_z^\infty(h\Z)}
\leq1.
\]
\end{lemma}

\begin{proof}
By the embedding
\(X_{h,\pm}^{s,b}\hookrightarrow C_tH_z^s(h\Z)\)
and the unitarity of \(e^{a\pa_h}\) on \(H_z^s(h\Z)\), we have
\[\sup_{t\in\mathbb R}
\|r_{h,j}(t)\|_{H_z^s}
\lesssim
\sum_\pm
\|u_{h,j}^\pm\|_{X_{h,\pm}^{s,b}}
\lesssim R_s.\]
Moreover, the elementary inequality
\(\|f_h\|_{L_z^\infty}\leq h^{-1/2}\|f_h\|_{L_z^2}\)
and the above inequality with \(s=0\) give
\[
h^2\|r_{h,j}(t)\|_{L_z^\infty}
\leq
h^{\frac32}\|r_{h,j}(t)\|_{L_z^2}
\leq
c h^{\frac32}R_0,
\]
where \(c>0\) is independent of \(h\), \(t\), \(j\), and \(R_0\). Choosing \(h_0:=\min\{1,(cR_0)^{-\frac23}\}\) proves the second bound and completes the proof.
\end{proof}

\begin{proof}[Proof of Lemma~\ref{lem: estimates for the higher-order remainder}]
Fix \(0\leq s'\leq s\leq1\) and \(0<h\leq h_0\), where \(h_0\) is chosen as in Lemma~\ref{lemma: rh uniform bound}. Throughout the proof, time norms are taken over \([-T,T]\).

\noindent\textbf{\underline{Proof of \eqref{ineq: general nonlinearity estimate 1}.}} We first reduce the estimate for
\(\nabla_h\mathcal R_h[r_{h,j}]\) to a bound for
\(\mathcal R_h[r_{h,j}]\) in \(\dot H_z^s\). Indeed, the Fourier multiplier of \(\nabla_h\) is given by \(\frac{2i}{h}\sin\big(\frac{h\xi}{2}\big)\), and \(\frac{2}{h}\big|\sin\big(\frac{h\xi}{2}\big)\big|\leq|\xi|\). Note also that if \(0<|\xi|\leq\frac{\pi}{h}\), then \(|\xi|^{1-s}\la\xi\ra^{s'}
\lesssim h^{s-s'-1}\). Hence, Plancherel's theorem gives
\begin{equation}\label{eq: higher-order remainder proof}
\big\|\nabla_h\mathcal R_h[r_{h,j}]\big\|_{L_t^2H_z^{s'}}
\leq
\big\|\pa_h\mathcal R_h[r_{h,j}]\big\|_{L_t^2H_z^{s'}}
\lesssim
h^{s-s'-1}
\|\mathcal R_h[r_{h,j}]\|_{L_t^2\dot H_z^s}.
\end{equation}

Next, we claim that for \(0\leq s\leq1\),
\begin{equation}\label{ineq: remainder Hs estimate}
\|\mathcal R_h[r_{h,j}]\|_{\dot{H}_z^s}
\lesssim 
h^2\|r_{h,j}\|_{L_z^\infty}^2
\|r_{h,j}\|_{\dot{H}_z^s}.
\end{equation}
Suppose that \(0<s<1\). Then, by the norm equivalence 
\eqref{ineq: characterization of the lattice Sobolev norm}, we have
\[
\|\mathcal R_h[r_{h,j}]\|_{\dot H_z^s}^2
\sim
h\sum_{k\in\Z\setminus\{0\}}
\frac{\|\mathcal R_h[r_{h,j}(\cdot+hk)]
-\mathcal R_h[r_{h,j}]\|_{L_z^2}^2
}{|hk|^{1+2s}}.
\]
Note that by Lemma~\ref{lemma: rh uniform bound},
\[
h^2\|r_{h,j}(\cdot+hk)\|_{L_z^\infty}
\leq1
\]
for every \(k\in\Z\). On the other hand, for \(h^2|y|\leq1\), Taylor's theorem gives
\[
|\mathcal R_h'[y]|
=
\frac{2}{h^2}
\big|V''(h^2y)-1-h^2y\big|=
h^2 |V^{(4)}(h^2y^*)|y^2
\lesssim h^2 y^2
\]
for some \(y^*\) between \(0\) and \(y\).
Here, the implicit constant depends on
\(\|V^{(4)}\|_{C([-1,1])}\). Consequently, the fundamental theorem of calculus yields
\begin{equation}\label{ineq: remainder function difference bound}
|\mathcal R_h[y_1]-\mathcal R_h[y_2]|
\lesssim
h^2(|y_1|^2+|y_2|^2)|y_1-y_2|
\end{equation}
whenever \(h^2|y_1|,h^2|y_2|\leq1\). Thus, applying \eqref{ineq: remainder function difference bound} and \eqref{ineq: characterization of the lattice Sobolev norm} once more, we obtain
\[\begin{aligned}
\|\mathcal R_h[r_{h,j}]\|_{\dot H_z^s}^2
&\lesssim h\sum_{k\in\Z\setminus\{0\}}
\frac{1
}{|hk|^{1+2s}}h^4\Big(\|r_{h,j}(\cdot+hk)\|_{L_z^\infty}^2+\|r_{h,j}\|_{L_z^\infty}^2\Big)^2\|r_{h,j}(\cdot+hk)-r_{h,j}\|_{L_z^2}^2\\
&\lesssim h^4\|r_{h,j}\|_{\dot H_z^s}^2\|r_{h,j}\|_{L_z^\infty}^4.
\end{aligned}\]
For \(s=0\), the difference estimate
\eqref{ineq: remainder function difference bound} with \(y_2=0\) and \(\mathcal R_h[0]=0\) immediately yields 
\begin{equation}\label{ineq: remainder L2 estimate}
\|\mathcal R_h[r_{h,j}]\|_{L_z^2}
\lesssim
h^2\|r_{h,j}\|_{L_z^\infty}^2
\|r_{h,j}\|_{L_z^2}.
\end{equation}
For \(s=1\), we observe that by \eqref{ineq: remainder function difference bound}, 
\[
\begin{aligned}
\big|\pa_h^+\mathcal R_h[r_{h,j}](z)\big|
&=
\frac{
|\mathcal R_h[r_{h,j}(z+h)]
-\mathcal R_h[r_{h,j}(z)]|
}{h}\\
&\lesssim
h^2\|r_{h,j}\|_{L_z^\infty}^2
\frac{|r_{h,j}(z+h)-r_{h,j}(z)|}{h}=
h^2\|r_{h,j}\|_{L_z^\infty}^2
|\pa_h^+r_{h,j}(z)|
\end{aligned}
\]
for every \(z\in h\Z\). Hence, by the norm equivalence \eqref{ineq: norm equivalence of the lattice Sobolev norm},
\[
\begin{aligned}
\|\mathcal R_h[r_{h,j}]\|_{\dot H_z^1}
&\lesssim
\|\pa_h^+\mathcal R_h[r_{h,j}]\|_{L_z^2}\lesssim
h^2\|r_{h,j}\|_{L_z^\infty}^2
\|\pa_h^+r_{h,j}\|_{L_z^2}\lesssim
h^2\|r_{h,j}\|_{L_z^\infty}^2
\|r_{h,j}\|_{\dot H_z^1}.
\end{aligned}
\]
Therefore, the claim \eqref{ineq: remainder Hs estimate} is proved.

Finally, returning to \eqref{eq: higher-order remainder proof} and applying H\"older's inequality, the Strichartz estimate \eqref{ineq: almost translation mixed L4 bound}, and Lemma~\ref{lemma: rh uniform bound}, we obtain
\[
\big\|\nabla_h\mathcal R_h[r_{h,j}]\big\|_{L_t^2H_z^{s'}}\lesssim
h^{1+s-s'}
\|r_{h,j}\|_{L_t^4L_z^\infty}^2
\|r_{h,j}\|_{L_t^\infty H_z^s}\lesssim
h^{1+s-s'}T^{\frac16}R_0^2R_s.
\]
This proves \eqref{ineq: general nonlinearity estimate 1}.

\noindent\textbf{\underline{Proof of \eqref{ineq: general nonlinearity estimate 2}.}}
Using \eqref{ineq: remainder function difference bound} and the unitarity of \(e^{\pm\frac{t}{h^2}\pa_h}\), together with the identity
\[r_{h,1}-r_{h,2}=e^{-\frac{t}{h^2}\pa_h}(u_{h,1}^+-u_{h,2}^+)
+e^{\frac{t}{h^2}\pa_h}(u_{h,1}^--u_{h,2}^-),\]
we obtain
\[
\begin{aligned}
\big\|\mathcal R_h[r_{h,1}]-\mathcal R_h[r_{h,2}]\big\|_{L_z^2}
&\lesssim
h^2
\Big(
\|r_{h,1}\|_{L_z^\infty}^2
+\|r_{h,2}\|_{L_z^\infty}^2
\Big)
\|r_{h,1}-r_{h,2}\|_{L_z^2}\\
&\leq
h^2
\Big(
\|r_{h,1}\|_{L_z^\infty}^2
+\|r_{h,2}\|_{L_z^\infty}^2
\Big)
\sum_\pm
\|u_{h,1}^\pm-u_{h,2}^\pm\|_{L_z^2}.
\end{aligned}
\]
By \(\|\nabla_hf_h\|_{L_z^2}\leq 2h^{-1}\|f_h\|_{L_z^2}\), H\"older's inequality, \eqref{ineq: almost translation mixed L4 bound},
and the embedding \(X_{h,\pm}^{0,b}\hookrightarrow C_tL_z^2\),
\[
\begin{aligned}
\big\|
\nabla_h\big(
\mathcal R_h[r_{h,1}]-\mathcal R_h[r_{h,2}]
\big)
\big\|_{L_t^2L_z^2}&\lesssim
h\Big(
\|r_{h,1}\|_{L_t^4L_z^\infty}^2
+\|r_{h,2}\|_{L_t^4L_z^\infty}^2
\Big)
\|\mathbf{u}_{h,1}-\mathbf{u}_{h,2}\|_{\mathbf{X}_h^{0,b}}\\
&\lesssim
hT^{\frac16}R_0^2
\|\mathbf{u}_{h,1}-\mathbf{u}_{h,2}\|_{\mathbf{X}_h^{0,b}},
\end{aligned}
\]
which completes the proof of the lemma.
\end{proof}

\section{Uniform bounds and persistence of regularity}\label{sec: Uniform bounds and persistence of regularity}
In this section, we prove global existence and a Sobolev norm bound that grows at most exponentially in time for the coupled FPUT system. We also establish analogous results for the KdV equation and the auxiliary \textit{frequency-localized decoupled FPUT system}.

\subsection{Local-in-time bound for the FPUT system}

We first establish a local-in-time bound for the coupled FPUT system, uniformly for sufficiently small \(h>0\).

\begin{proposition}[Local-in-time uniform bounds for the FPUT system]
\label{prop: Local uniform bounds for the FPUT system}
Let \(0\leq s\leq1\), \(0<\delta<\frac14\), and
\(\frac12<b<\frac34-\delta\), and let \(h_0>0\) be a small number given in Lemma \ref{lem: estimates for the higher-order remainder}. Suppose that 
\[
\sup_{h\in(0,1]}
\|u_{h,0}^\pm\|_{L_z^2(h\Z)}
\leq R_0
\qquad\textup{and}\qquad
\sup_{h\in(0,1]}
\|u_{h,0}^\pm\|_{H_z^s(h\Z)}
\leq R_s.
\]
Then, there exists \(T=T(R_0)>0\), independent of \(h\) and \(R_s\), such that for any
\(h\in(0,h_0]\), the coupled FPUT system
\eqref{eq: coupled FPUT integral form} with initial data
\(u_{h,0}^\pm\) admits a unique solution \(u_h^\pm\in C_t([-T,T];H_z^s(h\Z))\) that has an extension to \(\R\), still denoted by
\(u_h^\pm\), satisfying
\begin{equation}\label{ineq: uniform Xsb norm bounds for FPUT}
\sup_{h\in(0,h_0]}
\|u_h^\pm\|_{X_{h,\pm}^{s,b}}
\lesssim R_s.
\end{equation}
\end{proposition}

\begin{remark}
The lifespan depends only on the \(L^2\)-bound of the initial data, allowing us to propagate higher Sobolev regularity on the same time interval.
\end{remark}

\begin{proof}[Proof of Proposition \ref{prop: Local uniform bounds for the FPUT system}]
Let $T\in(0,1]$ be a small number to be chosen later. For \(\mathbf{u}_h=(u_h^+,u_h^-)\in \mathbf{X}_h^{s,b}
=
X_{h,+}^{s,b}\times X_{h,-}^{s,b}\), we define 
$$
\mathbf{\Phi}_h(\mathbf{u}_h)
=
(\Phi_h^+(\mathbf{u}_h),\Phi_h^-(\mathbf{u}_h))
$$
by
$$
\begin{aligned}
\Phi_h^\pm(\mathbf{u}_h)(t)
&:=
\eta_1(t)S_h^\pm(t)u_{h,0}^\pm \mp
\frac{\eta_1(t)}{4}\int_0^t S_h^\pm(t-t')\eta_T(t')\nabla_h
\Big\{\big(u_h^\pm+e^{\pm\frac{2t'}{h^2}\pa_h}u_h^\mp\big)^2\\
&\qquad\qquad\qquad\qquad\qquad\qquad\qquad+\mathcal{T}_h^\pm\big[r_h[\mathbf{u}_h]\big]+e^{\pm\frac{t'}{h^2}\pa_h}\mathcal{R}_h\big[r_h[\mathbf{u}_h]\big]
\Big\}(t') dt',
\end{aligned}
$$
where
$$
r_h[\mathbf{u}_h](t):=
e^{-\frac{t}{h^2}\pa_h}u_h^+(t)+e^{\frac{t}{h^2}\pa_h}u_h^-(t).
$$
Then, by the linear estimates in Lemma \ref{lem: Xsb properties},
$$
\begin{aligned}
\|\Phi_h^\pm(\mathbf{u}_h)\|_{X_{h,\pm}^{s,b}}
&\ls
\|u_{h,0}^\pm\|_{H_z^s}
+T^\delta\big\|
\nabla_h\big(u_h^\pm+e^{\pm\frac{2t}{h^2}\pa_h}u_h^\mp\big)^2
\big\|_{X_{h,\pm}^{s,-(1-b-\delta)}}\\
&+T^\delta\big\|\nabla_h\mathcal{T}_h^\pm\big[r_h[\mathbf{u}_h]\big]\big\|_{X_{h,\pm}^{s,-(1-b-\delta)}}
+T^{1-b}\big\|\nabla_h
e^{\pm\frac{t}{h^2}\pa_h}\mathcal{R}_h\big[r_h[\mathbf{u}_h]\big]\big\|_{L_t^2([-2T, 2T];H_z^s)}.
\end{aligned}
$$
Hence, it follows from the bilinear estimates (Lemma~\ref{lem: FPUT bilinear estimates}), the translation correction estimates (Lemma \ref{lem: estimates for the translation correction}) and the higher-order remainder estimate (Lemma~\ref{lem: estimates for the higher-order remainder}) that 
\[\|\mathbf{\Phi}_h(\mathbf{u}_h)\|_{\mathbf{X}_h^{s,b}}
\leq
cR_s
+
c\big(T^\delta \|\mathbf{u}_h\|_{\mathbf{X}_h^{0,b}}\|\mathbf{u}_h\|_{\mathbf{X}_h^{s,b}}+T^{\frac{7}{6}-b}\|\mathbf{u}_h\|_{\mathbf{X}_h^{0,b}}^2\|\mathbf{u}_h\|_{\mathbf{X}_h^{s,b}}\big).\]
In the same way, one can show that if \(\mathbf{u}_h, \mathbf{u}_{h,1}, \mathbf{u}_{h,2}\in \mathbf{X}_h^{s,b}\), then
\[\begin{aligned}
\|\mathbf{\Phi}_h(\mathbf{u}_h)\|_{\mathbf{X}_h^{0,b}}
&\leq
cR_0
+
c\big(T^\delta \|\mathbf{u}_h\|_{\mathbf{X}_h^{0,b}}^2+T^{\frac{7}{6}-b}\|\mathbf{u}_h\|_{\mathbf{X}_h^{0,b}}^3\big)\\
\|\mathbf{\Phi}_h(\mathbf{u}_{h,1})
-\mathbf{\Phi}_h(\mathbf{u}_{h,2})\|_{\mathbf{X}_h^{0,b}}
&\leq
c\big\{T^\delta (\|\mathbf{u}_{h,1}\|_{\mathbf{X}_h^{0,b}}+\|\mathbf{u}_{h,2}\|_{\mathbf{X}_h^{0,b}})+T^{\frac{7}{6}-b}(\|\mathbf{u}_{h,1}\|_{\mathbf{X}_h^{0,b}}^2+\|\mathbf{u}_{h,2}\|_{\mathbf{X}_h^{0,b}}^2)\big\}\\
&\qquad \times\|\mathbf{u}_{h,1}-\mathbf{u}_{h,2}\|_{\mathbf{X}_h^{0,b}}.
\end{aligned}\]
Note also that $2\delta\leq\frac{7}{6}-b$ by the assumption, so we may replace \(T^{\frac{7}{6}-b}\) by \(T^{2\delta}\) in the above three inequalities. Therefore, taking 
$$
T:=\min\bigg\{\bigg(\frac{1}{16c^2\la R_0\ra^2}\bigg)^{1/\delta}, \frac{1}{2}\bigg\},
$$
we prove that \(\mathbf{\Phi}_h\) is contractive on a complete\footnote{By Lemma~\ref{lem: completeness of the fixed-point space}, this space is complete. } space
$$
\Big\{\mathbf{u}_h\in \mathbf{X}_h^{s,b}: \|\mathbf{u}_h\|_{\mathbf{X}_h^{0,b}}\leq 2cR_0,\quad \|\mathbf{u}_h\|_{\mathbf{X}_h^{s,b}}\leq 2cR_s\Big\}.
$$
equipped with the metric \(d(\mathbf{u}_{h,1}, \mathbf{u}_{h,2}):=\|\mathbf{u}_{h,1}-\mathbf{u}_{h,2}\|_{\mathbf{X}_h^{0,b}}\). This completes the proof.
\end{proof}

\subsection{Local-in-time bounds for the KdV equation}

For \(s\geq0\) and \(\frac12<b<\frac34\), the argument used in
Lemma~\ref{lem: FPUT bilinear estimates} also yields the KdV
bilinear estimate
\begin{equation}\label{ineq:KdV bilinear estimate}
\|\pa_x(uv)\|_{X_\pm^{s,-\frac14}}
\lesssim
\|u\|_{X_\pm^{s,b}}\|v\|_{X_\pm^{0,b}}
+
\|u\|_{X_\pm^{0,b}}\|v\|_{X_\pm^{s,b}};
\end{equation}
see \cite{KPV1996}.
Using this estimate in place of Lemma~\ref{lem: FPUT bilinear estimates}, we repeat the argument in Proposition~\ref{prop: Local uniform bounds for the FPUT system} to obtain local bounds with persistence of regularity.
Again, the lifespan can be chosen depending only on the \(L^2\)-norm of the initial data.

\begin{proposition}[Local uniform bounds for KdV]
\label{prop: local uniform bounds for KdV}
Let \(s\geq0\) and \(\frac12<b<\frac34\).
If
\[
\|w_0^\pm\|_{L_x^2(\R)}\leq R_0
\quad\textup{and}\quad
\|w_0^\pm\|_{H_x^s(\R)}\leq R_s,
\]
then there exists \(T=T(R_0)\sim
\la R_0\ra^{-\frac{1}{3/4-b}}\), independent of \(R_s\), such that the KdV equation
\eqref{eq:KdV integral form} with initial data \(w_0^\pm\)
admits a unique solution \(w^\pm\in C_t([-T,T];H_x^s(\R))\) whose extension to \(\R\) in time,
still denoted by \(w^\pm\), satisfies 
\begin{equation}\label{ineq: uniform Xsb norm bounds for KdV}
\|w^\pm\|_{X_\pm^{s,b}}\lesssim R_s.
\end{equation}
\end{proposition}

\subsection{Local-in-time bounds for the frequency-localized FPUT equation}\label{subsec: Local-in-time bounds for the frequency-localized FPUT equation}
To simplify the proof of the FPUT-to-KdV approximation, we follow the strategy of Hong and Yang~\cite{HY2024} and introduce the \textit{frequency-localized decoupled FPUT} equation
\begin{equation}\label{eq: lattice auxiliary equation}
\boxed{\quad
v_h^\pm(t)
=
S_h^\pm(t)P_{\leq N_0}^hu_{h,0}^\pm
\mp\frac14
\int_0^t
S_h^\pm(t-t')P_{\leq N_0}^h
\nabla_h\bigl(v_h^\pm(t')^2\bigr) dt',
\quad}
\end{equation}
where
\[
N_0:=\frac12 h^{-\frac25}.
\]
This equation retains only the quadratic self-interaction of each component, so \(v_h^+\) and \(v_h^-\) evolve independently.
Since \(P_{\leq N_0}^h\) commutes with \(S_h^\pm(t)\) and is applied to both the initial data and the nonlinear term, we also have
\[
P_{\leq N_0}^hv_h^\pm=v_h^\pm.
\]
This auxiliary equation \eqref{eq: lattice auxiliary equation} serves only as an analytical tool. In the next section, we first compare the coupled FPUT system with the auxiliary equation and then compare the latter with the KdV equations, dividing the proof into two simpler steps.

\begin{remark}[Choice of \(N_0\)]
The cutoff \(P_{\leq N_0}^h\) balances the low-frequency phase error with the
high-frequency tail. As noted in Remark~\ref{rmk: sensitivity on the phase function}, the convergence \(s_h(\xi)\to s(\xi)\) does not ensure uniform comparability of the FPUT and Airy Bourgain norms.
Indeed,
\[
|s_h(\xi)-s(\xi)|
=
\frac{1}{h^3}
\left|
h\xi-2\sin\left(\frac{h\xi}{2}\right)
-\frac{(h\xi)^3}{24}
\right|
\sim
\frac{1}{h^3}\to\infty,
\quad
\textup{for}\quad\frac{\pi}{2h}\leq|\xi|<\frac{\pi}{h}.
\]
A direct comparison may require more spatial regularity than is available in our setting.

To address this issue, we use the frequency-localized model \eqref{eq: lattice auxiliary equation}. For \(0\leq s\leq5\) and \(|\xi|\leq N_0\), Taylor's theorem gives
\begin{equation}\label{ineq: phase difference}
|s_h(\xi)-s(\xi)|
\lesssim
h^2|\xi|^5
\leq
h^2N_0^{5-s}|\xi|^s
\lesssim_s
h^{\frac{2s}{5}}|\xi|^s,
\end{equation}
whereas for high frequencies,
\begin{equation}\label{ineq: high freq tail estimate}
\|(1-P_{\leq N_0}^h)f_h\|_{L^2}
\lesssim
N_0^{-s}\|f_h\|_{H^s}.
\end{equation}
Thus, choosing \(N_0=\frac12h^{-\frac25}\) balances the two bounds, giving the common factor \(h^{\frac{2s}{5}}\).
\end{remark}

\begin{remark}
The cutoff $P_{\leq N_0}^h$ prevents frequency wrapping in the quadratic nonlinearity. Indeed, for every \(h\in(0,1]\), we have \(2N_0=h^{-\frac25}<\frac{\pi}{h}\). Thus, the product of two functions with Fourier support in \(\{|\xi|\leq N_0\}\) has Fourier support in \(\{|\xi|\leq2N_0\}\), which remains inside \([-\frac{\pi}{h},\frac{\pi}{h})\).
\end{remark}

The auxiliary equation \eqref{eq: lattice auxiliary equation} has the same linear propagator as the coupled FPUT system and contains only quadratic self-interactions.
Since \(P_{\leq N_0}^h\) is bounded on the Bourgain spaces, the bilinear estimate 
\eqref{ineq: FPUT bilinear estimate 1} and the argument in Proposition~\ref{prop: Local uniform bounds for the FPUT system} yield the following local bounds with persistence of regularity. The lifespan depends only on the uniform \(L^2\)-bound of the initial data.

\begin{proposition}[Local uniform bounds for the auxiliary equation]
\label{prop: local uniform bounds for the auxiliary equation}
Let \(s\geq0\), \(0<\delta<\frac14\), and
\(\frac12<b<\frac34-\delta\). Suppose that
\[
\sup_{h\in(0,1]}
\|u_{h,0}^\pm\|_{L_z^2(h\Z)}
\leq R_0
\quad\textup{and}\quad
\sup_{h\in(0,1]}
\|u_{h,0}^\pm\|_{H_z^s(h\Z)}
\leq R_s.
\]
Then there exists a time \(T=T(R_0)\sim\la R_0\ra^{-1/\delta}\), independent of
\(h\in(0,1]\) and \(R_s\), such that for any \(h\in(0,1]\), the auxiliary equation \eqref{eq: lattice auxiliary equation} with initial data \(P_{\leq N_0}^hu_{h,0}^\pm\) admits a unique solution \(v_h^\pm\in C_t([-T,T];H_z^s(h\Z))\). Moreover, this solution has an extension to \(\R\) in time, still denoted by \(v_h^\pm\), satisfying
\(P_{\leq N_0}^hv_h^\pm=v_h^\pm\) and
\begin{equation}\label{ineq: uniform Xsb norm bounds for auxiliary equation}
\sup_{h\in(0,1]}
\|v_h^\pm\|_{X_{h,\pm}^{s,b}}
\lesssim R_s.
\end{equation}
\end{proposition}

\subsection{Long-time bounds}

The key point of Propositions~\ref{prop: Local uniform bounds for the FPUT system} and \ref{prop: local uniform bounds for KdV} is that their lifespans depend only on the lower $L^2$-bound and are independent of the higher Sobolev bounds. Consequently, uniform-in-time control of $L^2$-norms allows the local estimates to be iterated, yielding global solutions and exponential-in-time $H^s$-bounds.

\begin{proposition}[Exponential-in-time bound]\label{prop: exponential bound}
Let $0 \leq s\leq 1$ and $R_0, R_s>0$, and let \(h_0>0\) be sufficiently small. Then, there exist constants $C=C(R_0,s)\geq 1$, $K=K(R_0,s)>0$, independent of $R_s$, $t$ and $h$, such that the following holds. Suppose that
$$
\sup_{h\in(0,1]}\|u_{h,0}^\pm\|_{L_z^2(h\Z)}, 
    \|w_0^\pm\|_{L_x^2(\R)} \leq R_0
\quad\textup{and}\quad
\sup_{h\in(0,1]}\|u_{h,0}^\pm\|_{H_z^s(h\Z)}, 
   \|w_0^\pm\|_{H_x^s(\R)} \leq R_s.
$$
Let $w^\pm(t)$ be the solution of the KdV equation \eqref{eq:KdV integral form} with initial data $w_0^\pm$. For each $h\in (0,h_0]$, let $u_h^\pm(t)$ be the solution of the coupled FPUT system \eqref{eq: coupled FPUT integral form} with initial data $u_{h,0}^\pm$. Then, these solutions extend globally in time and satisfy
\begin{equation}\label{ineq: FPUT KdV Hs norm growth}
\sum_\pm 
\Big(\|u_h^\pm(t)\|_{H_z^s(h\Z)}+\|w^\pm(t)\|_{H_x^s(\R)}\Big)
\leq Ce^{K|t|}R_s,
\end{equation}
for every $h\in(0,h_0]$ and $t\in\R$.
\end{proposition}
For the KdV, the required uniform $L^2$-control follows from the $L^2$-conservation law. The corresponding FPUT estimate, derived from conservation of the rescaled Hamiltonian, is established in the following lemma.
\begin{lemma}[Uniform $L^2$-bound for the FPUT system]\label{lem: uniform L2 bound for the FPUT}
Assume that $V$ satisfies \eqref{eq:potential-assumption}, and let $R>0$. For each $h\in(0,1]$, let $u_h^\pm$ be a solution of the coupled FPUT system \eqref{eq: coupled FPUT integral form}, on an interval containing $0$, with initial data $u_{h,0}^\pm$ satisfying
$$
\sum_\pm\|u_{h,0}^\pm\|_{L_z^2(h\Z)}^2\leq 2R^2.
$$
Then there exists $h_0=h_0(V,R)\in(0,1)$, such that, for every $h\in(0,h_0]$,
\begin{equation}\label{ineq: uniform L2 bound for the FPUT}
\sum_{\pm}\|u_h^\pm(t)\|_{L_z^2(h\Z)}^2\leq 3R^2,
\end{equation}
as long as the solution exists.
The same estimate holds with $(u_{h,0}^\pm,u_h^\pm)$ replaced by $(r_{h,0}^\pm,r_h^\pm)$, where $r_h^\pm$ are defined in \eqref{eq: def of transformed variables}.
\end{lemma}

\begin{proof}
Since $\|r_h^\pm\|_{L_z^2(h\Z)}=\|u_h^\pm\|_{L_z^2(h\Z)}$ by \eqref{eq: def of uh}, it suffices to prove the estimate for $r_h^\pm$ only. Recalling the conservation of the rescaled FPUT Hamiltonian (see \eqref{eq: rescaled FPUT-Hamiltonian}), by the definition of $r_h^\pm$ in \eqref{eq: def of transformed variables} and Taylor's theorem, we have
\begin{equation}\label{eq: rewritten energy}
\begin{aligned}
\sum_\pm\|r_h^\pm(t)\|_{L_z^2(h\Z)}^2
&=
\frac{1}{2}
\Big(\|r_h(t)\|_{L_z^2(h\Z)}^2
+\big\|h^2|\nabla_h|^{-1}\pa_tr_h(t)\big\|_{L_z^2(h\Z)}^2\Big)\\
&=
\mathcal{H}_h(t)
-h\sum_{x\in h\Z}
\bigg(\frac{1}{h^4}V\big(h^2r_h(t)\big)-\frac{1}{2}r_h(t)^2\bigg)\\
&=
\mathcal{H}_h(0)
-h\sum_{z\in h\Z}\frac{V'''(h^2r_h^{*}(t))}{6}h^2r_h(t)^3,
\end{aligned}
\end{equation}
where $r_h^{*}(t)$ lies between $0$ and $r_h(t)$. Suppose that, on an interval containing $0$,
$$
\sum_\pm\|r_h^\pm(t)\|_{L_z^2(h\Z)}^2\leq 4R^2.
$$ 
Then by the trivial bound \(\|f_h\|_{L_z^\infty}\leq\frac{1}{h^{1/2}}\|f_h\|_{L_z^2}\), for sufficiently small $h_0$,
$$
h^2\|r_h(t)\|_{L_z^\infty}
\leq
h^{\frac{3}{2}}\|r_h(t)\|_{L_z^2}
\leq 2\sqrt{2}h^{\frac{3}{2}}R\leq 1.
$$
For the last term in \eqref{eq: rewritten energy}, we obtain
\begin{equation}\label{ineq: energy latter term bound}
\begin{aligned}
\bigg|h\sum_{z\in h\Z}
\frac{V'''(h^2r_h^{*}(t))}{6}h^2r_h(t)^3\bigg|
&\leq \frac{\|V'''\|_{C([-1,1])}}{6}h^2\|r_h(t)\|_{L_z^\infty(h\Z)}\|r_h(t)\|_{L_z^2(h\Z)}^2 \\
&\leq \frac{\|V'''\|_{C([-1,1])}}{6}h^{\frac{3}{2}}\bigg(2\sum_\pm\|r_h^\pm(t)\|_{L_z^2(h\Z)}^2\bigg)^{3/2}.
\end{aligned}
\end{equation}
Applying \eqref{eq: rewritten energy} and \eqref{ineq: energy latter term bound} at $t=0$, we obtain, for some constant $C_1>0$,
\begin{equation}\label{ineq: H(0) bound}
\mathcal{H}_h(0)\leq 2R^2+C_1h^{\frac{3}{2}}R^3.
\end{equation}
Thus, choosing $h_0=h_0(V,R)$ sufficiently small, we obtain
$$
\sum_\pm\|r_h^\pm(t)\|_{L_z^2(h\Z)}^2
\leq
(2R^2+C_1h^{\frac{3}{2}}R^3) + 8C_2h^{\frac{3}{2}}R^3 \leq 3R^2,
$$
for all $h\in(0,h_0]$ and some constant $C_2>0$. Here $C_1$, $C_2$ depend only on $V$. Hence, \eqref{ineq: uniform L2 bound for the FPUT} follows from the bootstrap argument with continuity of the solution in time.
\end{proof}

\begin{proof}[Proof of Proposition~\ref{prop: exponential bound}]
Choose $h_0=h_0(V,R_0)>0$ sufficiently small so that Proposition~\ref{prop: Local uniform bounds for the FPUT system} and Lemma~\ref{lem: uniform L2 bound for the FPUT} hold. Lemma~\ref{lem: uniform L2 bound for the FPUT} and $L^2$ conservation law for the KdV imply 
$$
\|w^\pm(t)\|_{L_x^2(\R)}
\leq R_0
\quad\textup{and}\quad
\sum_{\pm}\|u_h^\pm(t)\|_{L_z^2(h\Z)}^2\leq 3R_0^2,
$$
as long as these solutions exist.
The system \eqref{eq: def of transformed variables} is autonomous. Thus, for any $t_0$ in the interval of existence, the profiles $e^{\mp\frac{t_0}{h^2}\pa_h}u_h^\pm(t_0+\cdot)$ solve \eqref{eq: coupled FPUT integral form} with initial data $r_h^\pm(t_0)$. The fixed almost translations preserve Sobolev norms, so Proposition~\ref{prop: Local uniform bounds for the FPUT system} applies uniformly in $t_0$.
Since the lifespans in Propositions~\ref{prop: Local uniform bounds for the FPUT system} and \ref{prop: local uniform bounds for KdV} depend only on these lower-order bounds, the local solutions extend globally in time.

By the local uniform bounds for the FPUT system \eqref{ineq: uniform Xsb norm bounds for FPUT} and KdV \eqref{ineq: uniform Xsb norm bounds for KdV}, there exist $T=T(R_0)>0$ and $C=C(R_0,s)\geq 1$, independent of $h$ and $R_s$, such that
$$
\sum_\pm
\Big(\|u_h^\pm(t+\Delta t)\|_{H_z^s(h\Z)}+\|w^\pm(t+\Delta t)\|_{H_x^s(\R)}\Big)
\leq
C\sum_\pm
\Big(\|u_h^\pm(t)\|_{H_z^s(h\Z)}+\|w^\pm(t)\|_{H_x^s(\R)}\Big),
$$
for every $t\in\R$ and $|\Delta t|\leq T$. Iterating this estimate yields \eqref{ineq: FPUT KdV Hs norm growth}.
\end{proof}

\section{KdV approximation and global extension}
\label{sec: KdV Approximation and Global Extension}
We now prove Theorem~\ref{thm: main theorem}, the main result of the paper.
\subsection{First local-in-time reduction}
\label{subsec: First local-in-time reduction}
In this subsection and the next, we establish a local-in-time FPUT--KdV approximation in the low-regularity range \(0<s\leq1\) (Proposition~\ref{prop: Local KdV limit for FPUT with persistence of regularity}). This extends the approximation result of Hong, Kwak, and Yang~\cite{HKY2021}, which required \(s>\frac34\), to lower regularity. We begin by reducing the coupled FPUT \eqref{eq: coupled FPUT integral form} to the frequency-localized decoupled FPUT  \eqref{eq: lattice auxiliary equation}.

\begin{lemma}[First local-in-time reduction]
\label{lem: first reduction}
Let \(0\leq s\leq1\), and assume that
\[
\sup_{h\in(0,1]}
\|u_{h,0}^\pm\|_{L_z^2(h\Z)}
\leq R_0
\quad\textup{and}\quad
\sup_{h\in(0,1]}
\|u_{h,0}^\pm\|_{H_z^s(h\Z)}
\leq R_s.
\]
For each \(h\in(0,1]\), let \(u_h^\pm\) (resp., \(v_h^\pm\)) be the solution to the coupled FPUT
\eqref{eq: coupled FPUT integral form} (resp., the frequency-localized decoupled FPUT \eqref{eq: lattice auxiliary equation}) with initial data \(u_{h,0}^\pm\) (resp., \(P_{\leq N_0}^h u_{h,0}^\pm\)), where \(N_0=\frac12h^{-\frac25}\).
Then, there exist \(h_0\in(0,1]\), \(T=T(R_0)>0\), and \(C_s>0\), independent of \(h\in(0,h_0]\) and \(R_s\), such that both solutions exist on \([-T,T]\) and
\begin{equation}\label{ineq: first reduction}
\|u_h^\pm-v_h^\pm\|_{C_t([-T,T];L_z^2(h\Z))}
\leq
C_sR_sh^{\frac{2s}{5}}.
\end{equation}
\end{lemma}

\begin{proof}
Fix $0<\delta<\frac{1}{4}$ and
$\frac{1}{2}<b<\frac{3}{4}-\delta$ so that the bilinear estimates in
Lemma~\ref{lem: FPUT bilinear estimates} hold. Next, choose $h_0=h_0(V,R_0,\alpha)\in(0,1]$ sufficiently small so that
\eqref{ineq: general nonlinearity estimate 1} holds and
\(N_0=\frac{1}{2}h^{-\frac{2}{5}}\leq\frac{\alpha}{2h}\)
for every \(h\in(0,h_0]\). Let $T\in(0,\frac{1}{2}]$ be sufficiently small, depending only on $R_0$, to be chosen below.
By the local well-posedness results (Propositions~\ref{prop: Local uniform bounds for the FPUT system} and \ref{prop: local uniform bounds for the auxiliary equation}), there exist $u_h^\pm(t)$ and $v_h^\pm(t)$ solve the equations \eqref{eq: coupled FPUT integral form} and \eqref{eq: lattice auxiliary equation}, respectively, on the time interval $[-T,T]$ for every $h\in(0,h_0]$, and
\begin{equation}\label{ineq: uniform Xsb bound for first reduction}
\|u_h^\pm\|_{X_{h,\pm}^{0,b}}
+
\|v_h^\pm\|_{X_{h,\pm}^{0,b}}
\lesssim R_0,
\qquad
\|u_h^\pm\|_{X_{h,\pm}^{s,b}}
+
\|v_h^\pm\|_{X_{h,\pm}^{s,b}}
\lesssim R_s.
\end{equation}

Subtracting the two Duhamel formulas with the time cutoffs \(\eta_1\) and \(\eta_T\), we write 
\[
\begin{aligned}
(u_h^\pm-v_h^\pm)(t)
&=
\eta_1(t)S_h^\pm(t)
P_{>N_0}^h u_{h,0}^\pm\\
&\quad
\mp\frac{\eta_1(t)}4
\int_0^t
S_h^\pm(t-t')\eta_T(t')\nabla_h
\Big\{
\big(
u_h^\pm
+
e^{\pm\frac{2t'}{h^2}\pa_h}u_h^\mp
\big)^2
-
P_{\leq N_0}^h(v_h^\pm)^2\\
&\hspace{6.6 cm}
+\mathcal T_h^\pm[r_h]
+
e^{\pm\frac{t'}{h^2}\pa_h}\mathcal R_h[r_h]
\Big\}(t') dt',
\end{aligned}
\]
where \(P_{>N_0}^h=1-P_{\leq N_0}^h\). Then, by the linear estimates (Lemma~\ref{lem: Xsb properties}),
$$
\begin{aligned}
\|u_h^\pm-v_h^\pm\|_{X_{h,\pm}^{0,b}}
&\ls
\|P_{> N_0}^hu_{h,0}^\pm\|_{L_z^2(h\Z)}+T^\delta
\Big\|\nabla_h
\Big\{
\big(u_h^\pm+e^{\pm\frac{2t}{h^2}\pa_h}u_h^\mp\big)^2
-P_{\leq N_0}^h(v_h^\pm)^2
\Big\}
\Big\|_{X_{h,\pm}^{0,-(1-b-\delta)}}\\
&\quad
+T^\delta
\|\nabla_h\mathcal{T}_h^\pm[r_h]\|_{X_{h,\pm}^{0,-(1-b-\delta)}}
+T^{1-b}
\big\|\nabla_he^{\pm\frac{t}{h^2}\pa_h}\mathcal{R}_h[r_h]
\big\|_{L_t^2([-2T,2T];L_z^2)}\\
&=\textup{(I)}+\textup{(II)}+\textup{(III)}+\textup{(IV)}.
\end{aligned}
$$
For \(\textup{(I)}\), we have
\[
\|P_{>N_0}^hu_{h,0}^\pm\|_{L_z^2}
\lesssim
N_0^{-s}\|u_{h,0}^\pm\|_{H_z^s}.
\]
For the quadratic nonlinear terms \(\textup{(II)}\), we expand
\[
\begin{aligned}
&\big(u_h^\pm+e^{\pm\frac{2t}{h^2}\pa_h}u_h^\mp\big)^2
-P_{\leq N_0}^h(v_h^\pm)^2\\
&=
P_{\leq N_0}^h\big\{(u_h^\pm)^2-(v_h^\pm)^2\big\}
+P_{>N_0}^h(u_h^\pm)^2
+\big(e^{\pm\frac{2t}{h^2}\pa_h}u_h^\mp\big)^2
+2u_h^\pm\big(e^{\pm\frac{2t}{h^2}\pa_h}u_h^\mp\big)
\end{aligned}
\]
and apply the bilinear estimates (Lemma~\ref{lem: FPUT bilinear estimates}). For the translation correction \(\textup{(III)}\), we apply
\eqref{ineq: translation correction estimate 1}.
For the higher-order remainder \(\textup{(IV)}\), we use the unitarity of $e^{\pm\frac{t}{h^2}\pa_h}$ and the bound \eqref{ineq: general nonlinearity estimate 1} on $[-2T,2T]$. Consequently, we obtain
$$
\begin{aligned}
\|u_h^\pm-v_h^\pm\|_{X_{h,\pm}^{0,b}}
&\ls
N_0^{-s}\|u_{h,0}^\pm\|_{H_z^s}
+T^\delta
\Big(
\|u_h^\pm\|_{X_{h,\pm}^{0,b}}
+\|v_h^\pm\|_{X_{h,\pm}^{0,b}}
\Big)
\|u_h^\pm-v_h^\pm\|_{X_{h,\pm}^{0,b}}\\
&\quad
+T^\delta N_0^{-s}
\|u_h^\pm\|_{X_{h,\pm}^{s,b}}
\|u_h^\pm\|_{X_{h,\pm}^{0,b}}
+T^\delta h^{\min\{s,\frac{1}{2}\}}
\|u_h^\mp\|_{X_{h,\mp}^{s,b}}
\|u_h^\mp\|_{X_{h,\mp}^{0,b}}\\
&\quad
+T^\delta h^{\min\{s,\frac{1}{2}\}}
\Big(
\|u_h^\pm\|_{X_{h,\pm}^{s,b}}
\|u_h^\mp\|_{X_{h,\mp}^{0,b}}
+\|u_h^\pm\|_{X_{h,\pm}^{0,b}}
\|u_h^\mp\|_{X_{h,\mp}^{s,b}}
\Big)\\
&\quad
+T^\delta h^sR_0R_s
+T^{1-b}T^{\frac{1}{6}}h^{1+s}R_0^2R_s.
\end{aligned}
$$
Recalling $N_0=\frac{1}{2}h^{-\frac{2}{5}}$ and applying the uniform $X^{s,b}$-bounds \eqref{ineq: uniform Xsb bound for first reduction}, we obtain
\[\|u_h^\pm-v_h^\pm\|_{X_{h,\pm}^{0,b}}
\ls
h^{\frac{2s}{5}}R_s
+T^\delta R_0\|u_h^\pm-v_h^\pm\|_{X_{h,\pm}^{0,b}}+T^\delta
h^{\frac{2s}{5}}R_0R_s
+T^{\frac{7}{6}-b}h^{1+s}R_0^2R_s.\]
Thus, there exist \(c_0, c_s\geq1\), independent of \(h\), \(T\), \(R_0\), and \(R_s\), such that 
\begin{equation}\label{ineq: u-v estimate}
\|u_h^\pm-v_h^\pm\|_{X_{h,\pm}^{0,b}}
\leq
c_s
\big\{
1+T^\delta R_0+(T^\delta R_0)^2
\big\}
h^{\frac{2s}{5}}R_s+
c_0T^\delta R_0
\|u_h^\pm-v_h^\pm\|_{X_{h,\pm}^{0,b}}.
\end{equation}
Taking \(T=T(R_0)>0\) sufficiently small so that
\(c_0T^\delta R_0\leq\frac12\), we absorb the last term in \eqref{ineq: u-v estimate} and obtain
\[
\|u_h^\pm-v_h^\pm\|_{X_{h,\pm}^{0,b}}
\leq
4c_s h^{\frac{2s}{5}}R_s.
\]
Finally, since \(b>\frac12\), by the embedding
\(X_{h,\pm}^{0,b}\hookrightarrow C_tL_z^2(h\Z)\), we prove \eqref{ineq: first reduction}.
\end{proof}

\subsection{Second local-in-time reduction}
\label{subsec: Second local-in-time reduction}
Next, we reduce the frequency-localized decoupled FPUT system \eqref{eq: lattice auxiliary equation} to the KdV equations \eqref{eq:KdV integral form}.
\begin{lemma}[Second local-in-time reduction]
\label{lem: second reduction}
Let \(0\leq s\leq1\), and assume that
$$
\sup_{h\in(0,1]}
\|u_{h,0}^\pm\|_{L_z^2(h\Z)}, 
\|w_0^\pm\|_{L_x^2(\R)}
\leq R_0
\quad\textup{and}\quad
\sup_{h\in(0,1]}
\|u_{h,0}^\pm\|_{H_z^s(h\Z)}, 
\|w_0^\pm\|_{H_x^s(\R)}
\leq R_s.
$$
For each \(h\in(0,1]\), let \(v_h^\pm\) (resp., \(w^\pm\)) be the solution to the
frequency-localized decoupled FPUT
\eqref{eq: lattice auxiliary equation} (resp., the KdV 
\eqref{eq:KdV integral form}) with initial data
\(P_{\leq N_0}^h u_{h,0}^\pm\) (resp., \(w_0^\pm\)), where \(N_0=\frac12h^{-\frac25}\). Then, there exist \(h_0\in(0,1]\), \(T=T(R_0)>0\) and  \(C_s>0\), independent of
\(h\in(0,h_0]\) and \(R_s\), such that, for every \(h\in(0,h_0]\),
both solutions exist on \([-T,T]\) and
\begin{equation}\label{ineq: second reduction}
\|\mathfrak{C}_hv_h^\pm-w^\pm\|_{C_t([-T,T];L_x^2(\R))}
\leq
C_s
\|\mathfrak{C}_hu_{h,0}^\pm-w_0^\pm\|_{L_x^2(\R)}
+
C_sR_sh^{\frac{2s}{5}},
\end{equation}
where \(\mathfrak{C}_h\) is the continuation
operator defined in \eqref{eq: definition of Ch}.
\end{lemma}

The proof follows a similar strategy to that of
Lemma~\ref{lem: first reduction}, but a key difference is that we must compare a function on \(h\Z\) with one on \(\R\). To make this comparison, we set \(N_0=\frac12h^{-\frac25}\), take \(h\) sufficiently small so that \(2N_0\leq\frac{\alpha}{h}\), and use the continuation
operator \(\mathfrak{C}_h\). Indeed, since Bourgain norms are sensitive to the underlying phase functions, localization to low frequencies is important for comparing the FPUT and Airy Bourgain norms uniformly in \(h\). This is the main reason for introducing the auxiliary equation
\eqref{eq: lattice auxiliary equation}.

We extend the FPUT phase \(s_h(\xi)\) from \(\T_h\) to \(\R\)
using the same formula \eqref{eq: FPUT phase function}, and define
the corresponding flow \(\widetilde{S}_h^\pm(t)\) on \(L_x^2(\R)\) by
$$
\mathcal{F}_x\big(\widetilde{S}_h^\pm(t)f\big)(\xi)
=e^{\pm its_h(\xi)}\hat{f}(\xi),
\qquad \xi\in\R.
$$
Let \(\widetilde{X}_{h,\pm}^{s,b}\) denote the Bourgain space associated
with this flow, equipped with the norm
$$
\|u\|_{\widetilde{X}_{h,\pm}^{s,b}(\R\times\R)}
:=
\big\|
\la\xi\ra^s\la\tau\mp s_h(\xi)\ra^b\tilde{u}(\tau,\xi)
\big\|_{L_{\tau,\xi}^2(\R\times\R)}.
$$
On the localized frequency range, the continuation operator gives
an exact identification between the lattice Bourgain norm and
the continuum Bourgain norm associated with this extended FPUT flow.

\begin{lemma}[Low-frequency identification of Bourgain norms]\label{lemma: Low-frequency identification of Bourgain norms}
Let $s,b\in\R$ and $0<N\leq \frac{\alpha}{h}$.
Then
\begin{equation}\label{eq: equality of the continuum lattice Xsb norms}
\|\mathfrak{C}_hP_{\leq N}^hu_h\|_{\widetilde{X}_{h,\pm}^{s,b}(\R\times\R)}
=\|P_{\leq N}^hu_h\|_{X_{h,\pm}^{s,b}(\R\times h\Z)}  .
\end{equation}
\end{lemma}
\begin{proof}
By \eqref{eq: localized Ch FT}, we have
$$
\|\mathfrak{C}_hP_{\leq N}^hu_h\|_{\widetilde{X}_{h,\pm}^{s,b}(\R\times\R)}^2
=
\int_\R\int_{-N}^N\la\xi\ra^{2s}\la\tau\mp s_h(\xi)\ra^{2b}|\hat{\varphi}(h\xi)|^2|\tilde{u}_h(\tau, \xi)|^2  d\xi  d\tau.
$$
Since $|\hat{\varphi}(h\xi)|=1$ for $|\xi|\leq N\leq \frac{\alpha}{h}$, \eqref{eq: equality of the continuum lattice Xsb norms} follows.
\end{proof}
Since \(P_{\leq N_0}^hv_h^\pm=v_h^\pm\), applying \(\mathfrak{C}_h\)
to \eqref{eq: lattice auxiliary equation} and using
\eqref{eq: commutativity of multipliers in low frequency} gives the
following continuum representation on the interval of existence of
\(v_h^\pm\):
\begin{equation}\label{eq: continuum auxiliary equation}
\mathfrak{C}_hv_h^\pm(t)
=\widetilde{S}_h^\pm(t)\mathfrak{C}_hP_{\leq N_0}^hu_{h,0}^\pm 
\mp\frac{1}{4}
\int_0^t\widetilde{S}_h^\pm(t-t') m_h(-i\partial_x)\mathfrak{C}_hP_{\leq N_0}^h\big\{(v_h^\pm(t'))^2\big\} dt',
\end{equation}
where $m_h(-i\partial_x)$ is the Fourier multiplier on $\R$ defined by
$$
\big(m_h(-i\partial_x)f\big)\widehat{\hphantom{a}}(\xi)=m_h(\xi)\hat{f}(\xi),
\qquad \xi\in\R\,     \textup{ and }\,     f\in L_x^2(\R).
$$
The low-frequency comparison of \eqref{eq: continuum auxiliary equation}
with \eqref{eq:KdV integral form} requires estimates for three discrepancies:
the difference between
the FPUT and Airy propagators, the difference between
\(m_h(-i\partial_x)\) and \(\pa_x\), and the product discrepancy
\(\mathfrak{C}_h\{(v_h^\pm)^2\}-(\mathfrak{C}_hv_h^\pm)^2\).
We estimate these errors in
Lemmas~\ref{lem: Low-frequency comparison of the FPUT and Airy flows},
\ref{lem: Low-frequency comparison of derivative symbols}, and
\ref{lem: Low-frequency product discrepancy for Ch}, respectively.

We begin with the linear flows. Although their Bourgain norms need not
be uniformly comparable at arbitrary frequencies, the phase difference
is bounded on \(\{|\xi|\leq N_0\}\). This gives a low-frequency norm
equivalence and allows us to estimate the difference of the two flows
in the Airy Bourgain space.
\begin{lemma}[Low-frequency comparison of the FPUT and Airy flows]\label{lem: Low-frequency comparison of the FPUT and Airy flows}
For $0<T\leq1$, set $\eta_T(t)=\eta(\frac{t}{T})$ where $\eta\in C_c^\infty(\R)$ is a smooth time cut-off. Then, for every $s,b\in \R$, we have
\begin{align}
\|P_{\leq N_0}w\|_{X_{\pm}^{s,b}}
&\sim
\|P_{\leq N_0}w\|_{\widetilde{X}_{h,\pm}^{s,b}},\label{ineq: low frequency comparison estimate 1}\\
\|P_{\leq N_0}^hu_h\|_{X_{h,\pm}^{s,b}}
&\sim
\|\mathfrak{C}_hP_{\leq N_0}^hu_h\|_{X_{\pm}^{s,b}}.\label{ineq: Xsb norm equivalence of Ch in low frequency}
\end{align}
Moreover, if $0\leq s\leq 5$ and $\frac{1}{2}<b\leq 1$, then
\begin{equation}\label{ineq: low frequency comparison estimate 2}
\big\|\eta_T(t)\big(\widetilde{S}_{h}^\pm(t)-S^\pm(t)\big)P_{\leq N_0}w_0\big\|_{X_{\pm}^{0,b}}
\ls
h^{\frac{2s}{5}}T^{\frac{3}{2}-b}\|w_0\|_{H_x^s(\R)},
\end{equation}
\begin{equation}\label{ineq: low frequency comparison estimate 3}
\bigg\|
\eta_T(t)\int_0^t\big(\widetilde{S}_{h}^\pm(t-t')-S^\pm(t-t')\big)\eta_T(t')(P_{\leq N_0}F)(t')dt'
\bigg\|_{X_{\pm}^{0,b}}
\ls
h^{\frac{2s}{5}}T^{\frac{3}{2}-b}\|F\|_{X_{\pm}^{s,-(1-b)}}.
\end{equation}
\end{lemma}

\begin{proof}
For $|\xi|\leq N_0$, Taylor's theorem yields
\begin{equation}\label{ineq: phase difference0}
|s_h(\xi)-s(\xi)|\ls h^2|\xi|^5\ls 1,
\end{equation}
so that $\la \tau\mp\{\theta s_h(\xi)+(1-\theta)s(\xi)\}\ra
\sim \la\tau \mp s(\xi)\ra$ for any $\theta\in[0,1]$. If $P_{\leq N_0}u=u$, then
\begin{equation}\label{ineq: interpolated Xsb equivalence}
\|u\|_{X_{\pm,\theta}^{s,b}}
:=
\big\|\la\xi\ra^s\la \tau\mp\{\theta s_h(\xi)+(1-\theta)s(\xi)\}\ra^b\tilde{u}(\tau, \xi)\big\|_{L_{\tau, \xi}^2}
\sim
\|u\|_{X_{\pm}^{s,b}}.
\end{equation}
Taking $\theta=1$ proves \eqref{ineq: low frequency comparison estimate 1}. Also, \eqref{eq: equality of the continuum lattice Xsb norms} and \eqref{ineq: low frequency comparison estimate 1} give \eqref{ineq: Xsb norm equivalence of Ch in low frequency}:
$$
\|P_{\leq N_0}^hu_h\|_{X_{h,\pm}^{s,b}(\R\times h\Z)}
=
\|\mathfrak{C}_hP_{\leq N_0}^hu_h\|_{\widetilde{X}_{h,\pm}^{s,b}(\R\times\R)}
\sim
\|\mathfrak{C}_hP_{\leq N_0}^hu_h\|_{X_{\pm}^{s,b}(\R\times\R)}.
$$
In the time-cutoff estimates below, we also use
\(t\eta_T(t)=T\cdot\frac{t}{T}\eta(\frac{t}{T})\) and apply the same estimates with the
fixed smooth profile \(t\mapsto t\eta(t)\).
By the fundamental theorem of calculus,
\begin{equation}\label{eq: phase difference FTC}
e^{\pm its_h(\xi)}-e^{\pm its(\xi)}
=
\pm it\big(s_h(\xi)-s(\xi)\big)\int_0^1 e^{\pm it\{\theta s_h(\xi)+(1-\theta)s(\xi)\}}d\theta.
\end{equation}
Using \eqref{ineq: interpolated Xsb equivalence}, \eqref{ineq: phase difference} and Lemma~\ref{lem: Xsb properties}, we obtain \eqref{ineq: low frequency comparison estimate 2}:
$$
\begin{aligned}
&\big\|\eta_T(t)\big(\widetilde{S}_{h}^\pm(t)-S^\pm(t)\big)P_{\leq N_0}w_0\big\|_{X_{\pm}^{0,b}}\\
&\leq
\int_0^1
\Big\|t\eta_T(t)e^{\pm it\{\theta s_h(-i\partial_x)+(1-\theta)s(-i\partial_x)\}}\big\{s_h(-i\partial_x)-s(-i\partial_x)\big\}P_{\leq N_0}w_0\Big\|_{X_\pm^{0,b}}  d\theta\\
&
\ls
h^{\frac{2s}{5}}
\int_0^1
\Big\|t\eta_T(t)e^{\pm it\{\theta s_h(-i\partial_x)+(1-\theta)s(-i\partial_x)\}}P_{\leq N_0}w_0\Big\|_{{X_{\pm, \theta}^{s,b}}}  d\theta
\ls h^{\frac{2s}{5}}T^{\frac{3}{2}-b}\|w_0\|_{H_x^s(\R)}.
\end{aligned}
$$
Similarly, the inhomogeneous estimate \eqref{ineq: low frequency comparison estimate 3} follows by the same phase expansion, splitting the factor $t-t'$ into $t$ and $-t'$ and applying the inhomogeneous estimate in Lemma~\ref{lem: Xsb properties}.
\end{proof}
We next compare the lattice derivative multiplier with the continuum
derivative. Taylor expansion on \(\{|\xi|\leq N_0\}\) gives the following
estimate, which holds for every temporal exponent and will be used
with exponent \(-\frac14\).
\begin{lemma}[Low-frequency comparison of derivative symbols]\label{lem: Low-frequency comparison of derivative symbols}
For $0\leq s\leq1$ and $b\in\R$, one has
\begin{equation}\label{ineq: derivative symbol difference}
\big\|P_{\leq N_0}\{m_h(-i\partial_x)-\pa_x\}u\big\|_{X_\pm^{0,b}}
\ls
h^{\frac{2s}{5}}\|\pa_xu\|_{X_\pm^{s,b}}.
\end{equation}
\end{lemma}
\begin{proof}
By Taylor expansion, for $|\xi|\leq N_0$, we obtain
$$
|m_h(\xi)-i\xi|\ls h^2|\xi|^3\ls h^2N_0^{2-s}|\xi|\la\xi\ra^s 
\ls h^{\frac{6+2s}{5}}|\xi|\la\xi\ra^s\ls h^{\frac{2s}{5}}|\xi|\la\xi\ra^s,
$$
since $0<h\leq 1$. Therefore, by the definition of the $X_\pm^{s,b}$ norm, \eqref{ineq: derivative symbol difference} follows.
\end{proof}

The remaining error comes from the failure of \(\mathfrak{C}_h\) to
preserve products. The frequency restriction below prevents wrapping
across the boundary of the fundamental domain $[-\frac{\pi}{h},\frac{\pi}{h})$.
\begin{lemma}[Low-frequency product identity]\label{lem: Low-frequency product discrepancy for Ch}
Let $0<2N\leq \frac{\alpha}{h}$. If $P_{\leq N}^hu_h=u_h$ and $P_{\leq N}^hv_h=v_h$, then
\begin{equation}\label{ineq: almost multiplicativity}
\mathfrak{C}_h(u_h v_h)=(\mathfrak{C}_hu_h)(\mathfrak{C}_hv_h)
\end{equation}
\end{lemma}
\begin{proof}
The condition $2N\leq\frac{\alpha}{h}<\frac{\pi}{h}$ prevents frequency wrapping in $u_h v_h$. Thus, by \eqref{eq: localized Ch FT} and \eqref{phi assumption 2}, our lemma follows.
\end{proof}

We now prove the second reduction. The high-frequency part is controlled by the Sobolev tail of the KdV solution, while the preceding three estimates control the low-frequency comparison.
\begin{proof}[Proof of Lemma \ref{lem: second reduction}]
Fix \(0<\delta<\frac14\) and
\(\frac12<b<\frac34-\delta\). Let $T\in(0,\frac{1}{2}]$ be sufficiently small, depending only on $R_0$,
to be specified below. By
Propositions~\ref{prop: local uniform bounds for KdV}
and~\ref{prop: local uniform bounds for the auxiliary equation}, solutions \(v_h^\pm\) and \(w^\pm\) exist on \([-T,T]\).
We use the global extensions constructed as fixed points of the
corresponding time-cutoff equations with the same \(\eta_1\) and
\(\eta_T\). For each sufficiently small \(T\) chosen below, these
extensions satisfy
\begin{equation}\label{ineq: uniform Xsb bound for local KdV limit}
\|v_h^\pm\|_{X_{h,\pm}^{0,b}}
+
\|w^\pm\|_{X_\pm^{0,b}}
\lesssim R_0,
\qquad
\|v_h^\pm\|_{X_{h,\pm}^{s,b}}
+
\|w^\pm\|_{X_\pm^{s,b}}
\lesssim R_s.
\end{equation}
We take \(h_0=\min\{1,\alpha^{5/3}\}\) so that \(2N_0=h^{-\frac25}\leq\frac{\alpha}{h}\). 

For the difference, we write 
\begin{equation}\label{ineq: Cv-w decomposition}
\|\mathfrak{C}_hv_h^\pm-w^\pm\|_{X_{\pm}^{0,b}}
\leq
\|P_{\leq N_0}(\mathfrak{C}_hv_h^\pm-w^\pm)\|_{X_{\pm}^{0,b}}
+\|P_{>N_0}w^\pm\|_{X_{\pm}^{0,b}}.
\end{equation}
Indeed, $P_{\leq N_0}^hv_h^\pm=v_h^\pm$ by construction, so low-frequency compatibility \eqref{eq: localized Ch FT} gives
$$
P_{\leq N_0}\mathfrak{C}_hv_h^\pm
=
\mathfrak{C}_hP_{\leq N_0}^hv_h^\pm
=
\mathfrak{C}_hv_h^\pm.
$$
Moreover, the uniform $X^{s,b}$-bounds \eqref{ineq: uniform Xsb bound for local KdV limit} imply
\begin{equation}\label{ineq: w high freq estimate}
\|P_{>N_0}w^\pm\|_{X_{\pm}^{0,b}}
\ls
N_0^{-s}\|w^\pm\|_{X_{\pm}^{s,b}}
\ls
h^{\frac{2s}{5}}R_s.
\end{equation}
Hence, it suffices to estimate the low-frequency part. By the time-cutoff versions of \eqref{eq: continuum auxiliary equation}
and \eqref{eq:KdV integral form}, we write, for all \(t\in\R\), 
$$
\begin{aligned}
&P_{\leq N_0}(\mathfrak{C}_hv_h^\pm-w^\pm)(t)\\
&=\eta_1(t)\big(\widetilde{S}_h^\pm(t)-S^\pm(t)\big)\mathfrak{C}_hP_{\leq N_0}^hu_{h,0}^\pm +\eta_1(t)S^\pm(t)P_{\leq N_0}\big(\mathfrak{C}_hu_{h,0}^\pm-w_0^\pm\big) \\
&
\quad\mp\frac{\eta_1(t)}{4}
\int_0^t\big(\widetilde{S}_h^\pm(t-t')-S^\pm(t-t')\big) \eta_T(t')P_{\leq N_0}m_h(-i\partial_x)\mathfrak{C}_h\big\{(v_h^\pm)^2\big\}(t') dt'\\
&\quad\mp\frac{\eta_1(t)}{4}
\int_0^tS^\pm(t-t')\eta_T(t')P_{\leq N_0}\Big[m_h(-i\partial_x)\mathfrak{C}_h\big\{(v_h^\pm)^2\big\}-
\pa_x
\big\{(w^\pm)^2\big\}\Big](t')dt'.
\end{aligned}
$$
Here we used $P_{\leq N_0}\mathfrak{C}_hP_{\leq N_0}^h=P_{\leq N_0}\mathfrak{C}_h$ which holds by \eqref{eq: localized Ch FT}.

Since $N_0=\frac{1}{2}h^{-\frac{2}{5}}$, applying the linear estimates (Lemma~\ref{lem: Xsb properties}) and the low-frequency flow comparison (Lemma~\ref{lem: Low-frequency comparison of the FPUT and Airy flows}), we obtain 
$$
\begin{aligned}
\big\|P_{\leq N_0}(\mathfrak{C}_hv_h^\pm-w^\pm)\big\|_{X_\pm^{0,b}}
&\ls 
h^{\frac{2s}{5}}\|\mathfrak{C}_hP_{\leq N_0}^hu_{h,0}^\pm\|_{H_x^s}
+
\big\|P_{\leq N_0}(\mathfrak{C}_hu_{h,0}^\pm-w_0^\pm)\big\|_{L_x^2}\\
&\quad
+h^{\frac{2s}{5}} T^\delta\big\|P_{\leq N_0}m_h(-i\partial_x)\mathfrak{C}_h\big\{(v_h^\pm)^2\big\}\big\|_{X_\pm^{s,b-1+\delta}}\\
&\quad+
T^{\frac{3}{4}-b}\Big\|P_{\leq N_0}\Big[m_h(-i\partial_x)\mathfrak{C}_h\big\{(v_h^\pm)^2\big\}-\pa_x\big\{(w^\pm)^2\big\}\Big]\Big\|_{X_\pm^{0,-\frac{1}{4}}}\\
&=:
\textup{(I)} + \textup{(II)} + \textup{(III)} + \textup{(IV)}.
\end{aligned}
$$
Here the flow estimates are applied at time scale \(1\), with
\(\eta_T\) included in the forcing terms. We use
\(\eta_1\eta_T=\eta_T\), which follows from \(T\leq\frac12\), and then
apply the time-localization estimate in Lemma~\ref{lem: Xsb properties}
from temporal exponent \(b-1+\delta\) (resp., \(-\frac14\)) to \(b-1\),
gaining \(T^\delta\) (resp., \(T^{\frac34-b}\)).
By \eqref{ineq: Ch Hs equivalence}, we obtain
$$
\textup{(I)} + \textup{(II)}
\ls
h^{\frac{2s}{5}}\|P_{\leq N_0}^hu_{h,0}^\pm\|_{H_z^s}+\|\mathfrak{C}_hu_{h,0}^\pm-w_0^\pm\|_{L_x^2}
\ls
h^{\frac{2s}{5}} R_s +\|\mathfrak{C}_hu_{h,0}^\pm-w_0^\pm\|_{L_x^2}.
$$
For $\textup{(III)}$, \eqref{eq: localized Ch FT}, the low-frequency Bourgain norm equivalence \eqref{ineq: Xsb norm equivalence of Ch in low frequency} and the bilinear estimate \eqref{ineq: FPUT bilinear estimate 1} imply
$$
\begin{aligned}
\textup{(III)}
&\ls
h^{\frac{2s}{5}} T^\delta
\big\|\mathfrak{C}_hP_{\leq N_0}^h\nabla_h\big\{(v_h^\pm)^2\big\}\big\|_{X_\pm^{s,b-1+\delta}}\\
&\ls
h^{\frac{2s}{5}} T^\delta
\big\|P_{\leq N_0}^h\nabla_h\big\{(v_h^\pm)^2\big\}\big\|_{X_{h,\pm}^{s,b-1+\delta}}
\ls
h^{\frac{2s}{5}} T^\delta\|v_h^\pm\|_{X_{h,\pm}^{s,b}}\|v_h^\pm\|_{X_{h,\pm}^{0,b}}
\ls
h^{\frac{2s}{5}} T^\delta R_0R_s.
\end{aligned}
$$
Next, using \eqref{lem: Low-frequency product discrepancy for Ch}, the triangle inequality and $|m_h(\xi)|\leq |\xi|$, we decompose
$$
\begin{aligned}
\textup{(IV)}
&\leq
T^{\frac{3}{4}-b}\big\|P_{\leq N_0}(m_h(-i\partial_x)-\pa_x)(w^\pm)^2\big\|_{X_\pm^{0,-\frac{1}{4}}}
+
T^{\frac{3}{4}-b}\big\|P_{\leq N_0}\pa_x
\big\{(\mathfrak{C}_hv_h^\pm)^2-(w^\pm)^2\big\}\big\|_{X_\pm^{0,-\frac{1}{4}}}\\
&\quad+
T^{\frac{3}{4}-b}
\big\|P_{\leq N_0}\pa_x
\big\{\mathfrak{C}_h\big((v_h^\pm)^2\big)-(\mathfrak{C}_hv_h^\pm)^2\big\}\big\|_{X_\pm^{0,-\frac{1}{4}}}
=: \textup{(IV-a)} + \textup{(IV-b)}.
\end{aligned}
$$
For $\textup{(IV-a)}$, the derivative comparison \eqref{ineq: derivative symbol difference} and the KdV bilinear estimate \eqref{ineq:KdV bilinear estimate} give
$$
\textup{(IV-a)}
\ls
T^{\frac{3}{4}-b}h^{\frac{2s}{5}}\|\pa_x(w^\pm)^2\|_{X_\pm^{s,-\frac{1}{4}}}
\ls
T^{\frac{3}{4}-b}h^{\frac{2s}{5}}\|w^\pm\|_{X_\pm^{s,b}}\|w^\pm\|_{X_\pm^{0,b}}
\ls
T^{\frac{3}{4}-b}h^{\frac{2s}{5}}R_0R_s.
$$
For $\textup{(IV-b)}$, the KdV bilinear estimate \eqref{ineq:KdV bilinear estimate} and $\|\mathfrak{C}_hv_h^\pm\|_{X_\pm^{0,b}}\sim\|v_h^\pm\|_{X_{h,\pm}^{0,b}}$ from \eqref{ineq: Xsb norm equivalence of Ch in low frequency} yield
$$
\textup{(IV-b)}
\ls
T^{\frac{3}{4}-b}
\Big(\|\mathfrak{C}_hv_h^\pm\|_{X_\pm^{0,b}}+\|w^\pm\|_{X_\pm^{0,b}}\Big)\|\mathfrak{C}_hv_h^\pm-w^\pm\|_{X_\pm^{0,b}}
\ls
T^{\frac{3}{4}-b}R_0\|\mathfrak{C}_hv_h^\pm-w^\pm\|_{X_\pm^{0,b}}.
$$
Combining these estimates yields
$$
\textup{(IV)}
\ls
T^{\frac{3}{4}-b}h^{\frac{2s}{5}}R_0R_s
+
T^{\frac{3}{4}-b}R_0\|\mathfrak{C}_hv_h^\pm-w^\pm\|_{X_\pm^{0,b}}.
$$
Substituting the bounds for $\textup{(I)}$--$\textup{(IV)}$ together with \eqref{ineq: w high freq estimate} into \eqref{ineq: Cv-w decomposition}, we have
\begin{equation}\label{ineq: Cv-w estimate}
\|\mathfrak{C}_hv_h^\pm-w^\pm\|_{X_{\pm}^{0,b}}
\leq
C_2\|\mathfrak{C}_hu_{h,0}^\pm-w_0^\pm\|_{L_x^2}
+C_2(1+3T^\delta R_0)h^{\frac{2s}{5}}R_s
+C_0T^\delta R_0\|\mathfrak{C}_hv_h^\pm-w^\pm\|_{X_{\pm}^{0,b}},
\end{equation}
where $C_0,C_2>0$ are independent of $h$, $T$, $R_0$, and $R_s$,
and $C_0$ is also independent of $s$. Here we used $\frac{3}{4}-b>\delta$ and $T\leq1$.
Taking \(T=T(R_0)>0\) sufficiently small so that
\(C_0T^\delta R_0\leq\frac12\), we absorb the last term in
\eqref{ineq: Cv-w estimate} and obtain
\[
\|\mathfrak{C}_hv_h^\pm-w^\pm\|_{X_\pm^{0,b}}
\leq C_s\|\mathfrak{C}_hu_{h,0}^\pm-w_0^\pm\|_{L_x^2}
+C_sh^{\frac{2s}{5}}R_s.
\]
Finally, the embedding \(X_\pm^{0,b}\hookrightarrow C_tL_x^2(\R)\),
valid for \(b>\frac12\), proves \eqref{ineq: second reduction}.
\end{proof}

We can now combine the two reductions. After decreasing \(h_0\) and
\(T=T(R_0)\) so that both lemmas apply, we decompose the approximation
error as
\begin{equation}\label{eq: auxiliary decomposition}
\mathfrak{C}_hu_h^\pm-w^\pm
=
\mathfrak{C}_h(u_h^\pm-v_h^\pm)
+
(\mathfrak C_hv_h^\pm-w^\pm).
\end{equation}
The first term is controlled by Lemma~\ref{lem: first reduction} and the
\(L^2\)-isometry \eqref{eq: L2 isometry}, while the second is controlled
by Lemma~\ref{lem: second reduction}. The triangle inequality therefore
gives the following local approximation result.

\begin{proposition}[Local-in-time KdV limit for the FPUT system with persistence of regularity]\label{prop: Local KdV limit for FPUT with persistence of regularity}
Let $0\leq s\leq 1$ and $\mathfrak{C}_h$ be the continuation operator defined in \eqref{eq: definition of Ch}. Suppose
$$
\sup_{h\in(0,1]}
\|u_{h,0}^\pm\|_{L_z^2(h\Z)}, 
\|w_0^\pm\|_{L_x^2(\R)}
\leq R_0
\quad\textup{and}\quad
\sup_{h\in(0,1]}
\|u_{h,0}^\pm\|_{H_z^s(h\Z)}, 
\|w_0^\pm\|_{H_x^s(\R)}
\leq R_s.
$$
Let $w^\pm$ be the solution to the KdV equation
\eqref{eq:KdV integral form} with initial data $w_0^\pm$.
For each $h\in(0,1]$, let $u_h^\pm$ be the solution to the coupled FPUT
system \eqref{eq: coupled FPUT integral form} with initial data
$u_{h,0}^\pm$. Then, there exist $h_0\in(0,1]$, a time $T=T(R_0)>0$
and a constant $C_s=C_s(R_0)>1$, independent of $h\in(0,h_0]$ and
$R_s$, such that, for every $h\in(0,h_0]$, both solutions exist on
$[-T,T]$ and
\begin{equation}\label{ineq: FPUT to KdV local difference}
\|\mathfrak{C}_hu_h^\pm-w^\pm\|_{C_t([-T,T];L_x^2(\R))}
\leq
C_s\|\mathfrak{C}_hu_{h,0}^\pm-w_0^\pm\|_{L_x^2(\R)}
+
C_sR_sh^{\frac{2s}{5}}.
\end{equation}
\end{proposition}

\begin{remark}
The lifespan depends only on the lower-order $L^2$-bounds, whereas the
higher $H^s$-bounds enter through the approximation error. Together with
the uniform-in-time $L^2$-bounds and the propagated $H^s$-bounds, this
allows the local approximation estimate to be iterated over successive
time intervals of a common length.
\end{remark}

\subsection{Global extension}
We now prove the main theorem by iterating the local approximation in Proposition~\ref{prop: Local KdV limit for FPUT with persistence of regularity}. Lemma~\ref{lem: uniform L2 bound for the FPUT} and the $L^2$-conservation law for KdV provide uniform-in-time lower-order bounds, and hence a common time step depending only on the initial $L^2$-bounds. Proposition~\ref{prop: exponential bound} controls the $H^s$-norms at the iteration times, so that the accumulated local errors grow at most exponentially in time.

\begin{proof}[Proof of Theorem~\ref{thm: main theorem}]
It suffices to consider $t\geq0$, since the case $t<0$ follows by the same argument. Choose $R_0, R_s\geq 1$ such that
$$
\sup_{h\in(0,1]}\|u_{h,0}^\pm\|_{L_z^2(h\Z)}, 
\|w_0^\pm\|_{L_x^2(\R)}
\leq R_0
\quad\textup{and}\quad
\sup_{h\in(0,1]}\|u_{h,0}^\pm\|_{H_z^s(h\Z)}, 
\|w_0^\pm\|_{H_x^s(\R)}
\leq R_s.
$$
By Lemma~\ref{lem: uniform L2 bound for the FPUT} and the $L^2$-conservation law for the KdV equation, the $L^2$-norms of both solutions remain uniformly bounded by a constant depending only on $R_0$. 
The proofs of Lemmas~\ref{lem: first reduction} and \ref{lem: second reduction} apply uniformly at any initial time: the additional fixed almost translations preserve Bourgain norms, and the additional constant scalar phases in \eqref{eq: general translation correction} have modulus one.
Consequently, Proposition~\ref{prop: Local KdV limit for FPUT with persistence of regularity} yields $h_0\in(0,1]$ and a time step 
$$
\Delta t=\Delta t(R_0)>0,
$$
independent of $h$, $R_s$, and the initial time, such that the local KdV approximation holds on every interval of length $\Delta t$.
Set 
$$
t_n:=n\Delta t, \quad n\in\Z_{\geq 0}.
$$ 
Applying Proposition~\ref{prop: Local KdV limit for FPUT with persistence of regularity} on $[t_n,t_{n+1}]$, and using Proposition~\ref{prop: exponential bound} to control the $H^s$-norms at $t=t_n$, for every $n\in \Z_{\geq 0}$ and $h\in(0,h_0]$, we obtain, after enlarging the constant if necessary, a constant $C=C(R_0,s)\geq 2$ such that
$$
\|\mathfrak{C}_hu_h^\pm-w^\pm\|_{C_t([t_n, t_{n+1}];L_x^2(\R))}
\leq
C\|\mathfrak{C}_hu_h^\pm(t_n)-w^\pm(t_n)\|_{L_x^2(\R)}+Ce^{Ct_n}R_sh^{\frac{2s}{5}}.
$$
An induction on $m$ then gives, for every $m\geq 0$,
$$
\|\mathfrak{C}_hu_h^\pm-w^\pm\|_{C_t([0, t_m];L_x^2(\R))}
\leq
C^m\|\mathfrak{C}_hu_{h,0}^\pm-w_0^\pm\|_{L_x^2(\R)}
+
R_sh^{\frac{2s}{5}}\sum_{k=0}^{m-1}C^{m-k}e^{Ct_k}.
$$
Since $t_k=k\Delta t$, 
$$
\sum_{k=0}^{m-1}C^{m-k}e^{Ct_k}
=
\sum_{k=0}^{m-1}C^{m-k}(e^{C\Delta t})^k
\leq
(C+e^{C\Delta t})^m. 
$$
Set
$$
K=K(R_0, s):=\frac{\log(C+e^{C\Delta t})}{\Delta t}.
$$
It follows that
$$
\|\mathfrak{C}_hu_h^\pm-w^\pm\|_{C_t([0,t_m];L_x^2(\R))}
\leq
e^{Kt_m}
\Big(
\|\mathfrak{C}_hu_{h,0}^\pm-w_0^\pm\|_{L_x^2(\R)}
+R_sh^{\frac{2s}{5}}\Big).
$$
Given $t\geq0$, choose $m\in\Z_{\geq 0}$ such that $t\leq t_m<t+\Delta t$. Since $e^{Kt_m}\leq e^{Kt}e^{K\Delta t}$, we obtain
$$
\|\mathfrak{C}_hu_h^\pm(t)-w^\pm(t)\|_{L_x^2(\R)}
\leq
 e^{Kt}e^{K\Delta t}\Big(\|\mathfrak{C}_hu_{h,0}^\pm-w_0^\pm\|_{L_x^2(\R)}
+R_sh^{\frac{2s}{5}}\Big).
$$
Define
$$
K_0:=\max\{K, e^{K\Delta t}\}
\quad\textup{and}\quad
K_s:=K_0R_s.
$$
Then $K_0$ depends only on $R_0$ and $s$, whereas $K_s$ may additionally depend on $R_s$. Thus, \eqref{ineq: convergence bound} follows.
\end{proof}

\subsection{Continuum and small-amplitude limits}
We now prove the continuum and small-amplitude limits in Theorem~\ref{thm: continuum limit}. The approximation for the strain $r_h$ does not follow directly from the moving-frame approximation for $u_h^\pm$, since $\mathfrak{C}_h$ does not generally intertwine lattice almost translations with continuum translations; see \cite[Section~8.3]{HKY2021} for the analogous issue. We address this using low-frequency compatibility and the high-frequency Sobolev tail estimate. The small-amplitude limit then follows by rescaling \eqref{eq:FPUT-rescaling}.
\begin{proof}[Proof of Theorem~\ref{thm: continuum limit}]
By \eqref{eq: def of transformed variables}, the hypothesis of the theorem, and \eqref{ineq: Ch Hs equivalence},
$$
\sup_{h\in(0,1]}
\sum_\pm
\Big(\|u_{h,0}^\pm\|_{H_z^s}+\|w^{\pm, h}(0)\|_{H_x^s}\Big)\ls
\sup_{h\in(0,1]}\Big(\|r_{h,0}\|_{H_z^s}+\big\|h^2\nabla_h^{-1}r_{h,1}\big\|_{H_z^s}\Big)<\infty.
$$
Since $w^{\pm, h}(0)=\mathfrak{C}_hu_{h,0}^\pm$,  Theorem~\ref{thm: main theorem} yields
$$
\|\mathfrak{C}_hu_h^\pm(t)-w^{\pm,h}(t)\|_{L_x^2(\R)}
\leq
K_se^{K_0|t|}h^{\frac{2s}{5}},
\qquad t\in\R.
$$
Reversing the diagonalization in Section~\ref{subsec: Derivation of the coupled FPUT system}, we reconstruct
$$
r_h^\pm(t)=e^{\mp\frac{t}{h^2}\pa_h}u_h^\pm(t),
\quad
r_h=r_h^++r_h^-,
$$
(see \eqref{decomposition of rh}). Then $r_h$ is a global solution of the rescaled FPUT system \eqref{eq:FPUT-system-expanded} with initial data $(r_{h,0}, r_{h,1})$. Its uniqueness follows from the uniqueness of $u_h^\pm$.

Using low-frequency compatibility \eqref{eq: commutativity of multipliers in low frequency}, the $L^2$-isometry of $\mathfrak{C}_h$ \eqref{eq: L2 isometry} and the unitarity of the lattice and continuum translation operators, we obtain 
$$
\begin{aligned}
\big\|\mathfrak{C}_hr_h^\pm-e^{\mp\frac{t}{h^2}\pa_x}w^{\pm,h}\big\|_{L_x^2}
&\leq
\big\|\mathfrak{C}_he^{\mp\frac{t}{h^2}\pa_h}(1-P_{\leq N_0}^h)u_h^\pm\big\|_{L_x^2}
+\big\|e^{\mp\frac{t}{h^2}\pa_x}(\mathfrak{C}_hP_{\leq N_0}^hu_h^\pm-w^{\pm,h})\big\|_{L_x^2}\\
&\leq
\big\|(1-P_{\leq N_0}^h)u_h^\pm\big\|_{L_z^2}
+
\big\|\mathfrak{C}_hP_{\leq N_0}^hu_h^\pm-w^{\pm,h}\big\|_{L_x^2}\\
&\leq 2\big\|(1-P_{\leq N_0}^h)u_h^\pm\big\|_{L_z^2}
+\|\mathfrak{C}_hu_h^\pm-w^{\pm,h}\|_{L_x^2}\\
&\leq
2N_0^{-s}\|u_h^\pm(t)\|_{H_z^s}+\|\mathfrak{C}_hu_h^\pm(t)-w^{\pm,h}(t)\|_{L_x^2}
\leq
K_se^{K_0|t|}h^{\frac{2s}{5}},
\end{aligned}
$$
after enlarging $K_0$ and $K_s$ if necessary. Here the last inequality follows from  \eqref{ineq: FPUT KdV Hs norm growth}. Since $r_h=r_h^++r_h^-$ and $e^{\mp\frac{t}{h^2}\pa_x}w^{\pm,h}(t,x)=w^{\pm,h}(t,x\mp\frac{t}{h^2})$, summing over the signs and taking the supremum over $|t|\leq T$ proves the continuum-limit estimate \eqref{ineq: continuum limit}.

For the small-amplitude limit \eqref{ineq: small-amplitude limit}, define
$$
r(t,z):=h^2r_h(h^3t, hz), \quad z\in\Z.
$$
Using the rescaling \eqref{eq:FPUT-rescaling}, $r$ solves the original FPUT system \eqref{eq:FPUT-system-original-scale} and 
$$
(\mathfrak{C}_1r)(t, x)
=
h^2\sum_{z\in\Z}\varphi(x-z)r_h(h^3t, hz)
=h^2(\mathfrak{C}_hr_h)(h^3t, hx).
$$
By changing variables $y=hx$ and applying the continuum limit estimate \eqref{ineq: continuum limit}, we obtain
$$
\begin{aligned}
&\big\|(\mathfrak{C}_1r)(t,x)-h^2w^{+,h}(h^3t, h(x-t))-h^2w^{-,h}(h^3t, h(x+t))\big\|_{L_x^2(\R)}\\
&\quad
=
h^{\frac{3}{2}}\big\|(\mathfrak{C}_hr_h)(h^3t, y)-w^{+,h}(h^3t,  y-ht)-w^{-,h}(h^3t, y+ht)\big\|_{L_y^2(\R)}\leq K_se^{K_0T}h^{\frac{3}{2}+\frac{2s}{5}},
\end{aligned}
$$
for $|t|\leq T/h^3$. This proves the small-amplitude estimate \eqref{ineq: small-amplitude limit} and completes the proof.
\end{proof}

\appendix

\section{Proof of Lemma \ref{lem: unified integral bounds}}\label{appendix: Proof of integral estimates}
\noindent\textbf{\underline{Reduction and symmetries.}}
The integral
\(\mathcal I_{\tau,\xi,\ell}^{h;\sigma_1,\sigma_2}\) vanishes whenever
\(|\xi|\geq\frac{2\pi}{h}\). Moreover, up to endpoints of measure
zero, the change of variables
\(\xi_1\mapsto\xi-\xi_1\) gives
\begin{equation}\label{eq: symmetry in sigma1 sigma2}
\mathscr I_{\tau,\xi,\ell}^{h;\sigma_1,\sigma_2}
=
\mathscr I_{\tau,\xi,\ell}^{h;\sigma_2,\sigma_1}.
\end{equation}
Thus, among the four possible sign pairs, it suffices to consider
\[
(\sigma_1,\sigma_2)=(-1,-1),  (1,1), (-1,1).
\]
Away from the endpoints \(\xi=\pm\frac{\pi}{h}\), the symmetry \((\tau,\xi,\xi_1,\ell)
\mapsto
(-\tau,-\xi,-\xi_1,-\ell)\) reduces the proof to the case \(\xi\geq0\). Since the case
\(\xi=0\) is trivial, we may assume
\[
0<\xi<\frac{2\pi}{h}.
\]
For such \(\xi\), the conditions \(-\frac{\pi}{h}\leq\xi_1<\frac{\pi}{h}\) and \(-\frac{\pi}{h}\leq\xi-\xi_1<\frac{\pi}{h}\) are equivalent, up to endpoints of measure zero, to
\[
\xi-\frac{\pi}{h}<\xi_1<\frac{\pi}{h}.
\]
Furthermore, either \([\xi]=\xi\) or
\([\xi]=\xi-\frac{2\pi}{h}\), and hence
\begin{equation}\label{eq: common factor}
\sin^2\left(\frac{h[\xi]}{2}\right)
=
\sin^2\left(\frac{h\xi}{2}\right)
=
4\sin^2\left(\frac{h\xi}{4}\right)
\cos^2\left(\frac{h\xi}{4}\right).
\end{equation}
Using the definition of \(s_h\), we can write the modulation as
\begin{equation}\label{eq: common modulation normal form}
\begin{aligned}
&\tau
-\frac{(1+\sigma_1)\xi_1
+(1+\sigma_2)(\xi-\xi_1)}{h^2}
-\frac{2\pi\ell}{h^3}
+\sigma_1s_h(\xi_1)
+\sigma_2s_h(\xi-\xi_1)\\
&=
\tau-\frac{\xi}{h^2}-\frac{2\pi\ell}{h^3}
-\frac{2}{h^3}
\left\{
\sigma_1\sin\left(\frac{h\xi_1}{2}\right)
+\sigma_2\sin\left(\frac{h(\xi-\xi_1)}{2}\right)
\right\}.
\end{aligned}
\end{equation}

\medskip
\noindent\textbf{\underline{Equal-sign interactions.}}
Suppose that
\(\sigma_1=\sigma_2=\sigma\in\{-1,1\}\). By
\eqref{eq: common modulation normal form} and the sum-to-product identity, the modulation becomes
\[
\tau-\frac{\xi}{h^2}-\frac{2\pi\ell}{h^3}
-\frac{4\sigma}{h^3}
\sin\left(\frac{h\xi}{4}\right)
\cos\left(\frac{h(2\xi_1-\xi)}{4}\right).
\]
This expression is invariant under the reflection \(\xi_1\mapsto\xi-\xi_1\), whose fixed point is \(\frac{\xi}{2}\). Therefore, the integral \(\mathcal{I}_{\tau,\xi,\ell}^{h;\sigma,\sigma}\) can be written as 
\[
\begin{aligned}
\mathcal I_{\tau,\xi,\ell}^{h;\sigma,\sigma}
=
2\int_{\frac{\xi}{2}}^{\frac{\pi}{h}}
\frac{d\xi_1}
{\big\langle
\tau-\frac{\xi}{h^2}-\frac{2\pi\ell}{h^3}
-\frac{4\sigma}{h^3}
\sin\big(\frac{h\xi}{4}\big)
\cos\big(\frac{h(2\xi_1-\xi)}{4}\big)
\big\rangle^{2b}}.
\end{aligned}
\]
Note that on the interval of integration, 
\[
0
\leq
\frac{h(2\xi_1-\xi)}{4}
<
\frac{\pi}{2}-\frac{h\xi}{4}
<
\frac{\pi}{2}.
\]
We define
\[
y=
y(\xi_1):=
\frac{4}{h^3}\sin\left(\frac{h\xi}{4}\right)
\left\{
1-\cos\left(\frac{h(2\xi_1-\xi)}{4}\right)
\right\}.
\]
Then, \(y(\xi_1)\) is nonnegative and non-decreasing, and
\[
\begin{aligned}
\left(\frac{dy}{d\xi_1}\right)^2
&=
\frac{1}{h}\sin\left(\frac{h\xi}{4}\right)
\left\{
1+\cos\left(\frac{h(2\xi_1-\xi)}{4}\right)
\right\}y\geq
\frac{1}{h}\sin\left(\frac{h\xi}{4}\right)y.
\end{aligned}
\]
Hence, making the change of variables $y=y(\xi_1)$ and applying the integral bound 
\[\int_{\mathbb{R}}
\frac{dx}
{\la x\ra^{2b}\sqrt{|x-y|}}
\lesssim
\frac{1}{\la y\ra^{\frac12}}\]
with \(b>\frac12\), we obtain
\[
\begin{aligned}
\mathcal I_{\tau,\xi,\ell}^{h;\sigma,\sigma}
&\lesssim
\frac{\sqrt h}
{\sqrt{\sin\big(\frac{h\xi}{4}\big)}}
\int_0^\infty
\frac{dy}
{\sqrt y
 \big\langle
 \tau-\frac{\xi}{h^2}-\frac{2\pi\ell}{h^3}
 -\frac{4\sigma}{h^3}\sin\big(\frac{h\xi}{4}\big)
 +\sigma y
 \big\rangle^{2b}}\\
 &\lesssim
\frac{\sqrt h}
{\sqrt{\sin\big(\frac{h\xi}{4}\big)}
 \big\langle
 \tau-\frac{\xi}{h^2}-\frac{2\pi\ell}{h^3}
 -\frac{4\sigma}{h^3}\sin\big(\frac{h\xi}{4}\big)
 \big\rangle^{\frac12}}.
\end{aligned}
\]
Consequently, the elementary inequality \(\frac{1}
{\la A\ra^{\frac12}\la B\ra^{\frac12}}
\lesssim
\frac{1}{\la A-B\ra^{\frac12}}
\leq
\frac{1}{|A-B|^{\frac12}}\), together with the double-angle identity, gives
\begin{equation}\label{ineq: common same-sign integral estimate}
\mathscr I_{\tau,\xi,\ell}^{h;\sigma,\sigma}
\lesssim
\frac{\sin^{\frac32}\big(\frac{h\xi}{4}\big)
\cos^2\big(\frac{h\xi}{4}\big)}
{h^{\frac32}
\big|
s_h([\xi])-\frac{\xi}{h^2}-\frac{2\pi\ell}{h^3}
-\frac{4\sigma}{h^3}\sin\big(\frac{h\xi}{4}\big)
\big|^{\frac12}}.
\end{equation}
Suppose first that
\(0<\xi<\frac{\pi}{h}\) and \(\ell=0\). Since \([\xi]=\xi\), we have
\[
s_h([\xi])
=
\frac{\xi}{h^2}
-\frac{4}{h^3}
\sin\left(\frac{h\xi}{4}\right)
\cos\left(\frac{h\xi}{4}\right).
\]
It follows from
\eqref{ineq: common same-sign integral estimate} that
\[
\mathscr I_{\tau,\xi,0}^{h;-1,-1}
\lesssim
\frac{
\sin\big(\frac{h\xi}{4}\big)
\cos^2\big(\frac{h\xi}{4}\big)
}
{\sqrt{1-\cos\big(\frac{h\xi}{4}\big)}}\lesssim
\cos^2\bigg(\frac{h\xi}{4}\bigg)
\sqrt{1+\cos(\frac{h\xi}{4})}
\lesssim1
\]
and
\[
\mathscr I_{\tau,\xi,0}^{h;1,1}
\lesssim
\frac{
\sin\big(\frac{h\xi}{4}\big)
\cos^2\big(\frac{h\xi}{4}\big)
}
{\sqrt{1+\cos\big(\frac{h\xi}{4}\big)}}
\lesssim
\sin(\frac{h\xi}{4}).
\]
Suppose next that
\(\frac{\pi}{h}\leq\xi<\frac{2\pi}{h}\). Since
\([\xi]=\xi-\frac{2\pi}{h}\), we have
\[
s_h([\xi])
=
\frac{\xi}{h^2}-\frac{2\pi}{h^3}
+\frac{4}{h^3}
\sin\left(\frac{h\xi}{4}\right)
\cos\left(\frac{h\xi}{4}\right),
\]
and consequently,
\[
\begin{aligned}
s_h([\xi])-\frac{\xi}{h^2}
-\frac{2\pi\ell}{h^3}
-\frac{4\sigma}{h^3}\sin\left(\frac{h\xi}{4}\right)
=
-\frac{1}{h^3}
\left[
2\pi(1+\ell)
+4\sin\left(\frac{h\xi}{4}\right)
\left\{
\sigma-\cos\left(\frac{h\xi}{4}\right)
\right\}
\right].
\end{aligned}
\]
For
\(\frac{\pi}{h}\leq\xi<\frac{2\pi}{h}\),
\(\ell\in\{-1,0,1\}\), and
\(\sigma\in\{-1,1\}\), we have the uniform lower bound
\[\left|
2\pi(1+\ell)
+4\sin\left(\frac{h\xi}{4}\right)
\left\{
\sigma-\cos\left(\frac{h\xi}{4}\right)
\right\}
\right|
\gtrsim1,\]
because \(1
\leq
\sin z\{
1+\cos z\}
\leq
\frac{3\sqrt3}{4}\) while \(\frac{\sqrt2-1}{2}
\leq
\sin z
\{
1-\cos z\}
\leq1\) for $\frac{\pi}{4}\leq z\leq\frac{\pi}{2}$. Thus, \eqref{ineq: common same-sign integral estimate} gives
\[
\mathscr I_{\tau,\xi,\ell}^{h;\sigma,\sigma}
\lesssim1.
\]

\medskip
\noindent\textbf{\underline{Opposite-sign interactions.}}
By \eqref{eq: symmetry in sigma1 sigma2}, it suffices to consider \((\sigma_1,\sigma_2)=(-1,1)\). In this case, \eqref{eq: common modulation normal form} and the sum-to-product identity give
\[
\begin{aligned}
&\tau-\frac{2(\xi-\xi_1)}{h^2}
-\frac{2\pi\ell}{h^3}
-s_h(\xi_1)+s_h(\xi-\xi_1)\\
&=
\tau-\frac{\xi}{h^2}-\frac{2\pi\ell}{h^3}
+\frac{4}{h^3}
\cos\left(\frac{h\xi}{4}\right)
\sin\left(\frac{h(2\xi_1-\xi)}{4}\right).
\end{aligned}
\]
Now, we make the change of variables
\[
y
:=
\frac{4}{h^3}
\cos\left(\frac{h\xi}{4}\right)
\sin\left(\frac{h(2\xi_1-\xi)}{4}\right).
\]
Since
\[
-\frac{\pi}{2}+\frac{h\xi}{4}
<
\frac{h(2\xi_1-\xi)}{4}
<
\frac{\pi}{2}-\frac{h\xi}{4},
\]
we have
\[
\cos\left(\frac{h(2\xi_1-\xi)}{4}\right)
\geq
\sin\left(\frac{h\xi}{4}\right)>0,
\]
and hence, 
\[
\frac{dy}{d\xi_1}
=
\frac{2}{h^2}
\cos\left(\frac{h\xi}{4}\right)
\cos\left(\frac{h(2\xi_1-\xi)}{4}\right)\geq
\frac{2}{h^2}
\sin\left(\frac{h\xi}{4}\right)
\cos\left(\frac{h\xi}{4}\right).
\]
Thus, it follows that 
\[
\mathcal I_{\tau,\xi,\ell}^{h;-1,1}
\lesssim 
\frac{h^2}
{\sin\big(\frac{h\xi}{4}\big)\cos\big(\frac{h\xi}{4}\big)}
\int_{\mathbb R}
\frac{dy}
{\big\langle
\tau-\frac{\xi}{h^2}
-\frac{2\pi\ell}{h^3}+y
\big\rangle^{2b}}\lesssim\frac{h^2}
{\sin\big(\frac{h\xi}{4}\big)\cos\big(\frac{h\xi}{4}\big)},
\]
because \(b>\frac12\). Together with \eqref{eq: common factor}, this bound yields 
\[\mathscr I_{\tau,\xi,\ell}^{h;-1,1}
\lesssim
\sin\left(\frac{h\xi}{4}\right)
\cos\left(\frac{h\xi}{4}\right)\lesssim1.\]
Indeed, if
\(0<\xi<\frac{\pi}{h}\) and \(\ell=0\), we have a stronger bound
\[
\mathscr I_{\tau,\xi,0}^{h;-1,1}
\lesssim
\sin\left(\frac{h\xi}{4}\right).
\]
By the symmetries used in the reduction, the preceding estimates also hold in the corresponding negative-frequency regions. Finally, writing the integrals over a fixed compact interval with the corresponding indicator functions, we obtain the estimates at the endpoint \(\xi=-\frac{\pi}{h}\) by letting \(\xi\to(-\frac{\pi}{h})^+\) and applying the dominated convergence theorem. Combining the equal-sign and opposite-sign estimates proves \eqref{ineq: unified integral bound-1}, \eqref{ineq: unified integral bounds 2 and 3}, and \eqref{ineq: unified integral bound-1'}.


\bibliographystyle{abbrv}
\bibliography{Reference}

\end{document}